%% file: qschur-arxiv.tex
\documentclass[reqno,11pt]{amsart}

\usepackage{
    amsmath,
    amsfonts,
    amssymb,
    amsthm,
    amscd,
    comment,
    enumitem,
    etoolbox,
    textcomp,
    gensymb,
    mathtools,
    mathdots,
}
\usepackage[usenames,dvipsnames]{xcolor}
\usepackage[all]{xy}
\usepackage{tikz}

\usetikzlibrary{shapes.geometric}
\usetikzlibrary{decorations.markings,decorations.pathreplacing}
\usetikzlibrary{cd}
\usetikzlibrary{patterns}
\usetikzlibrary{calc}
\usepackage{setspace}
\usepackage{bm}
\usepackage{bbm} 

\usepackage[colorlinks=true, linkcolor=black, citecolor=purple, urlcolor=blue, breaklinks=true]{hyperref}

\tikzset{anchorbase/.style={baseline={([yshift=-0.5ex]current bounding box.center)}}}
\tikzset{wipe/.style={white,line width=4pt}}

\usepackage[nameinlink]{cleveref}

\crefname{corollary}{Corollary}{Corollaries}
\crefname{definition}{Definition}{Definitions}
\crefname{example}{Example}{Examples}
\crefname{remark}{Remark}{Remarks}
\crefname{lemma}{Lemma}{Lemmas}
\crefname{theorem}{Theorem}{Theorems}
\crefname{introtheorem}{Theorem}{Theorems}
\crefname{equation}{}{}
\crefname{enumi}{}{}
\crefname{section}{$\S\!\!$}{$\S\S\!\!$}

\makeatletter

\newtheorem{theorem}{Theorem}[section]
\newtheorem{introtheorem}{Theorem}

\newtheorem{lemma}[theorem]{Lemma}
\newtheorem{corollary}[theorem]{Corollary}

\theoremstyle{definition}
\newtheorem{definition}[theorem]{Definition}
\newtheorem{remark}[theorem]{Remark}
\newtheorem{example}[theorem]{Example}

\numberwithin{equation}{section}
\allowdisplaybreaks

\newcommand\one{\mathbbm{1}}

\def\qTilt{q\hspace{-.1mm}\text{-}\hspace{-.2mm}\mathbf{Tilt}^{\!+}}

\newcommand{\Hom}{\operatorname{Hom}}
\newcommand{\Z}{\mathbb{Z}}

\newcommand{\Q}{\mathbb{Q}}
\newcommand{\N}{\mathbb{N}}

\newcommand{\Ext}{\operatorname{Ext}}

\newcommand{\End}{\operatorname{End}}

\def\Comp{\Lambda}
\def\sComp{\Lambda_{s}}
\def\Par{\Lambda^+}
\def\tr{\mathtt{T}}
\def\Y{\mathtt{Y}}
\def\P{\mathtt{P}}
\newcommand{\rev}{\mathrm{rev}}
\newcommand{\op}{\mathrm{op}}

\newcommand{\cC}{\mathbf{C}}

\newcommand{\cI}{\mathbf{I}}
\newcommand{\cJ}{\mathbf{J}}

\renewcommand{\k}{\Bbbk}
\newcommand{\qbinom}{\textstyle\genfrac{[}{]}{0pt}{}}

\def\T{\mathtt{T}}

\def\barT{\overline{\mathtt{T}}}
\def\qG{{q\hspace{-.05mm}\text{-}\hspace{-.1mm}GL}}
\def\height{\operatorname{ht}}

\def\qdet{{\operatorname{det}_{q}}}
\def\A{\mathcal O}

\def\nset#1{\{1,\dots,#1\}}
\def\Std{\operatorname{Std}}
\def\Row{\operatorname{Row}}

\def\Mat#1#2{\mathrm{Mat}(#1,#2)}

\def\Schur{\mathbf{Schur}}
\def\qSchur{q\hspace{-.05mm}\text{-}\hspace{-.1mm}\mathbf{Schur}}
\def\schur{K} % main path algebra
\def\sch{H}  % Intro path algebra
\def\R{c}

\def\diag{\operatorname{diag}}

\def\eps{\varepsilon}

\def\bh{\text{\boldmath$h$}}
\def\bi{\text{\boldmath$i$}}
\def\bj{\text{\boldmath$j$}}
\def\bk{\text{\boldmath$k$}}
\def\I{\mathrm I}
\def\mod{\!\operatorname{-mod}}

\newcommand\spot[1]{\filldraw (#1) circle (.9pt)} 
\newcommand\smallspot[1]{\filldraw (#1) circle (.6pt)} 

\def\zigzag{{\begin{tikzpicture}[baseline=0mm,scale=.45]
\draw[-,thin] (.4,.4) to (0,.4) to (.4,0) to (0,0) to (.08,.08) to
(0,0) to (0.08,-.08) to (0,0);
\draw[-,white] (-.04,0) to (.03,-.07);
\end{tikzpicture}}}

\def\ZigzagDownRev{{\begin{tikzpicture}[baseline=.2,scale=.6]
\draw[-,thin] (.4,.4) to (0,.4) to (.4,0) to (0,0) to (.08,.08) to
(0,0) to (0.08,-.08) to (0,0);
\end{tikzpicture}}}

\def\ZigzagDown{{\begin{tikzpicture}[baseline=.2,scale=.6]
\draw[-,thin] (-.4,.4) to (0,.4) to (-.4,0) to (0,0) to (-.08,.08) to
(0,0) to (-0.08,-.08) to (0,0);
\end{tikzpicture}}}

\def\ZigzagUp{{\begin{tikzpicture}[baseline=.2,scale=.6]
\draw[-,thin] (.32,.32) to (.4,.4) to (.32,.48) to (.4,.4)  to (0,.4) to (.4,0) to (0,0);
\end{tikzpicture}}}

\begin{document}

\title[The $q$-Schur category]{\boldmath The $q$-Schur category and polynomial tilting modules for quantum $GL_n$}
\author[J. Brundan]{Jonathan Brundan}
\address{
  Department of Mathematics \\
  University of Oregon \\
  Eugene, OR, USA
}
\urladdr{\href{https://pages.uoregon.edu/brundan}{https://pages.uoregon.edu/brundan}, \textrm{\textit{ORCiD}:} \href{https://orcid.org/0009-0009-2793-216X}{https://orcid.org/0009-0009-2793-216X}}
\email{brundan@uoregon.edu}

\thanks{Research supported in part by NSF grants DMS-2101703
and DMS-2348840.}

\begin{abstract}
The {\em $q$-Schur category} 
is a $\Z[q,q^{-1}]$-linear monoidal category
closely related to the $q$-Schur algebra.
We explain how to construct it from  coordinate algebras of quantum $GL_n$ for all $n \geq 0$. 
Then we use Donkin's work on Ringel duality for $q$-Schur algebras to make precise the relationship between the $q$-Schur category and a $\Z[q,q^{-1}]$-form for the $U_q\mathfrak{gl}_n$-web category of Cautis, Kamnitzer and Morrison.
We construct explicit 
integral bases for morphism spaces in the latter category, and extend the
Cautis-Kamnitzer-Morrison theorem
to polynomial representations of quantum $GL_n$ at a root of unity over a field of any characteristic.
\end{abstract}

\maketitle

\input{s1-introduction}
\input{s2-combinatorics}
\input{s3-manin}
\input{s4-schuralgebra}
\input{s5-schurcategory}
\input{s6-presentations-arxiv}
\input{s7-basis}

\input{s8-tilting}

\input{s9-appendix}

\bibliographystyle{alphaurl}
\bibliography{qschur}
\end{document}

%% file: s1-introduction.tex
\section{Introduction}\label{s1-intro}

In this article, we revisit some algebra from the 1990s using the diagrammatic technique of string calculus for strict monoidal categories which has become ubiquitous in this area since then. 
The initial goal is to give a self-contained construction of a strict $\Z[q,q^{-1}]$-linear monoidal category, the {\em $q$-Schur category}, together with three 
important bases for its morphism spaces.
The path algebra of this category
is Morita equivalent to the direct sum of the
$q$-Schur algebras $S_q(n,n)$
of Dipper and James \cite{DJgl} for all $n \geq 0$.
In that context,
all three bases were studied in detail already 
30 years ago, and this part of the article is mainly expository. Indeed, there are already many generalizations in the literature---cyclotomic 
\cite{DJM}, affine
\cite{GreenAffine,MS,MakS}, and 2-categorical
\cite{Will,MSV,WebCat}, to name but a few.

Once the general framework is in place, we use the $q$-Schur category to
define a $\Z[q,q^{-1}]$-form
for the positive half of the
$U_q\mathfrak{gl}_n$-web category of Cautis, Kamnitzer and Morrison \cite{CKM}, complete with bases for its morphism spaces as free $\Z[q,q^{-1}]$-modules.
Integral bases in the latter category have previously been constructed in the unpublished paper of Elias \cite{Elias},
and their existence also 
follows theoretically from \cite{AST}, but the relationship to the known bases for the $q$-Schur algebra
is not apparent from that work.
We also explain how the canonical basis fits into this picture, something which is not mentioned at all in \cite{Elias}.

Our starting point is the definition of a
strict $\Z$-linear monoidal category 
called the {\em Schur category}, denoted simply by $\Schur$, from \cite[Def.~4.2]{BEEO}.
The object set of $\Schur$ is the set $\sComp$ of all {\em strict compositions}, that is, sequences $\lambda = (\lambda_1,\dots,\lambda_\ell)$ of positive integers for $\ell \geq 0$,
with tensor product of objects 
defined by concatenation.
For strict compositions $\lambda$ and $\mu$, the morphism space
 $\Hom_{\Schur}(\mu,\lambda)$ 
is zero unless $r := \sum_i \lambda_i = \sum_i \mu_i$, in which case
this morphism space is a free $\Z$-module with a distinguished {\em standard basis}
 parametrized by the set
 $(S_\lambda \backslash S_r / S_\mu)_{\min}$ of minimal length representatives for the double cosets 
 of the parabolic subgroups $S_\lambda$ and $S_\mu$ in the symmetric group $S_r$.
 Vertical composition making $\Schur$ into a $\Z$-linear category
 is defined by {\em Schur's product rule} as in the classical Schur algebra
 (see \cite[2.3b]{Greenbook}), and the horizontal composition making it into a monoidal category
 is induced by the natural embeddings $S_{a} \times S_{b} \hookrightarrow S_{a+b}$. 

As usual with strict monoidal categories, it is convenient to represent morphisms in $\Schur$ by certain string diagrams;
the vertical composition $f \circ g$ of 
morphisms $f$ and $g$
is obtained by stacking the string diagram for $f$ on top of the one for $g$, and their horizontal composition
$f \star g$ is obtained by stacking $f$ to the left of $g$.
We represent the standard basis elements for
 $\Hom_{\Schur}(\mu,\lambda)$ by $\lambda\times\mu$ {\em double coset diagrams}\footnote{Called ``chicken foot diagrams'' in \cite{BEEO}.}, such as the diagram on the left:
 $$
  \begin{tikzpicture}[anchorbase,scale=1.55]
\draw[-,line width=.6mm] (.212,.5) to (.212,.39);
\draw[-,line width=.75mm] (.595,.5) to (.595,.39);
\draw[-,line width=.15mm] (0.0005,-.396) to (.2,.4);
\draw[-,line width=.3mm] (0.01,-.4) to (.59,.4);
\draw[-,line width=.3mm] (.4,-.4) to (.607,.4);
\draw[-,line width=.45mm] (.79,-.4) to (.214,.4);
\draw[-,line width=.15mm] (.8035,-.398) to (.614,.4);
\draw[-,line width=.3mm] (.4006,-.5) to (.4006,-.395);
\draw[-,line width=.6mm] (.788,-.5) to (.788,-.395);
\draw[-,line width=.45mm] (0.011,-.5) to (0.011,-.395);
\node at (0.05,0.05) {$\scriptstyle 1$};
\node at (0.76,0.05) {$\scriptstyle 1$};
\node at (0.35,0.35) {$\scriptstyle 3$};
\node at (0.2,-0.26) {$\scriptstyle 2$};
\node at (0.5,-0.26) {$\scriptstyle 2$};
\end{tikzpicture}
\leftrightarrow
\begin{tikzpicture}[anchorbase,scale=1.2]
\draw[ultra thick] (-.01,-.02) to (0.41,-.02);
\node at (0.2,-0.13) {$\scriptstyle 3$};
\draw[ultra thick] (.59,-.02) to (0.81,-.02);
\node at (0.7,-0.13) {$\scriptstyle 2$};
\draw[ultra thick] (.99,-.02) to (1.61,-.02);
\node at (1.3,-0.13) {$\scriptstyle 4$};
\draw[ultra thick] (-.01,1.02) to (0.61,1.02);
\node at (0.3,1.13) {$\scriptstyle 4$};
\draw[ultra thick] (.79,1.02) to (1.61,1.02);
\node at (1.2,1.13) {$\scriptstyle 5$};
\draw[-] (0,0) to (0,1);
\draw[-] (0.2,0) to (.8,1);
\draw[-] (0.4,0) to (1,1);
\draw[-] (0.6,0) to (1.2,1);
\draw[-] (.8,0) to (1.4,1);
\draw[-] (1,0) to (.2,1);
\draw[-] (1.2,0) to (.4,1);
\draw[-] (1.4,0) to (.6,1);
\draw[-] (1.6,0) to (1.6,1);
\end{tikzpicture}\leftrightarrow
(2\:5\:8\:4\:7\:3\:6)
\in (S_{(4,5)} \backslash S_9 / S_{(3,2,4)})_{\min}
\leftrightarrow
A=\begin{bmatrix} 1&0&3\\2&2&1\end{bmatrix}.
$$
In this double coset diagram, there are strings of various thicknesses indicated by the numerical labels. Thick strings at the bottom split into thinner strings, which are allowed to cross each other 
forming a {\em reduced} diagram for a permutation
in the middle of the picture, before merging back into thick strings at the top.
Subsequently, we will
index $S_\lambda \backslash S_r / S_\mu$-double cosets also 
by the set
$\Mat{\lambda}{\mu}$ consisting of matrices of non-negative integers whose row and column sums are the entries of the compositions $\lambda$ and $\mu$, respectively.
The $ij$-entry $a_{i,j}$ of the matrix $A$ records the thickness of the string that connects the $i$th thick string at the top 
 to the $j$th thick string at the bottom of the corresponding double coset diagram. 

The $q$-analog of the Schur category is a strict 
$\Z[q,q^{-1}]$-linear monoidal category
denoted $\qSchur$
whose specialization at $q=1$ recovers $\Schur$.
In our approach, $\qSchur$ is defined from the outset 
to be the $\Z[q,q^{-1}]$-linear category with
 the same objects as $\Schur$,
 tensor product of objects being by concatenation as before. Its morphism spaces are
defined so that $\Hom_{\qSchur}(\mu,\lambda)$ 
is the
free $\Z[q,q^{-1}]$-module with a {\em standard basis} $\{\xi_A\:|\:A \in \Mat{\lambda}{\mu}\}$, which we represent graphically 
by almost the same double coset diagrams as above, except that we replace each singular crossing 
$\begin{tikzpicture}[anchorbase,scale=.6]
	\draw[-,thick] (0.3,-.3) to (-.3,.4);
	\draw[-,thick] (-0.3,-.3) to (.3,.4);
\end{tikzpicture}$
with a positive
crossing
$\begin{tikzpicture}[anchorbase,scale=.6]
	\draw[-,thick] (0.3,-.3) to (-.3,.4);
	\draw[-,line width=5pt,white] (-0.3,-.3) to (.3,.4);
	\draw[-,thick] (-0.3,-.3) to (.3,.4);
\end{tikzpicture}$.
Then we need rules for computing vertical and horizontal compositions of standard basis vectors.
Horizontal composition is defined by horizontally stacking diagrams just as in $\Schur$.
Vertical composition is defined by the $q$-analog of Schur's product rule; see \cref{schurs,qschurs}. Although there is no simple closed 
formula for this in general, 
it can be computed algorithmically using relations in
Manin's quantized coordinate algebra $\A_q(n)$ of $n \times n$ matrices from \cite{Manin}.

Our first theorem gives a
presentation for $\qSchur$ which incorporates the positive crossings as one of three types of generating morphism.
Setting $q=1$ in this recovers the presentation for $\Schur$ derived in \cite{BEEO}.

\begin{introtheorem}\label{th1alt}
As a strict $\Z[q,q^{-1}]$-linear monoidal category,
$\qSchur$ is generated by
the objects $(r)$ for $r > 0$ 
and morphisms called {\em merges}, {\em splits} and {\em positive crossings}
represented by 
\begin{align*}
\begin{tikzpicture}[anchorbase,scale=.8]
	\draw[-,line width=1pt] (0.28,-.3) to (0.08,0.04);
	\draw[-,line width=1pt] (-0.12,-.3) to (0.08,0.04);
	\draw[-,line width=2pt] (0.08,.4) to (0.08,0);
        \node at (-0.22,-.4) {$\scriptstyle a$};
        \node at (0.35,-.38) {$\scriptstyle b$};
        \node at (0.08,.52) {$\scriptstyle a+b$};
\end{tikzpicture} 
&:(a)\star(b) \rightarrow (a+b),&
\begin{tikzpicture}[anchorbase,scale=.8]
	\draw[-,line width=2pt] (0.08,-.3) to (0.08,0.04);
	\draw[-,line width=1pt] (0.28,.4) to (0.08,0);
	\draw[-,line width=1pt] (-0.12,.4) to (0.08,0);
        \node at (-0.22,.5) {$\scriptstyle a$};
        \node at (0.36,.52) {$\scriptstyle b$};
        \node at (0.08,-.42) {$\scriptstyle a+b$};
\end{tikzpicture}
&:(a+b)\rightarrow (a)\star(b),
&\begin{tikzpicture}[anchorbase,scale=.8]
	\draw[-,line width=1pt] (0.3,-.3) to (-.3,.4);
	\draw[-,line width=4pt,white] (-0.3,-.3) to (.3,.4);
	\draw[-,line width=1pt] (-0.3,-.3) to (.3,.4);
    \node at (-0.36,.52) {$\scriptstyle b$};
        \node at (0.36,.5) {$\scriptstyle a$};
        \node at (-0.36,-.4) {$\scriptstyle a$};
        \node at (0.36,-.4) {$\scriptstyle b$};
\end{tikzpicture}:(a)\star (b)
\rightarrow(b)\star (a)
\end{align*}
for $a,b > 0$,
subject to the {\em associativity} and {\em coassociativity 
relations}
\begin{align}
\begin{tikzpicture}[baseline = -.14mm,scale=1]
	\draw[-,thick] (0.35,-.3) to (0.08,0.14);
	\draw[-,thick] (0.1,-.3) to (-0.04,-0.06);
	\draw[-,line width=1pt] (0.085,.14) to (-0.035,-0.06);
	\draw[-,thick] (-0.2,-.3) to (0.07,0.14);
	\draw[-,line width=2pt] (0.08,.45) to (0.08,.1);
        \node at (0.45,-.41) {$\scriptstyle c$};
        \node at (0.07,-.4) {$\scriptstyle b$};
        \node at (-0.28,-.41) {$\scriptstyle a$};
\end{tikzpicture}
&=
\begin{tikzpicture}[baseline = -0.14mm,scale=1]
	\draw[-,thick] (0.36,-.3) to (0.09,0.14);
	\draw[-,thick] (0.06,-.3) to (0.2,-.05);
	\draw[-,line width=1pt] (0.07,.14) to (0.19,-.06);
	\draw[-,thick] (-0.19,-.3) to (0.08,0.14);
	\draw[-,line width=2pt] (0.08,.45) to (0.08,.1);
        \node at (0.45,-.41) {$\scriptstyle c$};
        \node at (0.07,-.4) {$\scriptstyle b$};
        \node at (-0.28,-.41) {$\scriptstyle a$};
\end{tikzpicture}\:,
&
\begin{tikzpicture}[baseline = -1mm,scale=1]
\draw[-,thick] (0.35,.3) to (0.08,-0.14);
	\draw[-,thick] (0.1,.3) to (-0.04,0.06);
	\draw[-,line width=1pt] (0.085,-.14) to (-0.035,0.06);
	\draw[-,thick] (-0.2,.3) to (0.07,-0.14);
	\draw[-,line width=2pt] (0.08,-.45) to (0.08,-.1);
        \node at (0.45,.4) {$\scriptstyle c$};
        \node at (0.07,.42) {$\scriptstyle b$};
        \node at (-0.28,.4) {$\scriptstyle a$};
\end{tikzpicture}
&=\begin{tikzpicture}[baseline = -1mm,scale=1]
	\draw[-,thick] (0.36,.3) to (0.09,-0.14);
	\draw[-,thick] (0.06,.3) to (0.2,.05);
	\draw[-,line width=1pt] (0.07,-.14) to (0.19,.06);
	\draw[-,thick] (-0.19,.3) to (0.08,-0.14);
	\draw[-,line width=2pt] (0.08,-.45) to (0.08,-.1);
        \node at (0.45,.4) {$\scriptstyle c$};
        \node at (0.07,.42) {$\scriptstyle b$};
        \node at (-0.28,.4) {$\scriptstyle a$};
\end{tikzpicture}
\label{asscoassoc}\end{align}
for $a,b,c > 0$, together with
\begin{align}
\label{altrels}
\begin{tikzpicture}[anchorbase,scale=.8]
	\draw[-,line width=2pt] (0.08,-.8) to (0.08,-.5);
	\draw[-,line width=2pt] (0.08,.3) to (0.08,.6);
\draw[-,thick] (0.1,-.51) to [out=45,in=-45] (0.1,.31);
\draw[-,thick] (0.06,-.51) to [out=135,in=-135] (0.06,.31);
        \node at (-.33,-.05) {$\scriptstyle a$};
        \node at (.45,-.05) {$\scriptstyle b$};
\end{tikzpicture}
&= 
\qbinom{a+b}{a}_{\!q}\:\:
\begin{tikzpicture}[anchorbase,scale=.8]
	\draw[-,line width=2pt] (0.08,-.8) to (0.08,.6);
         \node at (.08,.7) {$\scriptstyle \phantom {a+b}$};
        \node at (.08,-.9) {$\scriptstyle a+b$};
\end{tikzpicture},&
\begin{tikzpicture}[anchorbase,scale=1]
	\draw[-,line width=1.2pt] (0,0) to (.275,.3) to (.275,.7) to (0,1);
	\draw[-,line width=1.2pt] (.6,0) to (.315,.3) to (.315,.7) to (.6,1);
        \node at (0,1.13) {$\scriptstyle c$};
        \node at (0.63,1.13) {$\scriptstyle d$};
        \node at (0,-.1) {$\scriptstyle a$};
        \node at (0.63,-.1) {$\scriptstyle b$};
\end{tikzpicture}
&=
\sum_{\substack{0 \leq s \leq \min(a,c)\\0 \leq t \leq \min(b,d)\\t-s=d-a=c-b}}
q^{st}
\begin{tikzpicture}[anchorbase,scale=1]
	\draw[-,thick] (0.58,0) to (0.58,.2) to (.02,.8) to (.02,1);
	\draw[-,line width=4pt,white] (0.02,0) to (0.02,.2) to (.58,.8) to (.58,1);
	\draw[-,thick] (0.02,0) to (0.02,.2) to (.58,.8) to (.58,1);
	\draw[-,thin] (0,0) to (0,1);
	\draw[-,thin] (0.6,0) to (0.6,1);
        \node at (0,1.13) {$\scriptstyle c$};
        \node at (0.6,1.13) {$\scriptstyle d$};
        \node at (0,-.1) {$\scriptstyle a$};
        \node at (0.6,-.1) {$\scriptstyle b$};
        \node at (-0.1,.5) {$\scriptstyle s$};
        \node at (0.75,.5) {$\scriptstyle t$};
\end{tikzpicture}
\end{align}
for $a,b,c,d > 0$ with $a+b=c+d$.
Here, $\qbinom{n}{s}_{\!q}$ is the $q$-binomial coefficient \cref{qbin},
and splits/merges with a string of thickness zero should be interpreted as identities.
\end{introtheorem}

The positive crossings are important because they define a braiding
making $\qSchur$ into a 
braided monoidal category. 
In fact, positive crossings and their inverses,
the {\em negative crossings}, can be written in terms of merges and splits:
\begin{align*}
\begin{tikzpicture}[baseline=-1mm]
	\draw[-,line width=1pt] (0.3,-.3) to (-.3,.4);
	\draw[-,line width=4pt,white] (-0.3,-.3) to (.3,.4);
	\draw[-,line width=1pt] (-0.3,-.3) to (.3,.4);
        \node at (-0.3,-.4) {$\scriptstyle a$};
        \node at (0.3,-.4) {$\scriptstyle b$};
\end{tikzpicture}
&=\sum_{s=0}^{\min(a,b)}
(-q)^{s}
\begin{tikzpicture}[anchorbase,scale=1]
	\draw[-,line width=1.2pt] (0,0) to (0,1);
	\draw[-,thick] (-0.8,0) to (-0.8,.2) to (-.03,.4) to (-.03,.6)
        to (-.8,.8) to (-.8,1);
	\draw[-,thin] (-0.82,0) to (-0.82,1);
        \node at (-0.81,-.1) {$\scriptstyle a$};
        \node at (0,-.1) {$\scriptstyle b$};
        \node at (-0.4,.9) {$\scriptstyle b-s$};
        \node at (-0.4,.13) {$\scriptstyle a-s$};
\end{tikzpicture}
=
\sum_{s=0}^{\min(a,b)}
(-q)^{s}
\begin{tikzpicture}[anchorbase,scale=1]
	\draw[-,line width=1.2pt] (0,0) to (0,1);
	\draw[-,thick] (0.8,0) to (0.8,.2) to (.03,.4) to (.03,.6)
        to (.8,.8) to (.8,1);
	\draw[-,thin] (0.82,0) to (0.82,1);
        \node at (0.81,-.1) {$\scriptstyle b$};
        \node at (0,-.1) {$\scriptstyle a$};
        \node at (0.4,.9) {$\scriptstyle a-s$};
        \node at (0.4,.13) {$\scriptstyle b-s$};
\end{tikzpicture},\\
\begin{tikzpicture}[baseline=-1mm]
 \draw[-,line width=1pt] (-0.3,-.3) to (.3,.4);
	\draw[-,line width=4pt,white] (0.3,-.3) to (-.3,.4);
 \draw[-,line width=1pt] (0.3,-.3) to (-.3,.4);
        \node at (-0.3,-.4) {$\scriptstyle a$};
        \node at (0.3,-.4) {$\scriptstyle b$};
\end{tikzpicture}
:=\left(\begin{tikzpicture}[baseline=-1mm]
	\draw[-,line width=1pt] (0.3,-.3) to (-.3,.4);
	\draw[-,line width=4pt,white] (-0.3,-.3) to (.3,.4);
	\draw[-,line width=1pt] (-0.3,-.3) to (.3,.4);
        \node at (-0.3,-.4) {$\scriptstyle b$};
        \node at (0.3,-.4) {$\scriptstyle a$};
\end{tikzpicture}\right)^{-1}
&=\sum_{s=0}^{\min(a,b)}
(-q)^{-s}
\begin{tikzpicture}[anchorbase,scale=1]
	\draw[-,line width=1.2pt] (0,0) to (0,1);
	\draw[-,thick] (-0.8,0) to (-0.8,.2) to (-.03,.4) to (-.03,.6)
        to (-.8,.8) to (-.8,1);
	\draw[-,thin] (-0.82,0) to (-0.82,1);
        \node at (-0.81,-.1) {$\scriptstyle a$};
        \node at (0,-.1) {$\scriptstyle b$};
        \node at (-0.4,.9) {$\scriptstyle b-s$};
        \node at (-0.4,.13) {$\scriptstyle a-s$};
\end{tikzpicture}
=
\sum_{s=0}^{\min(a,b)}
(-q)^{-s}
\begin{tikzpicture}[anchorbase,scale=1]
	\draw[-,line width=1.2pt] (0,0) to (0,1);
	\draw[-,thick] (0.8,0) to (0.8,.2) to (.03,.4) to (.03,.6)
        to (.8,.8) to (.8,1);
	\draw[-,thin] (0.82,0) to (0.82,1);
        \node at (0.81,-.1) {$\scriptstyle b$};
        \node at (0,-.1) {$\scriptstyle a$};
        \node at (0.4,.9) {$\scriptstyle a-s$};
        \node at (0.4,.13) {$\scriptstyle b-s$};
\end{tikzpicture}.
\end{align*}
The following gives a slighly more efficient presenatation for $\qSchur$ using only the merges and splits as generating morphisms.

\begin{introtheorem}\label{th1}
The monoidal category $\qSchur$ is generated
by the objects $(r)$ for $r > 0$ and the
morphisms
$\begin{tikzpicture}[baseline=-1mm,scale=.6]
	\draw[-,line width=1pt] (0.28,-.3) to (0.08,0.04);
	\draw[-,line width=1pt] (-0.12,-.3) to (0.08,0.04);
	\draw[-,line width=2pt] (0.08,.4) to (0.08,0);
        \node at (-0.22,-.43) {$\scriptstyle a$};
        \node at (0.35,-.43) {$\scriptstyle b$};
\end{tikzpicture}$ and
$\begin{tikzpicture}[baseline=0mm,scale=.6]
	\draw[-,line width=2pt] (0.08,-.3) to (0.08,0.04);
	\draw[-,line width=1pt] (0.28,.4) to (0.08,0);
	\draw[-,line width=1pt] (-0.12,.4) to (0.08,0);
        \node at (-0.22,.53) {$\scriptstyle a$};
        \node at (0.36,.55) {$\scriptstyle b$};
\end{tikzpicture}$ for $a,b > 0$, 
subject only to the relations \cref{asscoassoc} together with 
one of the equivalent {\em square-switch relations}
\begin{align}\label{squareswitch}
\begin{tikzpicture}[anchorbase,scale=1]
	\draw[-,thick] (0,0) to (0,1);
	\draw[-,thick] (.015,0) to (0.015,.2) to (.57,.4) to (.57,.6)
        to (.015,.8) to (.015,1);
	\draw[-,line width=1.2pt] (0.6,0) to (0.6,1);
        \node at (0.6,-.12) {$\scriptstyle b$};
        \node at (0,-.1) {$\scriptstyle a$};
        \node at (0.3,.84) {$\scriptstyle c$};
        \node at (0.3,.19) {$\scriptstyle d$};
\end{tikzpicture}
&=
\!\!\!\sum_{s=\max(0,c-b)}^{\min(c,d)}
\qbinom{a -b+c-\!d}{s}_{\!q}\ 
\begin{tikzpicture}[anchorbase,scale=.9]
	\draw[-,line width=1.2pt] (0,0) to (0,1);
	\draw[-,thick] (0.8,0) to (0.8,.2) to (.03,.4) to (.03,.6)
        to (.8,.8) to (.8,1);
	\draw[-,thin] (0.82,0) to (0.82,1);
        \node at (0.81,-.12) {$\scriptstyle b$};
        \node at (0,-.1) {$\scriptstyle a$};
        \node at (0.4,.9) {$\scriptstyle d-s$};
        \node at (0.4,.13) {$\scriptstyle c-s$};
\end{tikzpicture},&
\begin{tikzpicture}[anchorbase,scale=1]
	\draw[-,thick] (0,0) to (0,1);
	\draw[-,thick] (-.015,0) to (-0.015,.2) to (-.57,.4) to (-.57,.6)
        to (-.015,.8) to (-.015,1);
	\draw[-,line width=1.2pt] (-0.6,0) to (-0.6,1);
        \node at (-0.6,-.12) {$\scriptstyle b$};
        \node at (0,-.1) {$\scriptstyle a$};
        \node at (-0.3,.84) {$\scriptstyle c$};
        \node at (-0.3,.19) {$\scriptstyle d$};
\end{tikzpicture}&=\!\!\!
\sum_{s=\max(0,c-b)}^{\min(c,d)}
\qbinom{a-b+c-d}{s}_{\!q}\ 
\begin{tikzpicture}[anchorbase,scale=.9]
	\draw[-,line width=1.2pt] (0,0) to (0,1);
	\draw[-,thick] (-0.8,0) to (-0.8,.2) to (-.03,.4) to (-.03,.6)
        to (-.8,.8) to (-.8,1);
	\draw[-,thin] (-0.82,0) to (-0.82,1);
        \node at (-0.81,-.12) {$\scriptstyle b$};
        \node at (0,-.1) {$\scriptstyle a$};
        \node at (-0.4,.9) {$\scriptstyle d-s$};
        \node at (-0.4,.13) {$\scriptstyle c-s$};
\end{tikzpicture}
\end{align}
for $a,b,c,d \geq 0$ 
with $d \leq a$ and $c \leq b+d$.
\end{introtheorem}

The presentations for $\qSchur$ 
in \cref{th1alt,th1} are not new, e.g., the relations can be found in \cite{LT} (with a different choice of normalization for the positive crossings coming from quantum $SL_n$ rather than quantum $GL_n$).
We give complete proofs here, rather than attempting to adapt related results 
already in the literature such as \cite{doty}.
Our general approach to the definition of $\qSchur$, equipping each of its morphism spaces with a standard basis over $\Z[q,q^{-1}]$ from the outset with structure constants which can computed algorithmically, 
facilitates calculations which seem quite awkward otherwise; 
e.g., see \cref{1984} for a 
formula for the composition of two positive crossings.
The ability to compute products effectively is also exploited in the proof of the
straightening formula in \cref{cellb}.

This straightening formula is
the key ingredient 
in the proof of \cref{th2}, which
constructs a second basis for morphism spaces in 
$\qSchur$. We formulate this in terms of the path algebra 
\begin{equation}\label{pathalgebra}
\sch := 
\bigoplus_{\lambda,\mu\in\sComp}
\Hom_{\qSchur}(\mu,\lambda)
\end{equation}
viewed as a locally unital algebra with 
distinguished idempotents $\{1_\lambda\:|\:\lambda \in \sComp\}$ arising from the identity endomorphisms of the objects of $\qSchur$.
Multiplication in $\sch$ is
induced by composition.
Let $\Par$ be the subset of $\sComp$ consisting of 
 all {\em partitions}, that is, ordered sequences
$\kappa = (\kappa_1 \geq \cdots \geq \kappa_\ell)$ of positive integers
for $\ell \geq 0$.
For $\lambda \in \sComp$ and $\kappa \in \Par$, we denote
the usual set of all semistandard tableaux of shape $\kappa$ and content $\lambda$ by $\Std(\lambda,\kappa)$.
%; this set is empty unless $\lambda \leq \kappa$ in the dominance ordering.
For $P \in \Std(\lambda,\kappa)$, let
$A(P) \in \Mat{\lambda}{\kappa}$ be the matrix whose $ij$-entry records the number of times $i$ appears on row $j$ of $P$.
For the definition of ``symmetrically-based quasi-hereditary algebra''
used in the statement of the theorem, see \cref{bqh}.
The triangular bases in this definition
are {\em cellular bases} in the sense of \cite{GL}.
However, the axioms are simpler than the 
ones for a cellular algebra; they are also more restrictive since it follows automatically that the underlying
algebra is a split quasi-hereditary algebra with duality in the sense of \cite{CPS2}.% (assuming now that the weight poset is finite).

\begin{introtheorem}\label{th2}
The locally unital algebra 
$\sch = \bigoplus_{\lambda,\mu\in\sComp} 1_\lambda \sch 1_\mu$ is a symmetrically-based quasi-hereditary algebra with
weight poset $\Par$ ordered by the dominance ordering $\leq$,
anti-involution $\T:\sch \rightarrow \sch, \xi_A \mapsto \xi_{A^\tr}$, and
triangular basis
consisting of the {\em codeterminants}
$\xi_{A(P)} \xi_{A(Q)^\tr}$
for 
$(P,Q) \in \bigcup_{\lambda,\mu \in \sComp, \kappa \in \Par}\Std(\lambda,\kappa)\times \Std(\mu,\kappa)$.
\end{introtheorem}

There is a third remarkable basis in this subject,
the {canonical basis}, which
appeared originally in the context of 
$q$-Schur algebras in \cite{BLM} and was studied in detail by Du \cite{Du1,Du,Du2} from the perspective of Hecke algebras; see also \cite[Ch.~9]{BDPW}.
To define it, 
take $\lambda,\mu \in \sComp$ such that $r := \sum_i \lambda_i = \sum_i \mu_i$,
and $A \in \Mat{\lambda}{\mu}$.
Writing $d_A^+ \in (S_\lambda \backslash S_r / S_\mu)_{\max}$ for the maximal length double coset representative indexed by $A$, let
$$
\theta_A := \sum_{B \in \Mat{\lambda}{\mu}}
q^{\ell(d_A^+) - \ell(d_B^+)} P_{d_A^+,d_B^+}(q^{-2})\xi_B,
$$
where
$P_{x,y}(t) \in \Z[t]$ is the
Kazhdan-Lusztig polynomial for $x,y \in S_r$.
Then the {\em canonical basis} for 
$\Hom_{\qSchur}(\mu,\lambda)$ is $\{\theta_A\:|\:A \in \Mat{\lambda}{\mu}\}$.
The canonical basis can also be defined in terms of
the {\em bar involution} $-:\qSchur \rightarrow \qSchur$,
the anti-linear
strict monoidal functor which fixes objects and the generating merge and split morphisms: $\theta_A$ is the unique
morphism in $\Hom_{\qSchur}(\mu,\lambda)$ such that
$\overline{\theta}_A = \theta_A$
and $\theta_A \equiv \xi_A \pmod{\sum_{B \in \Mat{\lambda}{\mu}}
q \Z[q] \xi_B}$.
Note also that the bar involution interchanges positive and negative crossings.

The canonical basis makes the path algebra $\sch$ into 
a ``standardly full-based algebra'' 
using the language  of 
\cite{DR}, with the same weight poset and cell ideals as the ones arising from the codeterminant basis in \cref{th2}. 
This follows from the results in 
\cite[$\S$5.3]{DR}, which imply that the canonical basis is cellular, hence, equivalent to a triangular basis; see \cref{fromsmallerones} for a precise statement.
\cref{th2} could be deduced as a consequence of this like in \cite[$\S$5.5]{DR}. 
It could also
be deduced from R.~M.~Green's construction \cite{RMGreen} of the 
$q$-analog of J.~A.~Green's codeterminant basis for the Schur algebra. The short 
self-contained proof of \cref{th2} given here is 
 similar to the one in \cite{RMGreen}
  (and in \cite{W} when $q=1$), but
incorporates simplifications made possible by working in the less constrained setting of the $q$-Schur category.
Analogous bases for cyclotomic $q$-Schur algebras of all levels (not merely level one)
have been constructed
in \cite[Th.~6.6]{DJM} by a different method.

At least one of these new bases (codeterminant or canonical) is needed 
in order to understand 
a certain truncation $\qSchur_n$
of the $q$-Schur category.
By definition, this is the quotient of $\qSchur$ by the two-sided tensor ideal $\cI_n$ generated by the identity endomorphisms $1_{(r)}$ for $r > n$.
The presentation for $\qSchur_n$ arising from \cref{th1} makes it clear that it is a version of Cautis-Kamnitzer-Morrison's $U_q\mathfrak{gl}_n$-web category, or rather, its positive half involving only upward-pointing strings.
The ideal $\cI_n$ is compatible with the basis from
\cref{th2}. Consequently, the path algebra of 
$\qSchur_n$ is also a symmetrically-based quasi-hereditary algebra with triangular basis given by the images of the codeterminants
$\xi_{A(P)}\xi_{A(Q)^\tr}$
for all pairs $(P,Q)$ of
semistandard tableaux whose shape $\kappa$ satisfies $\kappa_1 \leq n$.
This basis is of a similar nature to the integral bases
for morphism spaces in this category constructed in \cite{Elias}. The canonical basis also induces a cellular basis for the path algebra of $\qSchur_n$.

Let $\k$ be a field viewed as a $\Z[q,q^{-1}]$-algebra in some way, and consider the specialization $\qSchur_n(\k) := 
\k \otimes_{\Z[q,q^{-1}]} \qSchur_n$.
Also let $U_n$ be 
Lusztig's $\Z[q,q^{-1}]$-form for the quantized enveloping algebra $U_q\mathfrak{gl}_n$
with Chevalley generators $E_i,F_i\:(1 \leq i \leq n-1)$ and
$D_i^{\pm 1}\:(1 \leq i \leq n)$.
Let $U_n(\k) := \k \otimes_{\Z[q,q^{-1}]} U_n$.
We view it as a Hopf algebra with comultiplication $\Delta$ satisfying
\begin{align*}
\Delta(E_i) &= 1 \otimes E_i  + E_i \otimes D_i^{-1} D_{i+1},&
\Delta(F_i) &= F_i \otimes 1 + D_i D_{i+1}^{-1} \otimes F_i,&
\Delta(D_i)&=D_i \otimes D_i.
\end{align*}
The natural $U_n(\k)$-module $V$ 
is the vector space with
basis
$v_1,\dots,v_n$ such that
$E_i v_j = \delta_{i+1,j} v_i$,
$F_i v_j = \delta_{i,j} v_{i+1}$,
$D_i v_j = q^{\delta_{i,j}} v_j$.
Its $r$th quantum exterior power $\bigwedge^r V$ is
a certain quotient of the $r$th tensor power $V^{\otimes r}$
with a basis given by the monomials $v_{i_1}\wedge\cdots\wedge v_{i_r}$
that are images of the pure tensors
$v_{i_1}\otimes\cdots\otimes v_{i_r}$ for
$1 \leq i_1 < \dots < i_r \leq n$.
Let $\qTilt_n(\k)$, the category of 
{\em polynomial tilting modules}, be the full additive Karoubian monoidal subcategory of
$U_n(\k)\mod$
generated by the exterior powers
$\bigwedge^r V$ for all $r \geq 0$.
This is a braided (but not rigid) 
monoidal category with braiding 
$\R:-\otimes-\stackrel{\sim}{\Rightarrow} 
-\otimes^\rev-$ 
defined so that
\begin{align}\label{rmat}
\R_{V,V}:V \otimes V &\rightarrow V \otimes V,
&v_i \otimes v_j &\mapsto \begin{cases}
v_j \otimes v_i&\text{if $i < j$,}\\
q^{-1} v_j \otimes v_i&\text{if $i=j$,}\\
v_j \otimes v_i - (q-q^{-1}) v_i \otimes v_j &\text{if $i > j$.}
\end{cases}
\end{align}
If $\k$ is of characteristic 0 and the image of $q$ in $\k$ is not a root of unity, $\qTilt_n(\k)$ is a semisimple Abelian category, and
the following theorem can be deduced from \cite{CKM}.

\begin{introtheorem}\label{th3}
There is a $\k$-linear monoidal 
functor
$\Sigma_n:\qSchur_n(\k) \rightarrow \qTilt_n (\k)$
taking the generating object $(r)$ to $\bigwedge^r V$,
the merge 
$\begin{tikzpicture}[anchorbase,scale=.6]
	\draw[-,line width=1pt] (0.28,-.3) to (0.08,0.04);
	\draw[-,line width=1pt] (-0.12,-.3) to (0.08,0.04);
	\draw[-,line width=2pt] (0.08,.4) to (0.08,0);
        \node at (-0.22,-.43) {$\scriptstyle a$};
        \node at (0.35,-.43) {$\scriptstyle b$};
\end{tikzpicture}$
to
the natural surjection
%\begin{align}
${\textstyle\bigwedge^a V \otimes \bigwedge^b V}
\twoheadrightarrow {\textstyle\bigwedge^{a+b} V}$,
%,\\\notag v_{i_1}\wedge\cdots\wedge v_{i_a} \otimes v_{j_1}\wedge\cdots\wedge v_{j_b} &\mapsto v_{i_1}\wedge\cdots\wedge v_{i_a} \wedge  v_{j_1}\wedge\cdots\wedge v_{j_b},\\\intertext{
%for $1 \leq i_1<\cdots<i_a\leq n, 1 \leq j_1<\cdots<j_n\leq n$,
and the split
$\begin{tikzpicture}[anchorbase,scale=.6]
	\draw[-,line width=2pt] (0.08,-.3) to (0.08,0.04);
	\draw[-,line width=1pt] (0.28,.4) to (0.08,0);
	\draw[-,line width=1pt] (-0.12,.4) to (0.08,0);
        \node at (-0.22,.53) {$\scriptstyle a$};
        \node at (0.36,.55) {$\scriptstyle b$};
\end{tikzpicture}$
to the inclusion
$
{\textstyle\bigwedge^{a+b} V}
\hookrightarrow 
{\textstyle\bigwedge^a V \otimes \bigwedge^b V}$
defined by
$$
v_{i_1}\wedge\cdots\wedge v_{i_{a+b}}
\mapsto
q^{-ab} \!\!\!\!\!\!\!\!\!\!\!\!\sum_{w \in (S_{a+b}/ S_a \times S_b)_{\min}}\!\!\!\!\!\!\!\!\!\!\!\!
(-q)^{\ell(w)} 
v_{i_{w(1)}}\wedge\cdots\wedge v_{i_{w(a)}}\otimes
v_{i_{w(a+1)}}\wedge\cdots\wedge v_{i_{w(a+b)}}
$$
for $1 \leq i_1 < \cdots < i_{a+b}\leq n$.
%where $(S_{a+b} / S_a \times S_b)_{\min}$ is the set of minimal length left coset representatives and $\ell(g)$ is the usual length of the permutation $g$.
This functor is full and faithful, and it induces a monoidal equivalence
between the 
additive Karoubi envelope of $\qSchur_n(\k)$
and $\qTilt_n (\k)$.
\end{introtheorem}

The monoidal functor $\Sigma_n$ of \cref{th3}
is not a braided monoidal functor---it takes
 the positive crossing 
$\begin{tikzpicture}[anchorbase,scale=.6]
	\draw[-,thick] (0.3,-.3) to (-.3,.4);
	\draw[-,line width=5pt,white] (-0.3,-.3) to (.3,.4);
	\draw[-,thick] (-0.3,-.3) to (.3,.4);
    \node at (0.3,-.5) {$\scriptstyle b$};
        \node at (-0.3,-.5) {$\scriptstyle a$};
\end{tikzpicture}$
to $(-1)^{ab} \R_{\bigwedge^b V, \bigwedge^a V}^{-1}$ 
rather than to $\R_{\bigwedge^a V, \bigwedge^b V}$.
This twist, which may at first seem inconvenient,
is reasonable since the proof
involves some Ringel duality---the generating object $(r)$
of the $q$-Schur category corresponds more naturally to the $r$th
quantum {symmetric power} of the natural module rather than its
exterior power. 

There is one more important explanation to be made:
subsequently, the notation $\qSchur$ will be used to denote a slightly larger version of the $q$-Schur category than appears in this introduction,
with objects that are indexed by {\em all} compositions, not just strict ones. In other words, we adjoin an additional generating object $(0)$ which is isomorphic but not equal to the strict
identity object $\one$.
We prefer to use the same notation for both versions---it should be clear from context whether we are working with or without strings of thickness zero. The natural inclusion of the $q$-Schur category as defined in the introduction into the one with 
0-strings is a monoidal equivalence, making it easy to go back and forth between the two versions.
One advantage of $q$-Schur category {\em with} 0-strings is that there is a surjective algebra homomorphism from $U_n$ to the path algebra of the full subcategory
whose objects are compositions with exactly $n$ parts. Actually, it is more convenient to work with Lusztig's modified form $\dot U_n$ here; see \cref{chickencurry}.
Using this connection, the approach to $q$-Schur algebras
taken in \cite{doty}, exploiting Lusztig's refined Peter-Weyl theorem for $\dot U_n$ \cite[Sec.~29.3]{Lubook}, could be adapted to give yet another approach to the results here.

\vspace{2mm}
\noindent
{\em Acknowledgements.}
The first draft of this article was written in February 2020. I would like to thank Elijah Bodish, one of the few mathematicians I saw at the University of Oregon
for many months after that,
for his interest in this material.
This paper is dedicated to Gary Seitz who had a profound influence on my life---indeed, without Gary's encouragement and mentorship,
I would never have found my way to Oregon 
in the first place.

%% file: s2-combinatorics.tex
\section{Double coset combinatorics}\label{s2-combinatorics}

A {\em composition} $\lambda\vDash r$ is a finite 
sequence $\lambda =
(\lambda_1,\dots,\lambda_\ell)$
of non-negative integers summing to $r$.
We write $\ell(\lambda)$ for the total number $\ell$ of parts, which is allowed to be zero, and $|\lambda|$ for the sum of the parts.
We emphasize that we treat compositions of different lengths as being different, 
e.g., $() \neq (0) \neq (0,0)$.
A {\em partition} $\lambda \vdash r$ is a composition $\lambda=(\lambda_1,\dots,\lambda_\ell) \vDash r$
whose parts satisfy $\lambda_1 \geq \cdots \geq \lambda_\ell > 0$.
For partitions, we allow to write $\lambda_r$ even if $r > \ell(\lambda)$, in which case $\lambda_r = 0$.
We denote the sets of all compositions
and all partitions by $\Comp$ and $\Par$, respectively.
Let $\leq$ be the usual dominance ordering on $\Par$.

We denote the transposition $(i\ i\!+\!1)$
in the symmetric group $S_r$ by $s_i$,
$\ell:S_r \rightarrow \N$ is the length function, and $w_r \in S_r$ is the longest element.
Elements of $S_r$ act on the {\em left} on
the set $\{1,\dots,r\}$.
There is also a {\em right} action of $S_r$ on $\Z^r$ by place permutation: for 
$\bi =
(i_1,\dots,i_r) \in \Z^r$
and
$w \in S_r$, the $r$-tuple $\bi \cdot w$ has $j$th entry $i_{w(j)}$. 
For $\lambda=(\lambda_1,\dots,\lambda_\ell) \vDash r$,
the set
\begin{equation}\label{ilambda}
\I_\lambda := \left\{\bi = (i_1,\dots,i_r) \in \Z^r \:\big|\:
\#\!\left\{k=1,\dots,r\:|\:i_k = i\right\} = \lambda_i
\text{ for all }i \in \nset{\ell(\lambda)}\right\}
\end{equation}
is a single orbit under this action.
Also let 
$\bi^\lambda = (i^\lambda_1,\dots,i^\lambda_r)$
denote the unique element of $\I_\lambda$ whose entries 
are in weakly increasing order.
Its stabilizer in $S_r$ is
the parabolic subgroup
$S_\lambda = S_{\lambda_1}\times\cdots\times S_{\lambda_\ell}$.
\iffalse
Often useful is the bijection \begin{equation}
d:\I_\lambda \stackrel{\sim}{\rightarrow} (S_\lambda \backslash S_r)_{\min}
\end{equation}
taking $\bi \in \I_\lambda$ to the unique minimal
length $S_\lambda \backslash S_r$-coset representative such that
$\bi = \bi^\lambda \cdot d(\bi)$.
To compute $d(\bi)$ in practice, one replaces 
all the entries $1$ of
of $\bi$ in order from left to right by
$1,\dots,\lambda_1$,
the entries $2$ by $\lambda_1+1,\dots,\lambda_1+\lambda_2$,
and so on.
The resulting sequence is 
the permutation $d(\bi)$ written in one-line notation.
\fi

For $\lambda,\mu\vDash r$, the symmetric group $S_r$ acts diagonally
on the right on 
$\I_\lambda \times \I_\mu$. The orbits are parametrized by the set
$\Mat{\lambda}{\mu}$ of all $\ell(\lambda)\times\ell(\mu)$ matrices
with non-negative integer entries such that
the entries in the $i$th row sum to
$\lambda_i$
and the entries in the $j$th column  sum to $\mu_j$ for all
$i\in \nset{\ell(\lambda)}$ and $j\in\nset{\ell(\mu)}$.
For $A = (a_{i,j}) \in \Mat{\lambda}{\mu}$, the corresponding $S_r$-orbit
on $\I_\lambda\times\I_\mu$ is
\begin{equation}
\Pi_A := \left\{(\bi,\bj) \in \I_\lambda\times\I_\mu \:\Bigg|\:
\begin{array}{l}
\#\!\left\{k=1,\dots,r\:|\:(i_k,j_k) = (i,j)\right\} = a_{i,j}\\
\text{for all  }
i \in \nset{\ell(\lambda)},   j\in \nset{\ell(\mu)}\end{array}
\right\}\!.
\end{equation}
The set $\Mat{\lambda}{\mu}$ is actually just one of many different sets
used in the literature to parametrize
the orbits of $S_r$ on $\I_\lambda\times\I_\mu$. 
Another is by the set $\Row(\lambda,\mu)$ of {\em row tableaux} of shape $\mu$ and content $\lambda$, that is, left justified arrays with $\mu_1$ boxes in row 1 (the top row),
$\mu_2$ boxes in row 2, and so on, with boxes
filled with integers so that entries are weakly increasing in order from left to right along each row, and there are a total of $\lambda_1$ entries equal to $1$, $\lambda_2$ equal to 2, and so on.
We use the explicit bijection
\begin{equation}\label{Amap}
A:\Row(\lambda,\mu) \rightarrow
\Mat{\lambda}{\mu}
\end{equation}
taking $P \in \Row(\lambda,\mu)$
to the matrix $A(P) \in \Mat{\lambda}{\mu}$
whose $ij$-entry
records the number of times $i$ appears on row $j$ of $P$.
The inverse bijection maps
$A \in \Mat{\lambda}{\mu}$ to the row tableau $P \in \Row(\lambda,\mu)$ whose $j$th row is equal to 
$1^{a_{1,j}}\ 2^{a_{2,j}}\ \cdots\ \ell^{a_{\ell(\lambda),j}}$.

A third way to parametrize orbits is by the double coset diagrams introduced already in the introduction.
We gave already there 
an example in which
$\lambda = (4,5)$, $\mu = (3,2,4)$,
for which the matrix $A \in \Mat{\lambda}{\mu}$, the corresponding double coset diagram, and the corresponding
row tableau $P\in \Row(\lambda,\mu)$ are
\begin{align}
A=\begin{bmatrix} 1&0&3\\2&2&1\end{bmatrix}
\qquad\leftrightarrow\qquad
\begin{tikzpicture}[anchorbase,scale=1.5]
\draw[-,line width=.6mm] (.212,.5) to (.212,.39);
\draw[-,line width=.75mm] (.595,.5) to (.595,.39);
\draw[-,line width=.15mm] (0.0005,-.396) to (.2,.4);
\draw[-,line width=.3mm] (0.01,-.4) to (.59,.4);
\draw[-,line width=.3mm] (.4,-.4) to (.607,.4);
\draw[-,line width=.45mm] (.79,-.4) to (.214,.4);
\draw[-,line width=.15mm] (.8035,-.398) to (.614,.4);
\draw[-,line width=.3mm] (.4006,-.5) to (.4006,-.395);
\draw[-,line width=.6mm] (.788,-.5) to (.788,-.395);
\draw[-,line width=.45mm] (0.011,-.5) to (0.011,-.395);
\node at (0.05,0.15) {$\scriptstyle 1$};
\node at (0.79,0.05) {$\scriptstyle 1$};
\node at (0.35,0.37) {$\scriptstyle 3$};
\node at (0.17,-0.3) {$\scriptstyle 2$};
\node at (0.5,-0.3) {$\scriptstyle 2$};
\end{tikzpicture}
\qquad\leftrightarrow \qquad P=
\begin{tikzpicture}[anchorbase,scale=1.2]
  \draw (0,0) to (1,0);
  \draw (0,.25) to (1,.25);
  \draw (0,.5) to (.75,.5);
  \draw (0,.75) to (.75,.75);
  \draw (0,0) to (0,.75);
  \draw (.25,0) to (.25,.75);
  \draw (.5,0) to (.5,.75);
  \draw (.75,.5) to (.75,.75);
  \draw (.75,0) to (.75,.25);
  \draw (1,0) to (1,.25);
  \node at (0.125,0.125) {$\scriptstyle 1$};
    \node at (0.375,0.125) {$\scriptstyle 1$};
  \node at (0.625,0.125) {$\scriptstyle 1$};
  \node at (0.875,0.125) {$\scriptstyle 2$};
    \node at (0.125,0.375) {$\scriptstyle 2$};
    \node at (0.375,0.375) {$\scriptstyle 2$};
     \node at (0.125,0.625) {$\scriptstyle 1$};
    \node at (0.375,0.625) {$\scriptstyle 2$};
  \node at (0.625,0.625) {$\scriptstyle 2$};
  \end{tikzpicture}
\label{filler1}
\end{align}
Unlike in the introduction, we are now allowing compositions with parts equal to 0, so double coset diagrams can also have strings labelled by 0. In fact, it is harmless to omit these zero thickness strings from the diagram entirely, but one should mark their endpoints.
Here is an example with
$\lambda = (4,0,5,0)$ and $\mu = (3,2,0,4)$:
\begin{align}\label{filler2}
A=\begin{bmatrix} 1&0&0&3\\0&0&0&0\\2&2&0&1\\0&0&0&0\end{bmatrix}
\qquad\leftrightarrow\qquad
\begin{tikzpicture}[anchorbase,scale=1.5]
\smallspot{.8,-.49};
\smallspot{.4,.49};
\smallspot{1.2,.49};
\draw[-,line width=.6mm] (.014,.5) to (.014,.39);
\draw[-,line width=.75mm] (.795,.5) to (.795,.39);
\draw[-,line width=.15mm] (0.0005,-.396) to (0.0005,.4);
\draw[-,line width=.3mm] (0.01,-.4) to (.79,.4);
\draw[-,line width=.3mm] (.4,-.4) to (.807,.4);
\draw[-,line width=.45mm] (1.19,-.4) to (.014,.4);
\draw[-,line width=.15mm] (1.21,-.398) to (.806,.4);
\draw[-,line width=.3mm] (.4006,-.5) to (.4006,-.395);
\draw[-,line width=.6mm] (1.197,-.5) to (1.197,-.395);
\draw[-,line width=.45mm] (0.011,-.5) to (0.011,-.395);
\node at (0.07,0.05) {$\scriptstyle 1$};
\node at (1.08,0.05) {$\scriptstyle 1$};
\node at (0.32,0.32) {$\scriptstyle 3$};
\node at (0.27,-0.25) {$\scriptstyle 2$};
\node at (0.57,-0.25) {$\scriptstyle 2$};
\end{tikzpicture}
\qquad\leftrightarrow \qquad P=
\begin{tikzpicture}[anchorbase,scale=1.2]
  \draw (0,-.25) to (1,-.25);
  \draw (0,0) to (1,0);
  \draw (0,.25) to (.5,.25);
  \draw (0,.5) to (.75,.5);
  \draw (0,.75) to (.75,.75);
  \draw (0,-.25) to (0,.75);
  \draw (.25,-.25) to (.25,0);
  \draw (.5,-.25) to (.5,0);
  \draw (.25,.25) to (.25,.75);
  \draw (.5,0.25) to (.5,.75);
  \draw (.75,.5) to (.75,.75);
  \draw (.75,0) to (.75,-.25);
  \draw (1,0) to (1,-.25);
  \node at (0.125,-0.125) {$\scriptstyle 1$};
    \node at (0.375,-0.125) {$\scriptstyle 1$};
  \node at (0.625,-0.125) {$\scriptstyle 1$};
  \node at (0.875,-0.125) {$\scriptstyle 3$};
    \node at (0.125,0.375) {$\scriptstyle 3$};
    \node at (0.375,0.375) {$\scriptstyle 3$};
     \node at (0.125,0.625) {$\scriptstyle 1$};
    \node at (0.375,0.625) {$\scriptstyle 3$};
  \node at (0.625,0.625) {$\scriptstyle 3$};
  \end{tikzpicture}
\end{align}
Two other sets in bijection with $\Mat{\lambda}{\mu}$
are the sets $(S_\lambda \backslash S_r / S_\mu)_{\min}$
and $(S_\lambda\backslash S_r / S_\mu)_{\max}$ of minimal length and maximal length double coset representatives.
For $A \in\Mat{\lambda}{\mu}$, we denote the corresponding
elements of
$(S_\lambda \backslash S_r / S_\mu)_{\min}$
and $(S_\lambda\backslash S_r / S_\mu)_{\max}$
by $d_A$ and $d_A^+$, respectively.

\iffalse
If $P \in \Row(\lambda,\mu)$ is the row tableau associated to $A$, we also have that
\begin{align}\label{dcosets}
d_A &= d(\ZigzagDown(P))),&
d_A^+ &= d(\ZigzagUp(P)) w_r
\end{align}
where 
$\ZigzagDown(P),\ZigzagUp(P) \in I_\lambda$ are the sequences obtained by reading the entries of $P$ in the order indicated by the direction of the arrow.
\fi

\iffalse
Note also that $\ell(d_A)$ is the sum of the products $ab$ for each
of the crossings of strings of thickness $a$ and $b$ in the
double coset diagram of $A$. In the example, $\ell(d_A) =
12$.
One more statistic that will be needed later is
\begin{equation}\label{strata}
\height(A) := 
\sum_{1 \leq i < j \leq \ell(\lambda)} \sum_{k=1}^{\ell(\mu)} a_{i,k} a_{j,k}.
\end{equation}
This is the length of the longest element of
$(S_{\mu^+}\backslash
S_{\mu})_{\min}$ for $\mu^+$ as in the following lemma.
\fi

\begin{lemma}\label{doso}
Given $\lambda,\mu\vDash r$ and $A \in \Mat{\lambda}{\mu}$, 
let $\lambda^- \vDash r$ (resp., $\mu^+ \vDash r$)
be obtained by reading the entries of $A$ in
order along rows starting with the top row (resp., in order down columns starting with the leftmost column).
We have that 
\begin{align*}
\big(S_\lambda d_A\big) \cap \big(d_A S_\mu\big)
&=
d_A S_{\mu^+}= S_{\lambda^-}\ d_A.
\end{align*}
Every
element $w\in S_\lambda d_A S_\mu$ can be written
uniquely as $x d_A\, y$ for $x \in (S_\lambda / S_{\lambda^-})_{\min}$,
$y \in S_\mu$
(resp., $x d_A\, y$ for $x \in S_\lambda$, $y \in
(S_{\mu^+}\backslash S_\mu)_{\min}$), and we have that
$\ell(x d_A\, y) = \ell(x)+\ell(d_A)+\ell(y)$.
%In particular, $d_A^+ = w_\lambda w_{\lambda^-} d_A w_\mu = w_\lambda d_A w_{\mu^-} w_\mu$, writing $w_\nu$ for the longest element of the parabolic subgroup $S_\nu$.
\end{lemma}

\begin{proof}
This follows from \cite[Lemma 1.6]{DJgl}.
\end{proof}

The double coset diagram
gives a convenient visual way to translate $A
\in \Mat{\lambda}{\mu}$ into
the minimal length double coset representatives $d_A$.
Alternatively, to obtain $d_A$, let
$(i_1,\dots,i_r) \in I_\lambda$ 
be the sequence $\ZigzagDown(P)$
obtained by reading the entries of the corresponding row tableau $P$ from left to right 
along rows, starting with the top row.
Then replace the $\lambda_1$ entries equal to 1 in this sequence by $1,\dots,\lambda_1$ in increasing order, 
the $\lambda_2$ entries equal to 2 by $\lambda_1+1,\dots,\lambda_1+\lambda_2$
in increasing order, and so on.
The result is $d_A$ written in one-line notation.
To compute $d_A^+$, we instead start from 
the sequence $\ZigzagDownRev(P)$ obtained by reading
entries of $P$ from right to left along rows, starting with the top row. Then we replace the entries $1$
by $\lambda_1,\dots,1$ in decreasing order, the entries 2
by $\lambda_1+\lambda_2,\dots,\lambda_1+1$ in decreasing order, and so on.
In the example \cref{filler1}, 
$\ZigzagDown(P) = (1,2,2,2,2,1,1,1,2)$ so $d_A = (1,5,6,7,8,2,3,4,9)$,
and $\ZigzagDownRev(P) = (2,2,1,2,2,2,1,1,1)$ so $d_A^+
= (9,8,4,7,6,5,3,2,1)$.

Let $\leq$ be the {\em Bruhat ordering} on the symmetric group
(so the identity element is {\em minimal}).
This restricts to partial
orders on the sets 
$(S_\lambda\backslash S_r / S_\mu)_{\min}$ and
$(S_\lambda\backslash S_r / S_\mu)_{\max}$,
such that 
\begin{equation}
d_A \leq d_B \qquad\Leftrightarrow\qquad d_A^+ \leq d_B^+
\end{equation}
if $d_A$ and $d_B$ are minimal length double coset representatives and $d_A^+$ and $d_B^+$ are the corresponding maximal ones (this coincidence is proved in \cite{HS}).
Using the bijections between these sets, we 
transport the Bruhat order to partial
orders
on $\Row(\lambda,\mu)$ and $\Mat{\lambda}{\mu}$. 
The resulting partial order on $\Mat{\lambda}{\mu}$ is given explicitly in terms of matrices by
\begin{equation}\label{bruhat}
A \leq B\Leftrightarrow 
\left(\sum_{i=1}^s \sum_{j=1}^t a_{i,j}
\geq 
\sum_{i=1}^s \sum_{j=1}^t b_{i,j}\text{ for all }s \in
\nset{\ell(\lambda)},
t \in \nset{\ell(\mu)}\right).
\end{equation}
One finds this elementary combinatorial observation
in many places in the literature, e.g., see \cite{BLM} which also
explains the geometric origin of this ordering.
%The Bruhat order on $\Row(\lambda,\mu)$ is given by $S \leq T$ if the shape of $S[\leq k]$ (the row tableau obtained from $S$ by removing all entries bigger than $k$) dominates the shape of  $T[\leq k]$ for all $k = 1,\dots,\ell(\lambda)$; this is noted in \cite[$\S$2]{Bdual} (where $\leq$ denotes the opposite of the order here).

%% file: s3-manin.tex
\section{The quantized coordinate algebra}\label{s4-manin}

The ring $\Z[q,q^{-1}]$ has a bar involution $-$ which sends $q$ to $q^{-1}$. 
We will use the term
``anti-linear map'' for a $\Z$-module homomorphism between
$\Z[q,q^{-1}]$-modules which intertwines $q$ and $q^{-1}$ in this way.
For $\Z[q,q^{-1}]$-modules, $V \otimes W$ means tensor product over $\Z[q,q^{-1}]$
and $V^*$ denotes $\Hom_{\Z[q,q^{-1}]}(V, \Z[q,q^{-1}])$.
We will need the quantum integer 
\begin{equation}\label{qint}
[n]_q := \frac{q^n - q^{-n}}{q-q^{-1}}
\end{equation}
and 
quantum binomial coefficient 
\begin{equation}\label{qbin}
\qbinom{n}{s}_{\!q} := \displaystyle\frac{[n]_q [n-1]_q \cdots [n-s+1]_q}{
[s]_q [s-1]_q \cdots [1]_q},
\end{equation}
which we interpret as zero in case $s < 0$.
These satisfy the Pascal-type recurrence relation:
\begin{equation}\label{pascal}
\qbinom{n}{s}_{\!q} = 
q^s \qbinom{n-1}{s}_{\!q} + q^{s-n} \qbinom{n-1}{s-1}_{\!q}
= q^{-s} \qbinom{n-1}{s}_{\!q} + q^{n-s} \qbinom{n-1}{s-1}_{\!q}.
\end{equation}
The following play the role of the binomial theorem for positive
and negative exponents:
\begin{align}\label{qbinomial}
\prod_{s=1}^n \left(1+q^{2s-n-1} x\right) &=
\sum_{s=0}^n \qbinom{n}{s}_{\!q} x^s,
&
\prod_{s=1}^n \frac{1}{\left(1+q^{2s-n-1} x\right)} &=
\sum_{s=0}^n \qbinom{-n}{s}_{\!q} x^i.
\end{align}
Here are some more identities that will be needed later.

\begin{lemma}\label{vanishing}
For $n \geq 0$, we have that
$\displaystyle\sum_{s=0}^n (-1)^s q^{s(n-1)} \qbinom{n}{s}_{\!q}= \delta_{n,0}.$
\end{lemma}

\begin{proof}
Set $x = -q^{-n-1}$ in the first identity from \cref{qbinomial}.
\end{proof}

\begin{lemma}\label{second}
For $m,n \in \Z$ and $s \geq 0$, we have that
$\displaystyle
\sum_{a+b=s} q^{mb-na}\qbinom{m}{a}_{\!q}\qbinom{n}{b}_{\!q} = \qbinom{m+n}{s}_{\!q}.
$
\end{lemma}

\begin{proof}
This is proved by a standard argument 
using \cref{qbinomial}.
See also \cite[Prop.~4.1(5)]{fiebig} (where this is called the Chu-Vandermonde convolution formula).
\end{proof}

\iffalse
\begin{proof}
If $m,n \geq 0$, one computes $t^c$-coefficients on both sides of the
identity
$$
\prod_{a=1}^m \left(1+q^{m+1-2a} (q^n t)\right) 
\prod_{b=1}^n \left(1+q^{n+1-2b} (q^{-m} t)\right)
= \prod_{c=1}^{m+n} (1 + q^{m+n+1-2c} t)
$$
using the first form of \cref{qbinomial} twice on the left hand side
and once on the right hand side. 
If $m +n \geq 0$ and $n < 0$, one computes $t^c$-coefficients on both
sides of
$$
\prod_{a=1}^m \left(1+q^{m+1-2a} (q^{n} t)\right) 
\Big / \prod_{b=1}^{-n} \left(1+q^{n+1-2b} (q^{-m} t)\right)
= \prod_{c=1}^{m+n} (1 + q^{m+n+1-2c} t)
$$
using the second form of \cref{qbinomial} once on the left hand side
instead.
The cases $m \geq 0, m+n < 0$ and $m,n < 0$ are similar.
\end{proof}
\fi

\begin{lemma}\label{A}
For $m \in \Z$ and  $s \geq 0$, we have that
$\displaystyle
\sum_{a+b=s} (-q)^{-b}
\qbinom{m+a}{a}_{\!q} \qbinom{m}{b}_{\!q}
= q^{ms}.
$
\end{lemma}

\begin{proof}
This is the $q$-analog of \cite[Lem.~A.1]{BEEO}.
See \cite[Lem.~3.1(3)]{BK} for its proof.
\end{proof}

\iffalse
\begin{proof}
Let $f(m,n) := \sum_{r+s=n} (-q)^s
\qbinom{m+r}{r,s}$ for short.
Using \cref{pascal}, one deduces the following 
recurrence relation for trinomial coefficients:
$$
\qbinom{n}{r,s} = q^{r-s}\qbinom{n-1}{r,s} + q^{r-n}\qbinom{n-1}{r-1,s} +
q^{n-s}\qbinom{n-1}{r,s-1}.
$$
In turn, this implies the following recurrence relation for $f(m,n)$:
$$
f(m,n) = q^n \overline{f(m-1,n)} + q^{-m} f(m,n-1) - q^{m+n-1}
\overline{f(m-1,n-1)}.
$$
Now the lemma follows by induction
starting from the identities $f(m,0)=f(0,n)=1$, which are
obvious from the definition.
\end{proof}
\fi

Let $\A_q(n)$ be Manin's quantized coordinate algebra of $n \times n$
matrices \cite{Manin}, which is the  $\Z[q,q^{-1}]$-algebra on generators
$\{x_{i,j}\:|\:1 \leq i,j \leq n\}$ subject to the relations
\begin{align}\label{r1}
x_{i,j} x_{k,l} &= 
\begin{cases}
x_{k,l} x_{i,j}&\text{if $i < k$ and $j > l$},\\
x_{k,l} x_{i,j} - (q-q^{-1}) x_{i,l} x_{k,j}&\text{if $i > k$ and $j >l$},\\
q^{-1} x_{k,l} x_{i,j}&\text{if $i=k, j > l$,}\\
q x_{k,l} x_{i,j}&\text{if $i<k, j = l$.}
\end{cases}
\end{align}
We view $\A_q(n)$ as a bialgebra with comultiplication $\Delta:\A_q(n) \rightarrow
\A_q(n) \otimes \A_q(n)$ and counit $\eps:\A_q(n) \rightarrow
\Z[q,q^{-1}]$
defined by 
\begin{align}
\Delta(x_{i,k}) &=
\sum_{j=1}^n x_{i,j} \otimes x_{j,k},
&
\eps(x_{i,j}) &= \delta_{i,j}.
\end{align}

\begin{lemma}\label{fact}
In $\A_q(2)$, we have for $a,b\geq 0$ that
$$
x_{2,2}^a x_{1,1}^b = \sum_{s=0}^{\min(a,b)} q^{-s(s-1)/2} (q^{-1}-q)^s [s]^!_q
\qbinom{a}{s}_{\!q}\qbinom{b}{s}_{\!q}
x_{2,1}^s x_{1,1}^{b-s} x_{2,2}^{a-s} x_{1,2}^s.
$$
\end{lemma}

\begin{proof}
Use induction on $a$ to check that
$x_{2,2}^a x_{1,1} = x_{1,1} x_{2,2}^a - (q-q^{-1}) [a] x_{2,1} x_{2,2}^{a-1}
x_{1,2}$.
This treats the case $b=1$. Then proceed by induction on $b$ using \cref{pascal}.
\end{proof}

\begin{lemma}\label{Assad}
In $\A_q(2)$, we have for $a \geq 0$ and $i,j \in\{1,2\}$ that
$$
\Delta(x_{i,j}^a) = \sum_{s=0}^a \qbinom{a}{s}_{\!q}
x_{i,1}^s x_{i,2}^{a-s} \otimes x_{2,j}^{a-s} x_{1,j}^s.
$$
\end{lemma}

\begin{proof}
Exercise.
\end{proof}

The character group of the $n$-dimensional torus 
consisting of diagonal matrices in $GL_n$
is naturally identified
with the
Abelian group $\Z^n$,
with standard coordinates $\eps_1,\dots,\eps_n$.
There is a scalar product on
$\Z^n$ such that $\eps_i \cdot \eps_j = \delta_{i,j}$.
We also have the {\em dominance order} on $\Z^n$  defined by $\lambda \leq \mu$
if the difference $\mu - \lambda$ is a sum of simple roots
$\alpha_i := \eps_i - \eps_{i+1}$ for $i=1,\dots,n-1$.

The algebra $\A_q(n)$
admits two different gradings.
It is $\Z$-graded with $x_{i,j}$ in degree one, and it is
bigraded by the character group $\Z^n$
with $x_{i,j}$ of bidegree $(\eps_i,\eps_j)$:
\begin{equation}\label{fool}
\A_q(n) = \bigoplus_{r \geq 0} \A_q(n,r)=
\bigoplus_{\lambda,\mu \in \Z^n} \A_q[\lambda,\mu].
\end{equation}
These two gradings are compatible with each other:
\begin{equation}\label{fooler}
\A_q(n,r) = \bigoplus_{\lambda,\mu \in \Comp(n,r)} \A_q[\lambda,\mu],
\end{equation}
where $\Comp(n,r):=\{\lambda
\vDash r\:|\:\ell(\lambda)=n\}$ is the set of all $\lambda \in \Z^n$
such that $\lambda_1,\dots,\lambda_n \geq 0$ and
$\lambda_1+\cdots+\lambda_n = r$.
It is also important to observe that $\A_q(n,r)$
is a subcoalgebra of $\A_q(n)$.

Let 
\begin{equation}
\I(n,r) := 
\big\{\bi = (i_1,\dots,i_r)\in \Z^r\:\big|\:1 \leq i_1,\dots,i_r \leq
n\big\}=
\bigcup_{\lambda \in \Comp(n,r)}
\I_\lambda.
\end{equation}
For $\bi, \bj \in \I(n,r)$, we use the shorthand
$x_{\bi,\bj} := x_{i_1,j_1} \cdots x_{i_r,j_r}$.
Then $\A_q(n,r)$ is free as a $\Z[q,q^{-1}]$-module
with the following basis,
which we call the {\em normally-ordered monomial basis}:
\begin{equation}\label{costandard}
\left\{
x_{\bi,\bj}\:\big|\:
\bi,\bj \in \I(n,r),
j_1 \leq \cdots \leq j_r\text{ and } i_s \geq i_{s+1}\text{ when }j_s=j_{s+1}
\right\}.
\end{equation}
There are several different proofs of this, e.g., in \cite[$\S$6]{Bdual} it
is derived from another realization of $\A_q(n)$ as a braided tensor
product of quantum symmetric algebras; normally-ordered here corresponds to the ``terminal double indexes'' in \cite{Bdual}.
Another relevant basis
is
\begin{equation}\label{standard}
\left\{
x_{\bi,\bj}\:\big|\:
\bi,\bj \in \I(n,r),
i_1 \geq \cdots \geq i_r\text{ and } j_s \leq j_{s+1}\text{ when }i_s=i_{s+1}
\right\}.
\end{equation}
This is the monomial basis in \cite{Bdual} indexed by ``initial double
indexes''.

Following \cite[Theorem 16]{Bdual}, 
the {\em bar involution} on $\A_q(n)$ is the anti-linear map
\begin{equation}
-:\A_q(n) \rightarrow \A_q(n)
\end{equation}
which fixes all of the generators $x_{i,j}$ and satisfies
\begin{equation}
\overline{x y} = q^{\lambda\cdot \mu - \lambda'\cdot \mu'} \overline{y}\: \overline{x}
\end{equation}
for $x$ of bidegree $(\lambda, \lambda')$ and $y$ of bidegree
$(\mu, \mu')$.
It is indeed an involution. 
Moreover:

\begin{lemma}\label{inv}
The bar involution is an anti-linear coalgebra automorphism.
\end{lemma}

\begin{proof}
Let $\overline{\Delta}$
denote the composition $- \otimes - \circ \Delta$.
We must show that $\overline{\Delta}(x) = \Delta(\overline{x})$ for any $x \in \A_q(n)$.
This follows by induction on degree.
\end{proof}

For $\lambda,\mu \in \Comp(n,r)$,
recall the set $\Mat{\lambda}{\mu}$ 
of matrices with these
row and column sums from \cref{s2-combinatorics},
which parametrizes the
orbits $\Pi_A$ 
of $S_r$ on $I_\lambda \times I_\mu$.
For $A \in \Mat{\lambda}{\mu}$, let
\begin{equation}\label{shoo}
x_A := x_{\bi,\bj} \text{ for $(\bi,\bj) \in \Pi_A$ such that
$j_1 \leq \cdots \leq j_d$ and $i_k \geq i_{k+1}$ when $j_k =
j_{k+1}$}.
\end{equation}
In other words, if $A$ corresponds to $P \in \Row(\lambda,\mu)$ under \cref{Amap} then
$\bi = \ZigzagDownRev(P)$ and $\bj = \bi^\mu$; the notation $\ZigzagDownRev(P)$ means the sequence obtained by reading the entries of $P$ in the order suggested by the arrow.
Hence $\bi = \bi^\lambda \cdot d(A)w_0$ where
$w_0$ is the longest element of $S_r$.
The set 
$\{x_A\:|\:\lambda,\mu \in \Comp(n,r), A \in \Mat{\lambda}{\mu}\}$ is the normally-ordered
monomial basis of
$\A_q(n,r)$
from \cref{costandard}, we have merely parametrized it in a more
convenient way.
By \cite[Theorem 16]{Bdual} again, the image of the normally-ordered
monomial $x_A$ under the bar involution is
\begin{equation}\label{shoe}
\overline{x}_A 
:= x_{\bi,\bj}\text{ for $(\bi,\bj)\in\Pi_A$ such that 
$i_1 \geq \cdots \geq i_r$ and $j_k \leq j_{k+1}$ when
$i_k = i_{k+1}$.}
\end{equation}
In other words, if $A^\tr$ corresponds to $Q \in \Row(\mu,\lambda)$ under \cref{Amap} then
$\bi = \bi^\lambda \cdot w_r$ and $\bj = \ZigzagUp(Q)$.
The set $\{\overline{x}_A\:|\:\lambda,\mu\in\Comp(n,r),A \in \Mat{\lambda}{\mu}\}$ 
is the basis for $\A_q(n,r)$ from \cref{standard}.

Recall the partial order \cref{bruhat} on $\Mat{\lambda}{\mu}$.
The bar involution acts on the normally-ordered monomial basis
in a unitriangular fashion:
$$
\overline{x}_A = x_A + 
\left(\text{a $\Z[q,q^{-1}]$-linear combination of $x_B$'s for $B > A$}\right).
$$
This may be seen explicitly 
by using the relations 
\cref{r1} to rewrite \cref{shoe} in terms of normally-ordered
monomials.
So one can apply Lusztig's Lemma to define another basis for
$\A_q[\lambda,\mu]$,
the {\em dual canonical basis} $\{b_A\:|\:A \in
\Mat{\lambda}{\mu}\}$.
The dual canonical basis element $b_A$ is the unique bar-invariant
vector in $\A_q[\lambda,\mu]$ such that
$b_A \equiv x_A \pmod{\sum_{B \in \Mat{\lambda}{\mu}} q \Z[q] x_B}$.
The dual canonical basis is discussed further in \cite{Bdual} (and many
other places).
In particular, the polynomials $p_{A,B}(q) \in \Z[q]$ defined from
\begin{equation}\label{count}
x_B = \sum_{A \in \Mat{\lambda}{\mu}} p_{A,B}(q) b_A
\end{equation}
are (renormalized)
Kazhdan-Lusztig polynomials: writing $P_{x,y}(t) \in \Z[t]$ for 
the usual 
Kazhdan-Lusztig polynomial associated to $x,y \in S_r$, we have that
\begin{equation}\label{countier}
p_{A,B}(q) = q^{\ell(d_A^+) - \ell(d_B^+)} P_{d_A^+,d_B^+}(q^{-2}).
\end{equation}
This is explained in \cite[Rem.~10]{Bdual}.
We have that $p_{A,B}(q) = 0$ unless $A \geq 
B$, $p_{A,A}(q) = 1$, and
$p_{A,B}(q) \in q \N[q]$ if $A > B$. The last assertion, which follows from positivity of Kazhdan-Lusztig polynomials, will not be needed here.

\begin{lemma}\label{business}
Suppose we are given $A', B' \in \Mat{\lambda'}{\mu'}$ 
for $\lambda',\mu' \in \Comp(n,r)$ 
and $1 \leq
i,j \leq n$ 
such that $\lambda_i' = \mu_j' = 0$.
Let $A$ and $B$ be the matrices obtained from $A'$ and $B'$ by
removing the $i$th row and $j$th column.
Then $p_{A,B}(q) = p_{A',B'}(q)$.
\end{lemma}

\begin{proof}
This is clear from the nature of the defining relations
\cref{r1} for $\A_q(n)$: they only depend on the relative positions
of the indices in the total order on the
set $\nset{n}$, not on the actual values.
\end{proof}

\begin{example}\label{rank2} 
For $\lambda,\mu \in \Comp(2,r)$ and $A \in \Mat{\lambda}{\mu}$,
we have that
\begin{equation*}
b_A = x_{2,1}^{a_{2,1}} x_{1,1}^{a_{1,1}-\min(a_{1,1}, a_{2,2})}
\left(x_{1,1}x_{2,2}-q x_{2,1}x_{1,2}\right)^{\min(a_{1,1}, a_{2,2})}
x_{2,2}^{a_{2,2}-\min(a_{1,1}, a_{2,2})} x_{1,2}^{a_{1,2}}.
\end{equation*}
This follows from a special case
of \cite[Theorem 20]{Bdual}, which gives a closed formula for the dual canonical basis element $b_A$ for all $A \in \Mat{\lambda}{\mu}$ 
providing either $\lambda$ or $\mu$ has
at most two non-zero parts.
Expanding the binomial gives
\begin{equation*}
b_A = x_A - q^{M}\qbinom{m}{1}_{\!q}x_{A+B}
+ q^{2(M-1)} \qbinom{m}{2}_{\!q} x_{A+2B}
- \cdots + (-1)^m q^{m(M+1-m)}\qbinom{m}{m}_{\!q} x_{A+mB}
\end{equation*}
where $m := \min(a_{1,1},a_{2,2}), M := \max(a_{1,1},a_{2,2})$
and $B := \left[\begin{smallmatrix}-1&1\\1&-1\end{smallmatrix}\right]$.
\end{example}

\begin{lemma}\label{trump}
%Let $\A_q(m)\otimes \A_q(n)$ be the  tensor product of the bialgebras $\A_q(m)$ and $\A_q(n)$. Then 
There is a surjective bialgebra
homomorphism
$$
\Y^*:\A_q(m+n) \twoheadrightarrow \A_q(m)\otimes \A_q(n),
\quad
x_{i,j}\mapsto \begin{cases}
x_{i,j}\otimes 1&\text{if $1 \leq i,j \leq m$,}\\
1\otimes 
x_{i-m,j-m}&\text{if $m+1 \leq i,j \leq m+n$,}\\
0&\text{otherwise.}
\end{cases}
$$
Moreover, this intertwines the bar involution on $\A_q(m+n)$ with the
bar involution $- \otimes -$ on
$\A_q(m)\otimes \A_q(n)$.
\end{lemma}

\begin{proof}
The existence of this algebra homomorphism follows from the
relations. Then one checks that it is a coalgebra homomorphism too.
Finally, for the statement about the bar involution, note 
for an $m \times m$ matrix $A$ and an $n \times n$ matrix $B$ that
$\Y^*$
sends $x_{\diag(A,B)}$ to $x_A \otimes x_B$
and $\overline{x}_{\diag(A,B)}$ to $\overline{x}_A \otimes
\overline{x}_B$.
\end{proof}

There is also an anti-linear algebra anti-automorphism
\begin{equation}
\barT^*:\A_q(n) \rightarrow \A_q(n),
\qquad
x_{i,j} \mapsto x_{j,i}.
\end{equation}
This is a coalgebra anti-automorphism, i.e.,
$\barT^* \otimes \barT^* \circ \Delta = \P \circ \Delta \circ \barT^*$
where $\P$ is the tensor flip.
Comparing \cref{shoo} and \cref{shoe}, we see that
$\barT^*(x_A) = \overline{x}_{A^\T}$
where $A^\T$ is the transpose matrix. Since $\barT^*$ is an
involution, it follows that it commutes with the bar involution. Let
\begin{equation}\label{Transposed}
\T^* := - \circ \barT^* = \barT^* \circ - :\A_q(n)\rightarrow \A_q(n).
\end{equation}
This is a linear coalgebra anti-automorphism (but {\em not} an algebra anti-automorphism)
which commutes with the bar involution and sends $x_A$ to $x_{A^\T}$.
It follows that
\begin{equation}\label{dracula}
\T^*(b_A) = b_{A^\T}.
\end{equation}

The dual canonical basis element $b_A$ indexed by $A =
I_n$, which is minimal in the Bruhat order, is the {\em quantum determinant}
\begin{equation}\label{qd}
\qdet := \sum_{w \in S_n} (-q)^{\ell(w)} x_{w(1),1} \cdots
x_{w(n),n}.
\end{equation}
This is central in $\A_q(n)$.
It is also a group-like element, 
i.e., $\Delta(\qdet) = \qdet \otimes \qdet$
and $\eps(\qdet)=1$.
The coordinate algebra of the {\em quantum general linear group} $\qG_n$
is the Ore localization of $\A_q(n)$ at the quantum determinant.
The bialgebra structure on
$\A_q(n)$ extends to make this into a Hopf algebra.
We will not work explicitly with this Hopf algebra here, but
its existence 
underpins all subsequent language and notation.

By a {\em polynomial representation of 
$\qG_n$} we mean a right $\A_q(n)$-comodule.
We use the notation 
$\Hom_{\qG_n}(-,-)$ to denote morphisms 
in the category of polynomial representations.
Since $\A_q(n)$ is a bialgebra, this is a monoidal category.
For example, we have the {\em natural representation} of $\qG_n$, which is 
the free $\Z[q,q^{-1}]$-module $V$
with basis
$v_1,\dots,v_n$ and comodule structure map
$\eta:V \rightarrow V \otimes \A_q(n,1)$ 
defined from
\begin{equation}
\eta(v_j) = \sum_{i=1}^n v_i \otimes x_{i,j}.
\end{equation}
It is a polynomial representation of degree 1, hence, its
$r$th tensor power $V^{\otimes r}$ is a 
polynomial representation of degree $r$, meaning that it is a right $\A_q(n,r)$-comodule.

The category of polynomial representations of $\qG_n$ is also braided, with 
braiding $c$ that is uniquely determined
by requiring that
$\R_{V,V} \in \End_{\qG_n}(V\otimes V)$
is the $\Z[q,q^{-1}]$-linear map defined by \cref{rmat}.
We have that $(\R_{V,V}+q)(\R_{V,V}-q^{-1}) = 0$, hence, $\R_{V,V}$ has eigenvalues
$-q$ and $q^{-1}$. After localizing at
$[2]=q+q^{-1}$, the tensor square 
$V \otimes V$ decomposes as the direct sum of the corresponding eigenspaces.
The $q^{-1}$-eigenspace is spanned by 
\begin{equation}\label{qeigen}
\{v_j \otimes v_i +
q v_i \otimes v_j \:|\:1 \leq i < j \leq n\}\cup \{v_k \otimes
v_k\:|\:1 \leq k \leq n\}.
\end{equation}
The {\em quantum exterior algebra}
\begin{equation}
{\textstyle\bigwedge}(V)
=\bigoplus_{r\geq 0} {\textstyle\bigwedge^r} V
\end{equation}
is the quotient
of the tensor algebra $T(V)$ by the two-sided ideal generated
by the quadratic tensors from \cref{qeigen}. 
This is studied in \cite{PW} (see also
\cite[$\S$5]{Bdual}),
where it is proved that $\bigwedge^r V$
is free as a $\Z[q,q^{-1}]$-module with basis 
$$
\big\{v_I := v_{i_1}\wedge\cdots\wedge v_{i_r}\:\big|\:I = \{i_1
< \cdots < i_r\}\subseteq\nset{n}\big\}.
$$
The
comodule structure map $\eta$ for 
$\bigwedge^r V$
satisfies
$\eta(v_J) = \sum_I v_I \otimes x_{I,J}$
where
\begin{equation}
x_{I,J} := 
\sum_{w \in S_r} (-q)^{\ell(w)}
x_{i_{w(1)}, j_1} \cdots x_{i_{w(r)}, j_r}
\end{equation}
for $I = \{i_1 < \cdots < i_r\}$ and $J = \{j_1 < \cdots< j_r\}$.
These so-called {\em quantum minors} include the quantum determinant \cref{qd} as a special case.

%% file: s4-schuralgebra.tex
\section{The \texorpdfstring{$q$}{q}-Schur algebra}\label{s4-schuralgebra}

We continue to work over $\Z[q,q^{-1}]$ like in the previous section.
The {\em $q$-Schur algebra} is the $\Z[q,q^{-1}]$-linear dual
\begin{equation}
S_q(n,r) := \A_q(n,r)^* = \bigoplus_{\lambda,\mu \in
  \Comp(n,r)}
\A_q[\lambda,\mu]^*.
\end{equation}
It is an algebra with multiplication 
$S_q(n,r)\otimes S_q(n,r)
\rightarrow S_q(n,r)$ defined by the dual map to 
the restriction 
$\A_q(n,r) \rightarrow \A_q(n,r) \otimes \A_q(n,r)$ of the comultiplication on $\A_q(n)$.
For this, we are identifying $f \otimes g \in S_q(n,r) \otimes S_q(n,r)$
 with an element of $(\A_q(n,r) \otimes \A_q(n,r))^*$ so that 
$\langle f \otimes g, x \otimes y \rangle := \langle f, x \rangle
\langle g, y \rangle$
for $f,g \in S_q(n,r), x, y \in \A_q(n,r)$.

The unit element $1 \in S_q(n,r)$ is the restriction of the counit
$\eps$ to $\A_q(n,r)$. 
For $\lambda \in \Comp(n,r)$, let $1_\lambda$ be the
function which is equal to $\eps$ on $\A_q[\lambda,\lambda]$ and is
zero on all other summands $\A_q[\lambda,\mu]$ in the decomposition
\cref{fooler}.
This defines mutually orthogonal idempotents $\{1_\lambda\:|\:\lambda
\in \Comp(n,r)\}$ in $S_q(n,r)$ whose sum is the
identity. Moreover, 
$1_\lambda S_q(n,r) 1_\mu = \A_q[\lambda,\mu]^*$.

The dual map to the bar involution on $\A_q(n,r)$ 
defines a bar involution on $S_q(n,r)$ which we denote with the same
notation, so
$\langle \overline{f}, x \rangle = \overline{\langle f,
  \overline{x}\rangle}$
for $f \in S_q(n,r), x \in \A_q(n,r)$.
\cref{inv} implies that $-:S_q(n,r)\rightarrow S_q(n,r)$  is an anti-linear algebra automorphism.
The dual of the restriction
$\A_q(m+n,r) \rightarrow \bigoplus_{a+b=r} \A_q(m,a)\otimes \A_q(n,b)$ of the
homomorphism $\Y^*$ from \cref{trump} 
defines an injective 
algebra homomorphism
\begin{align}\label{guard}
\Y_r:\bigoplus_{a+b=r}
S_q(m,a) \otimes S_q(n,b)
&\hookrightarrow
S_q(m+n,r),&
\xi_A \otimes \xi_B &\mapsto \xi_{\diag(A,B)}.
\end{align}
This intertwines the bar involutions $- \otimes -$ on each
$S_q(m,a)\otimes 
S_q(n,b)$ with the bar involution on $S_q(m+n,r)$.
The dual of \cref{Transposed} gives us a
transposition involution $\T:S_q(n,r)\rightarrow S_q(n,r)$.
This is a linear algebra anti-automorphism.

The dual bases to $\{x_A\:|\:A \in \Mat{\lambda}{\mu}\}$ and
$\{b_A\:|\:A \in \Mat{\lambda}{\mu}\}$ give bases for 
$1_\lambda S_q(n,r) 1_\mu$ denoted
$\{\xi_A\:|\:A \in \Mat{\lambda}{\mu}\}$, the {\em standard basis}, and
$\{\theta_A\:|\:A \in \Mat{\lambda}{\mu}\}$,
the {\em canonical basis}.
The canonical basis
element $\theta_A\in 1_\lambda
S_q(n,r) 1_\mu$ is the unique bar-invariant element
such that
$\theta_A \equiv \xi_A\pmod{\sum_{B \in \Mat{\lambda}{\mu}} q \Z[q] \xi_B}$.
In fact, we have that $\theta_A = \xi_A+$(a $q \N[q]$-linear combination of $\xi_B$
for $B < A$), because by \cref{count}
we have that
\begin{equation}
\label{weird}
\theta_A = \sum_{B \in \Mat{\lambda}{\mu}}
p_{A,B}(q) \xi_B,
\end{equation} 
where $p_{A,B}(q)$ is the Kazhdan-Lusztig
polynomial from \cref{countier}.
There is also a geometric construction of the canonical basis via
intersection cohomology.
This is explained in \cite[$\S$1.4]{BLM},
where the standard basis element $\xi_A$ is denoted $[A]$ and
$\theta_A$ is denoted $\{A\}$ (up to some renormalization).

The counit $\eps$ is zero on all of the normally-ordered 
monomials in $\A_q[\lambda,\lambda]$ except for 
$x_{1,1}^{\lambda_1} \cdots x_{n,n}^{\lambda_n}$, proving the first
equality in
\begin{equation}
1_\lambda = \xi_{\diag(\lambda_1,\dots,\lambda_n)} = \theta_{\diag(\lambda_1,\dots,\lambda_n)}.
\end{equation}
The second equality follows because 
$\overline{\xi}_A = \xi_A + ($a $\Z[q,q^{-1}]$-linear combination of $\xi_B$'s for $B < A)$ and
$A=\diag(\lambda_1,\dots,\lambda_n)$ is minimal in the Bruhat
ordering,
so 
$\xi_{\diag(\lambda_1,\dots,\lambda_n)}$
is bar invariant.
More generally, since the homomorphism $\Y_r$
is bar equivariant, we have that
\begin{equation}\label{sleepydog}
\Y_r(\theta_A \otimes \theta_B) = \theta_{\diag(A,B)}.
\end{equation}
Also, by \cref{dracula}, we have that 
\begin{align}\label{cabbage}
\T(\xi_A) &= \xi_{A^\T},&
\T(\theta_A) &= \theta_{A^\T}.
\end{align}

\begin{example}\label{rank2canonical}
For $A \in \Mat{\lambda}{\mu}$ with $\lambda,\mu \in\Comp(2,r)$ we
have that
\begin{equation}\label{pogs}
\theta_A = \sum_{s=0}^{\min(a_{1,2}, a_{2,1})}
q^{s(s+\max(a_{1,1},a_{2,2}))} 
\qbinom{s+\min(a_{1,1},a_{2,2})}{s}_q
 \xi_{A-s B}
\end{equation}
where $B := \left[\begin{smallmatrix}-1&1\\1&-1\end{smallmatrix}\right]$.
This follows by inverting the transition matrix from \cref{rank2}.
\end{example}

For $n \times n$ matrices $A,B,C$ with non-negative integer entries,
define
\begin{equation}\label{schurs}
Z(A,B,C) := \langle \xi_A \otimes \xi_B, \Delta(x_C)\rangle \in \Z[q,q^{-1}].
\end{equation}
These are the {\em structure constants} for multiplication in the standard
basis of the $q$-Schur
algebra: 
we have that
\begin{equation}\label{qschurs}
\xi_A \circ \xi_B := 
\sum_{C} Z(A,B,C)
\xi_C.
\end{equation}
This formula can be viewed as a $q$-analog of Schur's product rule.
For a completely different approach to the definition of these
structure constants (counting points over a finite field), see \cite[$\S$1.1]{BLM}.
The structure constants have following stabilization property, which
will be relevant in the next
section.

\begin{lemma}\label{stability}
Suppose we are given $A'\in \Mat{\lambda'}{\mu'}, B' \in
\Mat{\mu'}{\nu'}$
and $C' \in \Mat{\lambda'}{\nu'}$
for $\lambda',\mu',\nu' \in \Comp(n,r)$
and $1 \leq i,j,k \leq n$ such that $\lambda'_i = \mu'_j =
\nu'_k = 0$.
Let $A,B,C$ be the matrices obtained
by removing the $i$th row and $j$th column of $A'$,
the $j$th row and $k$th column of $B'$, and the $i$th row and $k$th
column of $C'$, respectively.
Then we have that
$Z(A,B,C) = Z(A',B',C')$.
\end{lemma}

\begin{proof}
This follows for the same reason as \cref{business}.
\end{proof}

Let $H_r$ be the Hecke algebra of the symmetric
group, that is, the $\Z[q,q^{-1}]$-algebra on generators
$\tau_1,\dots,\tau_{r-1}$
subject to the relations
\begin{equation}
(\tau_i+q)(\tau_i-q^{-1}) = 0,
\qquad
\tau_i\tau_j = \tau_j \tau_i\text{ if $|i-j|> 1$},
\qquad
\tau_i \tau_{i+1} \tau_i = \tau_{i+1}\tau_i \tau_{i+1}.
\end{equation}
For $w \in S_r$, we have the corresponding element $\tau_w \in H_r$
defined from a reduced expression for $w$, and the
elements $\{\tau_w\:|\:w \in S_r\}$ give a basis for $H_w$ as a free
$\Z[q,q^{-1}]$-module.
Recall also that the Hecke algebra has its own anti-linear bar involution $-:H_w\rightarrow H_w,
\tau_w \mapsto \tau_{w^{-1}}^{-1}$.

\begin{lemma}\label{schurfunctor}
Suppose that $r \leq n$ and let $\omega := (1^r\  0^{n-r}) \in
\Comp(n,r)$.
There is an algebra isomorphism
$H_r \stackrel{\sim}{\rightarrow} 1_\omega S_q(n,r) 1_\omega$
sending $\tau_w$ to the standard basis element $\xi_A$ for the
matrix $A \in \Mat{\omega}{\omega}$ such that $a_{w(i), i} =
1$ for $i=1,\dots,r$ and all other entries are zero.
This map intertwines the bar involutions on $H_r$ and $S_q(n,r)$.
\end{lemma}

\begin{proof}
Check that 
the relation 
$$
\tau_w \tau_{i} = 
\begin{cases}
\tau_{w s_i}&\text{if $w(i) < w(i+1)$}\\
\tau_{w s_i} - (q-q^{-1}) \tau_w&\text{if $w(i) > w(i+1)$}
\end{cases}
$$
holds in $S_q(n,r)$
by explicitly calculating the corresponding structure constants.
This is well known so we omit the details.
\end{proof}

Let $V$ be the natural representation of $\qG_n$.
In addition to our definition of $S_q(n,r)$
by dualizing $\A_q(n,r)$, and the approach in \cite{BLM} where the
$q$-Schur algebra arises as the endomorphism algebra of a permutation
representation of the finite general linear group, the $q$-Schur
algebra can be realized as an endomorphism algebra for an action of
the Hecke algebra $H_r$ on the tensor space $V^{\otimes r}$. 
To explain this, note that $V^{\otimes r}$ has basis
$v_\bi := v_{i_1}\otimes\cdots\otimes v_{i_r}$ for $\bi \in \I(n,r)$.
There is a {\em right} action of $H_r$ on $V^{\otimes r}$
such that $\tau_i$ acts as the braiding $1^{\otimes (i-1)} \otimes \R_{V,V}
\otimes 1^{r-i-1}$ from \cref{rmat}.
Since $V^{\otimes r}$ is a polynomial representation of degree $r$,
it is a left $S_q(n,r)$-module.
The action of $H_r$ commutes with the action of $S_q(n,r)$.
Hence, there is a well-defined algebra homomorphism
\begin{equation}\label{altdef}
S_q(n,r) \rightarrow 
\End_{H_r}\left(V^{\otimes r}\right).
\end{equation}
This homomorphism is actually an algebra {\em isomorphism}. There are several ways to see this, e.g., it can be deduced from
\cite{DJ}.
In fact, in \cite{DJ}, the authors work with a different realization
of the right $H_r$-module $V^{\otimes r}$ as a direct sum of
permutation modules. In this form, one obtains a basis for the endomorphism algebra on the
right hand side of \cref{altdef} quite easily from the Mackey
theorem, and then just need to check that this basis is also the image of the
standard basis for $S_q(n,r)$ under the homomorphism \cref{altdef}.
Since this is quite important for us, we go through some details in
the next paragraph.

For $\lambda \in \Comp(n,r)$, 
let $H_\lambda$ be the parabolic subalgebra of $H_r$ associated to $S_\lambda$.
Let $X_\lambda$ be the
the free $\Z[q,q^{-1}]$-module of rank one with basis $m_\lambda$ viewed as a right $H_\lambda$-module so that $m_\lambda  \tau_i = q^{-1} m_\lambda$ for each $\tau_i \in H_\lambda$.
The (right) {\em permutation module} is the induced module $M(\lambda) := X_\lambda \otimes_{H_\lambda} H_r$.
There is a unique $H_r$-module homomorphism
\begin{equation}\label{reptheory}
f_\lambda:M(\lambda) \rightarrow 1_\lambda
V^{\otimes r},\qquad
m_\lambda \otimes 1 \mapsto v_{\bi^\lambda}.
\end{equation} 
This is actually an {\em isomorphism} because the vectors $\left\{m_\lambda \otimes \tau_w\:\big|\:w \in
  (S_\lambda\backslash S_r)_{\min}\right\}$ give a basis for $M(\lambda)$, and $f_\lambda$ maps
them to the basis 
 $\{v_\bi\:|\:\bi \in \I_\lambda\}$
for $1_\lambda V^{\otimes r}$.
Summing over all $\lambda \in \Lambda(n,r)$, this gives us an $H_r$-module isomorphism
\begin{equation}\label{fnr}
f:
\bigoplus_{\lambda \in \Comp(n,r)} M(\lambda)
\stackrel{\sim}{\rightarrow} 
V^{\otimes r}.
\end{equation}
The following lemma explains how to transport the natural action of $S_q(n,r)$ on $V^{\otimes r}$
through $f$ to obtain an action on this direct sum of permutation modules.

\begin{lemma}\label{extralargebone}
Suppose that $\lambda,\mu \in \Comp(n,r)$ and $A \in \Mat{\lambda}{\mu}$.
The diagram 
$$
\begin{tikzcd}
1_\mu V^{\otimes r}\arrow[r,"\xi_A"]& 1_\lambda V^{\otimes r}\\
M(\mu)\arrow[r]\arrow[u,"f_\mu" left]& M(\lambda)\arrow[u,"f_\lambda" right]
\end{tikzcd}
$$
commutes, 
where the top map is defined by acting on the left
with
$\xi_A$,
and the bottom map
is the $H_r$-module 
homomorphism sending
\begin{equation}\label{hardplace}
m_\mu \otimes 1 
\mapsto \sum_{w \in (S_{\mu^+} \backslash S_\mu)_{\min}} q^{\ell(w_0)-\ell(w)} m_\lambda \otimes \tau_{d_A} \tau_w,
\end{equation}
where
$\mu^+ \vDash r$ is as in \cref{doso}
and $w_0$ is the longest element of 
$(S_{\mu^+} \backslash S_\mu)_{\min}$.
\end{lemma}

\begin{proof}
The comodule structure map $\eta$ of $V^{\otimes r}$
satisfies $\eta(v_\bj) = \sum_{\bi \in \I(n,r)} v_\bi \otimes x_{\bi,\bj}$.
Hence, for $\bj \in\I_\mu$, we have that
\begin{equation}\label{sss}
\xi_A v_\bj = \sum_{\bi \in \I_\lambda} \langle \xi_A, x_{\bi,\bj}
\rangle v_\bi.
\end{equation}
By the definition
\cref{shoo},
we have that
$x_A = x_{\bi^\lambda \cdot d_A w_0, \bi^\mu}$.
The $S_\mu$-orbit of $\bi^\lambda\cdot d_A w_0$ is
$\{\bi^\lambda \cdot d_A w\:|\:w \in 
(S_{\mu^+} \backslash S_\mu)_{\min}\}$.
Also for $w \in 
(S_{\mu^+} \backslash S_\mu)_{\min}$ we have that
$x_{\bi^\lambda \cdot d_A w, \bi^\mu} = 
q^{\ell(w_0)-\ell(w)} x_{\bi^\lambda\cdot d_A w_0, \bi^\mu}$ as one needs
  to use the last relation in \cref{r1} a total of
  $\ell(w_0)-\ell(w)$ times.
Putting this together shows that
$$
\xi_A v_{\bi^\mu} =
\sum_{w \in (S_{\mu^+} \backslash S_\mu)_{\min}} q^{\ell(w_0)-\ell(w)}
v_{\bi^\lambda \cdot d_A w}.
$$
The lemma now follows
since
$f_\lambda$ sends 
$m_\lambda \otimes 1$ to $v_{\bi^\lambda}$, $f_\mu$ sends $m_\mu \otimes 1$
to $v_{\bi^\mu}$, and
$v_{\bi^\lambda \cdot d_A w} = v_{\bi^\lambda} \tau_{d_A} \tau_w$ as $i^\lambda_1 \leq
\cdots\leq i_r^\lambda$.
\end{proof}

Let $m$ be another natural number.
For $\lambda \in \Comp(m,r)$, let 
$Y(\lambda)$ be the free $\Z[q,q^{-1}]$-module of rank one with basis $n_\lambda$
viewed as a left $H_\lambda$-module so that
$\tau_i n_\lambda = -q n_\lambda$ for each $\tau_i \in H_\lambda$. The (left) signed permutation module is the induced module $N(\lambda) := H_r
\otimes_{H_\lambda} Y(\lambda)$.

\begin{lemma}\label{schurf}
There is an algebra isomorphism
$S_q(m,r)
\stackrel{\sim}{\rightarrow} \End_{H_r}\left(\bigoplus_{\lambda \in \Comp(m,r)}
  N(\lambda)\right)$
sending $\xi_A \in 1_\lambda S_q(m,r) 1_\mu$
to the unique $H_r$-module homomorphism such that
$$
1 \otimes n_\mu \mapsto 
\sum_{w \in (S_{\mu^+}\backslash S_\mu)_{\min}}
 (-1)^{\ell(w)+\ell(d_A)} q^{\ell(w_0)-\ell(w)}
\tau_w^{-1} \tau_{d_A}^{-1} \otimes n_\lambda
$$
where $\mu^+$ is as in \cref{doso}
and $w_0$ is the longest element of $(S_{\mu^+}\backslash S_\mu)_{\min}$,
and $1 \otimes n_\nu \mapsto 0$ for $\nu \neq \mu$.
\end{lemma}

\begin{proof}
We start from the algebra isomorphism
\cref{altdef}. Using 
\cref{fnr} and \cref{extralargebone}, 
and replacing $n$ by $m$, this gives us
an algebra isomorphism 
$S_q(m,r) \stackrel{\sim}{\rightarrow}
\End_{H_r}\left(\bigoplus_{\lambda \in \Lambda(m,r)} M(\lambda)\right)$ such that $\xi_A \in 1_\lambda S_q(m,r) 1_\mu$
acts on $m_\mu \otimes 1 \in M(\mu)$ according to
\cref{hardplace}, 
and it acts as zero on all other summands.
Then we use the
algebra anti-automorphism
$H_r \rightarrow H_r, \tau_x\mapsto(-1)^{\ell(x)}
\tau_x^{-1}$.
The pull-back of the right $H_r$-module $M(\lambda)$ along this map is isomorphic
the left $H_r$-module $N(\lambda)$, there being a unique isomorphism such that
 $m_\lambda \otimes \tau_x \mapsto (-1)^{\ell(x)} \tau_x^{-1} \otimes n_\lambda$ for all $x \in S_r$. 
 We deduce that
 $\End_{H_r}\left(\bigoplus_{\lambda \in \Lambda(m,r)} M(\lambda)\right)
\cong
 \End_{H_r}\left(\bigoplus_{\lambda \in \Lambda(m,r)} N(\lambda)\right)$.
It just remains to note that the action of $\xi_A \in 1_\lambda S(m,r)1_\mu $ on
$\bigoplus_{\lambda \in \Lambda(m,r)} M(\lambda)$
translates into the action on
$\bigoplus_{\lambda \in \Lambda(m,r)} N(\lambda)$
described explicitly 
in the statement of the lemma.
\end{proof}

The goal now is to replace $H_r$ 
and the signed permutation modules $N(\lambda)$
in
\cref{schurf}
with the quantum general linear group $\qG_n$
and its polynomial representations
\begin{equation}\label{bigbones}
{\textstyle\bigwedge^\lambda V} := {\textstyle\bigwedge^{\lambda_1}
  V}
\otimes\cdots\otimes {\textstyle\bigwedge^{\lambda_{\ell(\lambda)}} V}.
\end{equation}

\begin{lemma}\label{haircut}
Take $\lambda,\mu \in \Lambda(m,r)$ and $A \in \Mat{\lambda}{\mu}$.
There is a unique $\qG_n$-module homomorphism
$\phi_A:\bigwedge^\mu V \rightarrow \bigwedge^\lambda V$ such that
the diagram 
$$
\begin{tikzcd}
V^{\otimes r}\arrow[d]\arrow[r]&V^{\otimes r}\arrow[d]\\
\bigwedge^\mu V\arrow[r,"\phi_A"]&\bigwedge^\lambda V
\end{tikzcd}
$$
commutes, where the vertical maps are the
quotient maps and the
top map 
is right multiplication by 
$\sum_{w \in (S_{\mu^+}\backslash S_\mu)_{\min}}
(-1)^{\ell(w)+\ell(d_A)} q^{\ell(w_0)-\ell(w)}
\tau_w^{-1} \tau_{d_A}^{-1}$
where $\mu^+$ is defined as in \cref{doso}
and $w_0$ is the longest element of $(S_{\mu^+} \backslash S_\mu)_{\min}$.
\end{lemma}

\begin{proof}
By the definition of quantum exterior powers,
the kernel of the projection $V^{\otimes r} \twoheadrightarrow
\bigwedge^\mu V$ is generated by the kernels of the endomorphisms
$\tau_j-q^{-1} = \tau_j^{-1} - q$ for all $j$ with $s_j \in S_\mu$.
So we need to show for such a $j$ and $v \in V^{\otimes r}$ with
$v\tau_j^{-1} = q v$
that the vector 
$$
v' := \sum_{w \in  (S_{\mu^+}\backslash S_\mu)_{\min}}
(-1)^{\ell(w)+\ell(d_A)} q^{\ell(w_0)-\ell(w)} v \tau_w^{-1} \tau_{d_A}^{-1}
$$ 
is in the sum of the kernels of the maps $\tau_i^{-1}-q$ for all $i$ with
$s_i \in S_\lambda$.
We have that
$$
(S_{\mu^+}\backslash S_\mu)_{\min}=
X \sqcup X s_j
\sqcup Y
$$ such that 
$\ell(x s_j) = \ell(x)+1$ for all $x \in X$, and 
$y s_j y^{-1} \in S_{\mu^+}$ for all $y \in Y$.
This follows from \cite[Lemma 1.1]{DJgl}.
For $x \in X$, we have that
$$
(-1)^{\ell(x)+\ell(d_A)} q^{\ell(w_0)-\ell(x)} v \tau_x^{-1}
\tau_{d_A}^{-1}
+
(-1)^{\ell(x s_j)+\ell(d_A)} q^{\ell(w_0)-\ell(x s_j)} v \tau^{-1}_j \tau_x^{-1}
\tau_{d_A}^{-1} = 0
$$
as $v \tau_j^{-1} = q v$.
This implies that
$$
v' = \sum_{y \in Y} 
(-1)^{\ell(y)+\ell(d_A)} q^{\ell(w_0)-\ell(y)} v \tau_y^{-1}
\tau_{d_A}^{-1}.
$$
It remains to show for $y \in Y$ that $v \tau_y^{-1} \tau_{d_A}^{-1}$ is in the kernel
of 
$\tau_i^{-1}-q$ for some $i$ with
$s_i \in S_\lambda$.
We have that $d_A y s_j = t d_A y$
for $t := d_A (y s_j y^{-1}) d_A^{-1}$.
Since $y s_j y^{-1} \in S_{\mu^+}$, 
we deduce using \cref{doso} that $t \in S_\lambda$ (in fact, it is in $S_{\lambda^-} \leq S_\lambda$ in the notation from the lemma), and that
$\ell(t d_A y) =
\ell(t)+\ell(d_A)+\ell(y)$. Since $\ell(d_A y s_j) \leq
\ell(d_A)+\ell(y)+1$,
we deduce that $\ell(t) = 1$. Hence, $t = s_i$ for some $i$ such that $s_i \in S_\lambda$.
Moreover
$v \tau^{-1}_y \tau_{d_A}^{-1}\tau_i^{-1} = v \tau_j^{-1}
\tau_y^{-1} \tau_{d_A}^{-1} 
= q v \tau_y^{-1}\tau_{d_A}^{-1}$.
\end{proof}

The following theorem is the quantum analog of \cite[Proposition 3.11]{D1}. See also
\cite[4.2(19)]{D2} for a closely
related result already in the quantum setting.

\begin{theorem}\label{met}
Fix $m,r \in\N$.
For any $n \geq 0$,
there is a surjective algebra homomorphism
\begin{equation}\label{donkiniso}
g_{n}:S_q(m,r) \twoheadrightarrow \End_{\qG_n}\left(\bigoplus_{\lambda \in \Comp(m,r)}
{\textstyle\bigwedge^\lambda} V\right)
\end{equation}
sending 
$\xi_A \in 1_\lambda S_q(m,r) 1_\mu$
to the endomorphism that is equal to the homomorphism
$\phi_A$ from \cref{haircut} 
on the summand
$\bigwedge^\mu V$,
and is zero on all other summands.
Moreover, $g_n$ is an isomorphism if $n\geq r$.
\end{theorem}

\begin{proof}
Using the base change functor $\k \otimes_{\Z[q,q^{-1}]}-$, it
suffices to prove the analogous statement
when $\Z[q,q^{-1}]$ is replaced by a field $\k$ 
and $q$ is any non-zero element. In the remainder of the proof, we assume we are
working over a field in this way, writing $\qG_n(\k)$ for the quantum general linear group over $\k$, whose coordinate algebra is $\k \otimes_{\Z[q,q^{-1}]} \A_q(n)$.
The category of polynomial
representations of $\qG_n(\k)$ is a highest weight category
satisfying standard homological properties. This is justified e.g. in
\cite{PW} or \cite{D2}\footnote{It can also be deduced by using the results of \cref{s7-basis} to show that $S_q(n,r)$ is a split quasi-hereditary algebra.}.
In the next paragraph, we treat the case that $n \geq r$. Then the existence and
surjectivity
of $g_{n}$ for $n < r$ follows from the existence and
surjectivity of $g_{N}$
for $N \geq r$
by an argument involving truncation to the subgroup $\qG_n <
\qG_N$ using \cite[4.2(11)]{D2} (this requires the standard homological properties).

So now assume that $n \geq r$ and that we are working over a field. 
We must show that $g_{n}$ is a
well-defined algebra isomorphism.
To see this, we use the 
{\em Schur functor}, that is, the 
idempotent truncation functor $\pi:
S_q(n,r)\mod \rightarrow H_r\mod$ defined by the idempotent $1_\omega$, notation as
in \cref{schurfunctor}. This sends an
$S_q(n,r)$-module to its $\omega$-weight space viewed as an
$H_r$-module via the isomorphism from that lemma.
We have that $\pi \left(\bigwedge^\lambda V\right) \cong
N(\lambda)$,
there being a unique such isomorphism sending
the canonical image of $v_1 \otimes \cdots \otimes v_r$
in $\bigwedge^\lambda V$ to $1 \otimes n_\lambda$.
Moreover, the Schur functor induces an isomorphism
$$
\Hom_{S_q(n,r)}(\textstyle\bigwedge^\mu V, \bigwedge^\lambda V)
\stackrel{\sim}{\rightarrow}
\Hom_{H_r}(N(\mu), N(\lambda)).
$$
This follows by general principles (e.g., see \cite[Th.~2.12]{JS}) because
the head of $\bigwedge^\mu V$ and the socle of $\bigwedge^\lambda
V$ are $p$-restricted, i.e., they only involve irreducible
modules $L$ which are not annihilated by $\pi$. Indeed,
these modules are both submodules and quotient modules of the tensor
space $V^{\otimes r}$, which has $p$-restricted head and socle because
$V^{\otimes r} \cong S_q(n,r) 1_\omega$
by the isomorphisms \cref{altdef,reptheory}, hence, 
$$
\Hom_{S_q(n,r)}(L, V^{\otimes r})\cong
\Hom_{S_q(n,r)}(V^{\otimes r},L)
\cong \Hom_{S_q(n,r)}(S_q(n,r) 1_\omega,
L) \cong 1_\omega L
$$
for any self-dual module $L$.
Consequently, $\pi$ induces an algebra isomorphism
$$
\End_{\qG_n}\left(\bigoplus_{\lambda \in \Comp(m,r)} \textstyle \bigwedge^\lambda V\right)
\cong
\End_{H_r}\left(\bigoplus_{\lambda \in \Comp(m,r)} N(\lambda)\right).
$$
Composing this with the isomorphism from \cref{schurf}
gives the
desired isomorphism $g_{n}$.

It just remains to identify the endomorphism $g_{n}(\xi_A)$ 
with $\phi_A$.
For this, it suffices to check for $\xi_A \in 1_\lambda S_q(m,r) 1_\mu$
that the maps 
$g_{n}(\xi_A)$ and $\phi_A$  are equal on the canonical
image of $v_1\otimes\cdots\otimes v_r$ in $\bigwedge^\mu V$.
By the definition from \cref{haircut}, $\phi_A$ sends this vector to 
the canonical image of 
$$
\sum_{w \in (S_{\mu^+}\backslash S_\mu)_{\min}} (-1)^{\ell(x)+\ell(d_A)} q^{\ell(w_0)-\ell(w)}
(v_1 \otimes \cdots \otimes v_r) \tau_{w}^{-1}
\tau_{d_A}^{-1}
$$
in $\bigwedge^\lambda V$.
On the other hand, $g_{n}(\xi_A)$ takes
this vector to the image of
$$
\sum_{w \in (S_{\mu^+} \backslash S_\mu)_{\min}}
(-1)^{\ell(x)+\ell(d_A)} q^{\ell(w_0)-\ell(w)}
\tau_{w}
\tau_{d_A}^{-1} (v_1 \otimes \cdots \otimes v_r) 
$$
where the left action of $H_r$ on $1_\omega V^{\otimes r}$ comes from
  the left action of $S_q(n,r)$ via the isomorphism of \cref{schurfunctor}.
Now observe for any $x \in S_r$ that $\tau_x (v_1 \otimes\cdots\otimes v_r) = (v_1
\otimes \cdots \otimes v_r) \tau_x$ as, by the definitions, both
vectors are equal to $v_{x(1)}\otimes\cdots\otimes v_{x(r)}$.
\end{proof}

%% file: s5-schurcategory.tex
\section{The \texorpdfstring{$q$}{q}-Schur category}\label{s5-schurcategory}

It is easy to adapt \cref{schurs} to define 
$Z(A,B,C) \in \Z[q,q^{-1}]$ for 
$A \in \Mat{\lambda}{\mu}$, $B \in \Mat{\mu}{\nu}$, $C \in
\Mat{\lambda}{\nu}$ and
compositions 
$\lambda,\mu,\nu \vDash r$ that are not necessarily of the same length. To do so,
we pick any 
$n \geq \ell(\lambda),\ell(\mu),\ell(\nu)$ and
let $\lambda',\mu'$ and $\nu'$ be compositions of length $n$
obtained from $\lambda,\mu$ and $\nu$ by adding some extra entries
equal to zero. Let $A' \in \Mat{\lambda'}{\mu'}$, $B' \in
\Mat{\mu'}{\nu'}$ and $C' \in \Mat{\lambda'}{\nu'}$ be the matrices
obtained 
by inserting corresponding rows and columns of
zeros into $A, B$ and $C$;
see \cref{filler1,filler2} for an example.
Then we define $Z(A,B,C)$ to be the structure constant
$Z(A', B', C')$ for the $q$-Schur algebra $S_q(n,r)$ 
exactly as defined earlier. The stability from \cref{stability} implies
that this is well-defined independent of all choices.

The following theorem defines the {\em $q$-Schur category with $0$-strings}. The version without $0$-strings discussed in the introduction is the full subcategory with object set $\sComp \subset \Comp$.

\begin{theorem}\label{bigdef}
There is a $\Z[q,q^{-1}]$-linear 
strict monoidal category $\qSchur$
with
\begin{itemize}
\item objects that are all compositions $\lambda
\in \Comp$;
\item 
for $\lambda\vDash r$ and $\mu \vDash s$,
the morphism space $\Hom_{\qSchur}(\mu,\lambda)$ is $\{0\}$ unless
$r=s$, and it is 
the free $\Z[q,q^{-1}]$-module
with basis $\{\xi_A\:|\:A \in \Mat{\lambda}{\mu}\}$ if $r=s$;
\item
tensor product of objects is defined by concatenation of compositions;
\item
tensor product 
of morphisms (horizontal composition) 
is defined by $\xi_A \star \xi_B :=
\xi_{\diag(A,B)}$;
\item
vertical composition of morphisms is defined 
as in \cref{qschurs}.
\end{itemize}
The strict identity object $\mathbbm{1}$ is the composition of length zero, and
the identity endomorphism $1_\lambda$ of an object
$\lambda \in \Comp$ is $\xi_{\diag(\lambda_1,\dots,\lambda_{\ell(\lambda)})}$.
\end{theorem}

\begin{proof}
Most of the axioms of strict monoidal category are straightforward to check.
The fact that vertical
composition is associative is a consequence of associativity of
multiplication in the $q$-Schur algebra.
To check the interchange law, we must show that
$$
(\xi_A \star 1_\sigma) \circ (1_\mu \star \xi_B) =
(1_\lambda \star \xi_B) \circ (\xi_A \star 1_\rho)
$$
for
$\lambda,\mu\vDash a,\sigma,\rho\vDash b$
and $A \in \Mat{\lambda}{\mu}, B \in \Mat{\sigma}{\rho}$, that is,
$$
\xi_{\diag(A, \sigma_1,\dots,\sigma_{\ell(\sigma)})}
\circ \xi_{\diag(\mu_1,\dots,\mu_{\ell(\mu)},B)}
= \xi_{\diag(\lambda_1,\dots,\lambda_{\ell(\lambda)}, B)} \circ \xi_{\diag(A,\rho_1,\dots,\rho_{\ell(\rho)})}.
$$
Using the stability from \cref{stability}, we may assume that $\ell(\lambda)=\ell(\mu)=m$ and
$\ell(\sigma)=\ell(\rho)=n$.
We have that
$(\xi_A \otimes 1_\sigma) (1_\mu \otimes \xi_B)
= (1_\lambda \otimes \xi_B) (A \otimes 1_\rho)$
in the algebra
$S_q(m,a) \otimes S_q(n,b)$.
Now apply the algebra homomorphism $\Y_{a+b}$ from \cref{guard}.
\end{proof}

\begin{remark}
It is clear from the definition that
the path algebra of the full subcategory of $\qSchur$ generated by objects in $\Lambda(n,r)$
may be identified with the $q$-Schur algebra,
that is, 
\begin{equation}
S_q(n,r) = \bigoplus_{\lambda,\mu \in \Lambda(n,r)}
\Hom_{\qSchur}(\mu,\lambda).
\end{equation}
By \cref{altdef,extralargebone}, we have that 
$1_\lambda S_q(n,r) 1_\mu \cong \Hom_{H_r}(M(\mu),M(\lambda))$ for $\lambda,\mu \in \Lambda(n,r)$.
It follows that the full subcategory
of $\qSchur$
generated by objects in $\Lambda(r) := \bigcup_{n \geq 0} \Lambda(n,r)$
is isomorphic to the category $\qSchur(r)$ with 
object set $\Lambda(r)$ and morphism spaces
\begin{equation}\label{anotherway}
\Hom_{\qSchur(r)}(\mu,\lambda) := \Hom_{H_r}(M(\mu),M(\lambda)),
\end{equation}
with the natural composition law.
The categories $\qSchur(r)$ for all $r$ can then be assembled to obtain an alternative approach to the definition of $\qSchur$, with 
tensor product arising from
the bifunctors
$\qSchur(r)\times \qSchur(s) \rightarrow \qSchur(r+s)$
induced by the natural embeddings
$H_r \otimes H_s \hookrightarrow H_{r+s}$.
We have emphasized the based approach in \cref{bigdef} since it allows composition of standard basis elements to be computed effectively 
using the coalgebra structure on $\A_q(n)$. This will be used several times later on.
\end{remark}

We have defined the morphism space
$\Hom_{\qSchur}(\mu,\lambda)$ so that it comes equipped with 
the standard basis $\{\xi_A\:|\:A \in \Mat{\lambda}{\mu}\}$.
We can also transfer the canonical basis
from the $q$-Schur algebra to $\qSchur$, as follows. 
Take any $\lambda,\mu \vDash r$ and 
$A, B \in
\Mat{\lambda}{\mu}$.
There is a corresponding Kazhdan-Lusztig polynomial $p_{A,B}(q) \in \Z[q]$.
To define this, 
we again
pick any $n \geq \ell(\lambda),\ell(\mu)$, add extra
zeros to $\lambda$ and $\mu$ to make them into compositions of the
same length
$n$,
and add corresponding rows and columns of zeros to $A$ and $B$ to
obtain
$A', B' \in \Mat{\lambda'}{\mu'}$.
Then we let $p_{A,B}(q) := p_{A',B'}(q)$, where the latter polynomial
comes from \cref{weird}.
This is well defined independent of the choices thanks to \cref{business}. It is also clear from the proof of that lemma
that the slightly more general polynomials $p_{A,B}(q)$ still satisfy \cref{countier}.
Let
\begin{equation}
\theta_A := \sum_{B \in \Mat{\lambda}{\mu}} p_{A,B}(q)\xi_B,
\end{equation}
thereby defining the canonical basis $\{\theta_A\:|\:A \in
\Mat{\lambda}{\mu}\}$
for $\Hom_{\qSchur}(\mu,\lambda)$.

\begin{lemma}\label{bio}
There is an anti-linear strict monoidal functor
$-:\qSchur \rightarrow \qSchur$
which is the identity on objects and, on 
the
morphism space
$\Hom_{\qSchur}(\mu,\lambda)$,
is the unique anti-linear map which fixes the canonical basis
$\{\theta_A\:|\:A \in \Mat{\lambda}{\mu}\}$.
\end{lemma}

\begin{proof}
Since the bar involution for $q$-Schur algebras is an anti-linear 
algebra automorphism, this prescription gives a well-defined anti-linear functor.
To see that it is strict monoidal,
it suffices to observe that
$\theta_A \star \theta_B = \theta_{\diag(A,B)}$.
This follows from \cref{sleepydog}.
\end{proof}

%The canonical basis for $\Hom_{\qSchur}(\mu,\lambda)$ can also be characterized as before as the unique bar-invariant basis such that $\theta_A \equiv \xi_A\pmod{\sum_B q \Z[q] \xi_B}$.
Similarly, we upgrade the involution $\T$ on $S_q(n,r)$ to a
strict linear monoidal functor
\begin{equation}\label{starinv}
\T:\qSchur \rightarrow (\qSchur)^{\operatorname{op}}
\end{equation}
which is the identity on objects, commutes with the bar involution, and sends $\xi_A\mapsto
\xi_{A^\T}$, $\theta_A \mapsto \theta_{A^\T}$. This follows
by \cref{cabbage}.

\begin{theorem}\label{dthm}
There is a full $\Z[q,q^{-1}]$-linear monoidal functor 
$\Sigma_n$
from $\qSchur$ to the category of polynomial representations of $\qG_n$
sending 
the object $\lambda \vDash d$ to the polynomial representation
$\bigwedge^\lambda V$ of degree $r$ from \cref{bigbones},
and the morphism $\xi_A$ for 
$\lambda,\mu \vDash r$ and $A \in \Mat{\lambda}{\mu}$
to the homomorphism 
$\phi_A:\bigwedge^\mu V \rightarrow\bigwedge^\lambda V$ from
\cref{haircut}.
\end{theorem}

\begin{proof}
To see that $\Sigma_n$ is a functor, we must show that
$\Sigma_n(\xi_A \circ \xi_B) = \Sigma_n(\xi_A) \circ \Sigma_n(\xi_B)$
for $A \in \Mat{\lambda}{\mu}$, $B \in
  \Mat{\mu}{\nu}$,
$\lambda,\mu,\nu \vDash r$ and $r \geq 0$.
In view of the definition of vertical composition in $\qSchur$, this follows because we have that
$\phi_A \circ \phi_B = \sum_{C \in \Mat{\lambda}{\nu}} Z(A,B,C) \phi_C$ by
\cref{met}, taking $m \geq \ell(\lambda), \ell(\mu), \ell(\nu)$.
The same theorem also shows that $\Sigma_n$ is full.
Finally, to see that $\Sigma_n$ is a monoidal functor,
we need to check that
$\phi_A \otimes \phi_B = \phi_{\diag(A,B)}$.
This is clear from the explicit description of these maps given by
\cref{haircut}.
\end{proof}

\begin{remark}\label{ginandtonic}
  (1)
  Using the final statement from \cref{met}, 
 the proof of \cref{dthm} also shows that
the functor $\Sigma_n$ defines an {\em isomorphism}
$\Hom_{\qSchur}(\mu,\lambda)
\stackrel{\sim}{\rightarrow}
\Hom_{\qG_n}({\textstyle\bigwedge^\mu
  V},{\textstyle\bigwedge^\lambda V})$
providing $n \geq |\lambda|,|\mu|$.
So one could say that $\Sigma_n$ is {\em
  asymptotically faithful} as $n \rightarrow \infty$.
In \cref{webn} below, we will give an explicit description of the
kernel of $\Sigma_n$, that is, the tensor ideal of $\qSchur$ consisting of the
morphisms that it annihilates.

\vspace{1mm}
\noindent  
  (2) 
  Let $\k$ be a field viewed as a $\Z[q,q^{-1}]$-algebra in some way, and
consider the specialization $\qSchur(\k) := \k \otimes_{\Z[q,q^{-1}]} \qSchur$. The functor $\Sigma_n$ in \cref{dthm} induces a
$\k$-linear monoidal functor from $\qSchur(\k)$ to the category of polynomial representations of $\qG_n(\k)$. By the proofs of \cref{dthm,met}, this induced functor is also full.
\end{remark}

By {\em merges}, {\em splits}, and {\em positive
  crossings}, 
we mean the morphisms
$\xi_{\left[\begin{smallmatrix}a&b\end{smallmatrix}\right]}$,
$\xi_{\left[\begin{smallmatrix}a\\b\end{smallmatrix}\right]}$,
and
$\xi_{\left[\begin{smallmatrix}0&b\\a&0\end{smallmatrix}\right]}$
for $a,b \geq 0$.
The images 
$\phi_{\left[\begin{smallmatrix}a&b\end{smallmatrix}\right]}$,
$\phi_{\left[\begin{smallmatrix}a\\b\end{smallmatrix}\right]}$,
and
$\phi_{\left[\begin{smallmatrix}0&b\\a&0\end{smallmatrix}\right]}$
of these special morphisms under the functor $\Sigma_n$
from \cref{dthm} are
the natural projection
\begin{equation}
{\textstyle\bigwedge^a V \otimes \bigwedge^b V}
\twoheadrightarrow {\textstyle\bigwedge^{a+b} V},
\qquad
v\otimes w \mapsto v \wedge w,
\end{equation}
the inclusion
\begin{align}
{\textstyle\bigwedge^{a+b} V}
&\hookrightarrow 
{\textstyle\bigwedge^a V \otimes \bigwedge^b V},\\\notag
v_{i_1}\wedge\cdots\wedge v_{i_{a+b}}
&\mapsto
q^{-ab} \!\!\!\!\!\!\!\!\sum_{w \in (S_{a+b}/ S_a \times S_b)_{\min}}\!\!\!\!\!\!\!\!
(-q)^{\ell(w)} 
v_{i_{w(1)}}\wedge\cdots\wedge v_{i_{w(a)}}\otimes
v_{i_{w(a+1)}}\wedge\cdots\wedge v_{i_{w(a+b)}},
\end{align}
and the isomorphism
\begin{equation}\label{why}
(-1)^{ab} \R_{\bigwedge^b V, \bigwedge^a V}^{-1}:
{\textstyle\bigwedge^a V \otimes \bigwedge^b V}
\stackrel{\sim}{\rightarrow}
{\textstyle\bigwedge^b V \otimes \bigwedge^a V}
\end{equation}
where $\R_{\bigwedge^b V, \bigwedge^a V}: \bigwedge^b V \otimes
\bigwedge^a V \rightarrow
\bigwedge^a V \otimes
\bigwedge^b V$ is the braiding
on the monoidal category of polynomial representations of $\qG_n$.
This follows from the explicit description of $\phi_A$ in
\cref{haircut}.
We refer to the morphisms
$\overline{\xi}_{\left[\begin{smallmatrix}0&b\\a&0\end{smallmatrix}\right]}$
for $a,b \geq 0$
as {\em negative crossings}.
The following lemma implies that the image
of
$\overline{\xi}_{\left[\begin{smallmatrix}0&b\\a&0\end{smallmatrix}\right]}$
under the functor $\Sigma_n$ is the isomorphism
\begin{equation}
(-1)^{ab} \R_{\bigwedge^a V, \bigwedge^b V}:
{\textstyle\bigwedge^a V \otimes \bigwedge^b V}
\stackrel{\sim}{\rightarrow}
{\textstyle\bigwedge^b V \otimes \bigwedge^a V}.
\end{equation}

\begin{lemma}\label{pigs}
For $a,b \geq 0$, we have that 
$\overline{\xi}_{\left[\begin{smallmatrix}0&b\\a&0\end{smallmatrix}\right]}
=
\xi_{\left[\begin{smallmatrix}0&a\\b&0\end{smallmatrix}\right]}^{-1}.$
Also the merge and split morphisms 
$\xi_{\left[\begin{smallmatrix}a&b\end{smallmatrix}\right]}$
and
$\xi_{\left[\begin{smallmatrix}a\\b\end{smallmatrix}\right]}$
are invariant under the bar involution, hence, they
coincide with the canonical basis elements
$\theta_{\left[\begin{smallmatrix}a&b\end{smallmatrix}\right]}$
and
$\theta_{\left[\begin{smallmatrix}a\\b\end{smallmatrix}\right]}$.
\end{lemma}

\begin{proof}
The part about merge and split morphisms is trivial as
these matrices are not comparable to any other in the Bruhat ordering.
For first part, we show that
$\overline{\xi}_{\left[\begin{smallmatrix}0&a\\b&0\end{smallmatrix}\right]}
\circ \xi_{\left[\begin{smallmatrix}0&b\\a&0\end{smallmatrix}\right]}
= 1_{(a,b)}$. 
This identity (with $a$ and $b$ switched) together with the image of this identity under the bar involution implies the result.
Take any $A \in \Mat{(a,b)}{(a,b)}$ and consider the coefficient of 
$\xi_A$ when the product 
$\overline{\xi}_{\left[\begin{smallmatrix}0&a\\b&0\end{smallmatrix}\right]}
\xi_{\left[\begin{smallmatrix}0&b\\a&0\end{smallmatrix}\right]} \in S_q(2,a+b)$
is expanded in terms of the standard basis.
Since multiplication in $S_q(2,a+b)$ is dual to comultiplication in $\A_q(2,a+b)$,
this coefficient is equal to the $x_{2,1}^b x_{1,2}^a \otimes x_{2,1}^a x_{1,2}^b$-coefficient
of $
\Delta\left(x_{2,1}^{a_{2,1}} x_{1,1}^{a_{1,1}} x_{2,2}^{a_{2,2}}
x_{1,2}^{a_{1,2}}\right)$ when expanded in terms of the basis
$\{\overline{x}_B \otimes x_C\:|\:B, C \in \Mat{(a,b),(a,b)}\}$.
By \cref{Assad},
\begin{multline*}
\Delta\left(x_{2,1}^{a_{2,1}} x_{1,1}^{a_{1,1}} x_{2,2}^{a_{2,2}}
x_{1,2}^{a_{1,2}}\right)
=\sum_{a_{2,1}'=0}^{a_{2,1}} 
\sum_{a_{1,1}'=0}^{a_{1,1}} 
\sum_{a_{2,2}'=0}^{a_{2,2}} 
\sum_{a_{1,2}'=0}^{a_{1,2}} 
\qbinom{a_{2,1}}{a_{2,1}'}_{\!q}
\qbinom{a_{1,1}}{a_{1,1}'}_{\!q}
\qbinom{a_{2,2}}{a_{2,2}'}_{\!q}
\qbinom{a_{1,2}}{a_{1,2}'}_{\!q} \times \\
x_{2,1}^{a_{2,1}'} x_{2,2}^{a_{2,1}-a_{2,1}'} x_{1,1}^{a_{1,1}'}
x_{1,2}^{a_{1,1}-a_{1,1}'}
x_{2,1}^{a_{2,2}'} x_{2,2}^{a_{2,2}-a_{2,2}'}
x_{1,1}^{a_{1,2}'} x_{1,2}^{a_{1,2}-a_{1,2}'}
\otimes \\
x_{2,1}^{a_{2,1}-a_{2,1}'} x_{1,1}^{a_{2,1}'}
x_{2,1}^{a_{1,1}-a_{1,1}'} x_{1,1}^{a_{1,1}'}
x_{2,2}^{a_{2,2}-a_{2,2}'} x_{1,2}^{a_{2,2}'}
x_{2,2}^{a_{1,2}-a_{1,2}'} x_{1,2}^{a_{1,2}'}.
\end{multline*}
To get $x_{2,1}^a x_{1,2}^b$ on straightening 
using \cref{r1}
into the normal order 
in the second tensor position, we must
have that $a_{2,1}'=a_{1,1}'=0$, $a_{1,2}'=a_{1,2}$ and
$a_{2,2}'=a_{2,2}$.
This term is
\begin{equation}\label{dyfed}
x_{2,2}^{a_{2,1}} 
x_{1,2}^{a_{1,1}}
x_{2,1}^{a_{2,2}}
x_{1,1}^{a_{1,2}}
\otimes
x_{2,1}^{b}
x_{1,2}^{a}=
q^{-a_{2,1}a_{2,2}-a_{1,1}a_{1,2}}
x_{2,1}^{a_{2,2}}
x_{2,2}^{a_{2,1}} 
x_{1,1}^{a_{1,2}}
 x_{1,2}^{a_{1,1}}
 \otimes
x_{2,1}^{b}
x_{1,2}^{a}.
\end{equation}
Because we are using the ordering from \cref{standard}
for monomials in the first tensor (rather than the normal ordering),
we only get a non-zero coefficient when
$A = \left[\begin{smallmatrix}a&0\\0&b\end{smallmatrix}\right]$,
when the coefficient is $1$. This shows the product is $1_{(a,b)}$.
\end{proof}

\begin{remark}\label{tolerate}
By a similar argument to the proof of \cref{pigs}, one can also prove
the following ``quadratic relation'':
$$
\xi_{\left[\begin{smallmatrix}0&a\\b&0\end{smallmatrix}\right]}\circ
\xi_{\left[\begin{smallmatrix}0&b\\a&0\end{smallmatrix}\right]}
=
\sum_{s=0}^{\min(a,b)}
q^{-s(s-1)/2-s(a+b-2s)}
(q^{-1}-q)^s 
[s]_q^! \ 
\xi_{\left[\begin{smallmatrix}a-s&s\\s&b-s\end{smallmatrix}\right]}.
$$
Indeed, from \cref{dyfed},
the coefficient of
$\xi_{\left[\begin{smallmatrix}a-s&s\\s&b-s\end{smallmatrix}\right]}$
in 
$\xi_{\left[\begin{smallmatrix}0&a\\b&0\end{smallmatrix}\right]}\circ \xi_{\left[\begin{smallmatrix}0&b\\a&0\end{smallmatrix}\right]}$
is $q^{-s(a-s)-s(b-s)}$ times the coefficient of $x_{2,1}^b x_{1,2}^a$
when $x_{2,1}^{b-s} x_{2,2}^s x_{1,1}^s x_{1,2}^{a-s}$ is expanded
in terms of the normally-ordered monomial basis.
The latter coefficient is $q^{-s(s-1)/2} (q^{-1}-q)^s [s]_q^!$ 
by \cref{fact}.
\end{remark}

More generally, a {\em merge of $n$ strings} is a morphism of the form
$\xi_{A}$
for a $1 \times n$ matrix $A$,
an
{\em split of $n$ strings} is a morphism of the form $\xi_A$
for an $n \times 1$ matrix $A$,
and a {\em positive permutation of $n$ strings}
is a morphism
of the form $\xi_A$ for an $n \times n$ matrix 
$A$ such that in each row and column there is at most one non-zero entry.
Positive permutations of $n$ strings can be parametrized instead by $w \in S_n$
and a composition
$\mu$ of length $n$, setting
\begin{equation}\label{genperm}
\tau_{w;\mu}:=\xi_A
\quad\text{where $A \in \Mat{\mu\cdot w^{-1}}{\mu}$
has $a_{w(1),1}=\mu_1,\dots,a_{w(n),n}=\mu_n$.}
%:= \xi_A \in \Hom_{\qSchur}(\mu,\mu\cdot w^{-1})
\end{equation}
%where $A$
%has $a_{w(1),1}=\mu_1,\dots,a_{w(n),n} = \mu_n$
%and all other entries are zero.
If $\mu \in \Lambda(n,r)$ then
the same formula defines an element of 
$1_{\mu\cdot w^{-1}} S_q(n,r) 1_\mu$;
for example, for $w \in S_r \leq S_n$, the image of $\tau_w \in H_r$ under the isomorphism of \cref{schurfunctor} is
$\tau_{w;\omega}$.
For $1 \leq i < n$, we have that
\begin{align}
\tau_{s_i;\mu} &= 1_{(\mu_1,\dots,\mu_{i-1})} \star
\xi_{\left[\begin{smallmatrix}0&\mu_{i+1}\\\mu_i
      &0\end{smallmatrix}\right]}
\star 1_{(\mu_{i+2},\dots,\mu_n)}.
\end{align}
So $\tau_{s_i;\mu}$, which we call a {\em simple permutation of $n$ strings},
is a positive crossing
tensored on the left and right with the appropriate
identity morphisms.
The following lemma implies that any positive permutation of $n$ strings
can be obtained by composing simple
permutations.

\begin{lemma}\label{chambers}
Suppose that $\mu\in \Lambda(n,r)$ and $w \in
S_n$
is a permutation such that $w(i) < w(i+1)$ for some $1 \leq i < n$. 
Then $\tau_{ws_i;\mu} = \tau_{w;\mu\cdot s_i} \circ \tau_{s_i;\mu}$.
\end{lemma}

\begin{proof}
It suffices to prove the analogous statement in the $q$-Schur algebra $S_q(n,r)$.
There is a left action of $S_n$ on $\I(n,r)$
by its action on entries.
This commutes with the right action of $S_r$.
We claim that the left action of $S_q(n,r)$ on $V^{\otimes r}$
satisfies
$\tau_{w;\mu} v_{\bi^\mu} = v_{w\cdot\bi^\mu}$.
To see this, the normally-ordered monomial in $\A_q(n,r)$
that is dual to the standard basis vector
$\tau_{w;\mu}$ is $x_{w\cdot\bi^\mu, \bi^\mu}$. 
Moreover, $w \cdot \bi^\mu$ is the only $\bi \in \I(n,r)$ such that
$x_{w\cdot\bi^\mu, \bi^\mu}$ appears in the normally-ordered monomial
basis expansion
of $x_{\bi, \bi^\mu}$.
So the claim follows from \cref{sss}.

To prove the lemma, it suffices to show that $\tau_{w;\mu\cdot s_i} \tau_{s_i;\mu}$
and $\tau_{w s_i;\mu}$ act in the same way on 
$v_{\bi^\mu}$.
The latter gives $v_{w s_i\cdot\bi^\mu}$ by the claim.% from the previous paragraph. Also
Also $\tau_{s_i;\mu} v_{\bi^\mu} = v_{s_i\cdot\bi^\mu}$.
So we are reduced to checking that
$
\tau_{w;\mu\cdot s_i} v_{s_i\cdot\bi^\mu} = v_{ws_i\cdot\bi^\mu}.
$
Let $d \in (S_{\mu\cdot s_i} \backslash S_r)_{\min}$ be the unique Grassmann permutation
such that $\bi^{\mu\cdot s_i} \cdot d = 
s_i\cdot\bi^\mu$. The action of
$H_r$ on $V^{\otimes r}$ was defined using
\cref{rmat}, from which we see that 
$v_{\bi^{\mu\cdot s_i}} \tau_d=
v_{\bi^{\mu\cdot s_i} \cdot d}$.
Similarly, because  $w(i)< w(i+1)$,
we get that 
$v_{w\cdot \bi^{\mu\cdot s_i})} \tau_d = v_{(w\cdot \bi^{\mu \cdot s_i}) \cdot d}$.
So
\begin{align*}
\tau_{w;\mu\cdot s_i} v_{s_i\cdot\bi^\mu}&=
\tau_{w;\mu\cdot s_i} v_{\bi^{\mu\cdot s_i}\cdot d}
=\tau_{w;\mu\cdot s_i} v_{\bi^{\mu\cdot s_i}} \tau_d
=
v_{w\cdot\bi^{\mu\cdot s_i}} \tau_d\\& = v_{(w\cdot\bi^{\mu \cdot s_i}) \cdot d}
= v_{w\cdot(\bi^{\mu \cdot s_i} \cdot d)}
=
v_{w\cdot (s_i\cdot \bi^\mu)}
= v_{ws_i\cdot\bi^\mu}.
\end{align*}

\vspace{-6mm}
\end{proof}

A special case of the next lemma implies that
\begin{align}\label{kitchen}
\xi_{\left[\begin{smallmatrix}a_1+\cdots+a_s&b_1+\cdots+b_t\end{smallmatrix}\right]}
\circ
\big(\xi_{\left[\begin{smallmatrix}a_1&\cdots&a_s\end{smallmatrix}\right]}\star \xi_{\left[\begin{smallmatrix}b_1&\cdots&b_t\end{smallmatrix}\right]}\big)
&=
\xi_{\left[\begin{smallmatrix}a_1&\cdots&a_s&b_1&\cdots&b_t\end{smallmatrix}\right]},\\
\big(\xi_{\left[\begin{smallmatrix}a_1&\cdots&a_s\end{smallmatrix}\right]^\T}\star\xi_{\left[\begin{smallmatrix}b_1&\cdots&b_t\end{smallmatrix}\right]^\T}
\big)\circ
\xi_{\left[\begin{smallmatrix}a_1+\cdots+a_s&b_1+\cdots+b_t\end{smallmatrix}\right]^\T}
&=\xi_{\left[\begin{smallmatrix}a_1&\cdots&a_s&b_1&\cdots&b_t\end{smallmatrix}\right]^\T},
\label{cabinet}\end{align}
for 
$a_1,\dots,a_s, b_1,\dots,b_t \geq 0$.
Hence, all merges/splits of $n$ strings can be expressed as compositions of tensor products of
merges/splits of 2 strings and appropriate identity morphisms.

\begin{lemma}\label{ged}
Suppose that $\lambda,\mu \vDash r$, 
 $A \in \Mat{\lambda}{\mu}$
and $1 \leq i \leq \ell(\lambda), 1 \leq j \leq \ell(\mu)$.
\begin{enumerate}
\item
Let $B$ be obtained from $A$ by replacing
its $i$th row
by two rows of length $\ell(\mu)$, the first of which has
entries $a_{i,1},\dots,a_{i,j},0,\dots,0$
with sum $\lambda_i'$,
and the second has entries
$0,\dots,0,a_{i,j+1},\dots,a_{i,\ell(\mu)}$
with sum $\lambda_i''$ (so $\lambda_i'+\lambda_i''=\lambda_i$).
Then we have that
\begin{equation*}
\xi_A = \big(1_{(\lambda_1,\dots,\lambda_{i-1})} \star \xi_{
\left[\begin{smallmatrix}\lambda_i'&\lambda_i''\end{smallmatrix}\right]} \star
1_{(\lambda_{i+1},\dots,\lambda_{\ell(\lambda)})}\big) \circ \xi_B.
\end{equation*}
\item
Let $B$ be obtained from $A$ by replacing
its $j$th column
by two columns of length $\ell(\lambda)$, the first of which has
entries $a_{1,j},\dots,a_{i,j},0,\dots,0$
with sum $\mu_j'$,
and the second has entries
$0,\dots,0,a_{i+1,j},\dots,a_{\ell(\lambda),j}$
with sum $\mu_j''$ (so $\mu_j'+\mu_j''=\mu_j$).
Then we have that
\begin{equation*}
\xi_{A}=
\xi_{B}\circ
\bigg(1_{(\mu_1,\dots,\mu_{j-1})} \star \xi_{\left[\begin{smallmatrix}\mu_j'\\\mu_j''\end{smallmatrix}\right]} \star
1_{(\mu_{j+1},\dots,\mu_{\ell(\mu)})}\bigg).
\end{equation*}
\end{enumerate}
\end{lemma}

\begin{proof}
We just prove (b). Then (a) follows on
applying $\T$.
By the way that composition in $\qSchur$ is defined, the statement we are trying to prove reduces to the following claim about multiplication in $S_q(n,r)$:

\vspace{1mm}
\noindent
{\em
Suppose that we are given $\lambda,\mu \in \Comp(n,r)$,
$A \in \Mat{\lambda}{\mu}$ and $1 \leq i \leq n, 1 \leq j \leq n-1$ with $\mu_{j+1}=0$.
Let $B\in\Mat{\lambda}{\mu'}$ be obtained from $A$ by replacing
the $j$th and $(j+1)$th columns with $[a_{i,1}\ \cdots\ a_{i,j}\ 0\ \cdots\ 0]^\tr$ and 
$[0\ \cdots\ 0\ a_{i+1,j}\ \cdots\ a_{n,j}]^\tr$,
respectively,
and $C\in\Mat{\mu'}{\mu}$
be $\diag\left(\mu_1,\dots,\mu_{j-1},\left[\begin{smallmatrix}\mu_j'&0\\\mu_{j+1}'&0\end{smallmatrix}\right],\mu_{j+2},\dots,\mu_n\right)$.
Then we have that 
$\xi_A = \xi_B\xi_C$ in $S_q(n,r)$.
 }

\vspace{1mm}
\noindent

To see this, it suffices to show for $D \in \Mat{\lambda}{\mu}$
that $g_D$,
the $x_B\otimes x_C$-coefficient of 
$\Delta(x_D)$ when expanded in terms of normally-ordered monomials,
is equal to $\delta_{A,D}$.
We have that
\begin{align*}
  x_B &= x_{\bi_1, (1^{\mu_1})} \cdots x_{\bi_{j-1}, ((j-1)^{\mu_{j-1}})} \big(x_{i,j}^{a_{i,j}}
\cdots x_{1,j}^{a_{1,j}} x_{n,j+1}^{a_{n,j}} \cdots x_{i+1,j+1}^{a_{i+1,j}}\big)
x_{\bi_{j+2}, ((j+2)^{\mu_{j+2}})} \cdots x_{\bi_n, (n^{\mu_n})},\\
x_C &= x_{1,1}^{\mu_1} \cdots x_{j-1,j-1}^{\mu_{j-1}}
\big(x_{j+1,j}^{\mu_{j+1}'} x_{j,j}^{\mu_j'}\big) x_{j+2,j+2}^{\mu_{j+2}}\cdots
      x_{n,n}^{\mu_n},\\
x_D &= x_{\bj_1,(1^{\mu_1})} \cdots x_{\bj_{j-1},((j-1)^{\mu_{j-1}})} \big(x_{n,j}^{d_{n,j}}
      \cdots x_{1,j}^{d_{1,j}}\big)
x_{\bj_{j+2},((j+2)^{\mu{j+2}})}\cdots x_{\bj_n, (n^{\mu_n})}    
\end{align*}
where $\bi_k := (n^{a_{n,k}}\ \cdots\ 1^{a_{1,k}})$ and 
$\bj_k = (n^{d_{n,k}}\ \cdots \ 1^{d_{1,k}})$.
It is easy to see that the $x_{\bi_k,(k^{\mu_k})} \otimes x_{k,k}^{\mu_k}$-coefficient
of $\Delta\big(x_{\bj_k,(k^{\mu_k})}\big)$ is 0 unless $\bj_k = \bi_k$, when it is
1. This implies that $g_D=0$ unless $\bj_k = \bi_k$
for each $k=1,\dots,j-1,j+2,\dots,n$
in which case, by weight considerations, we have that $d_{1,j}=a_{1,j},\dots,d_{n,j}=a_{n,j}$, hence, $D = A$.
Thus, we are reduced to showing that $g_A = 1$.

Our argument shows that $g_A$ is the coefficient of 
$x_{i,j}^{a_{i,j}}
\cdots x_{1,j}^{a_{1,j}} x_{n,j+1}^{a_{n,j}} \cdots x_{i+1,j+1}^{a_{i+1,j}}
\otimes 
x_{j+1,j}^{\mu_{j+1}'} x_{j,j}^{\mu_j'}$
in the normally-ordered expansion of 
$$
\Delta\big(x_{n,j}^{a_{n,j}}
      \cdots x_{1,j}^{a_{1,j}}) 
      =\sum_{\bk\in \I(n,r)} x_{(n^{a_{n,j}}\
  \cdots\ 1^{a_{1,j}}), \bk}\otimes x_{\bk, (j^{\mu_j})}.
$$
To complete the proof, we claim for $\bk \in \I(n,r)$ that
$x_{i,j}^{a_{i,j}}
\cdots x_{1,j}^{a_{1,j}} x_{n,j+1}^{a_{n,j}} \cdots x_{i+1,j+1}^{a_{i+1,j}}$
appears with non-zero coefficient in the normally-ordered expansion of
$x_{(n^{a_{n,j}}\
  \cdots\ 1^{a_{1,j}}), \bk}$ if and only if $\bk = ((j+1)^{\mu_{j+1}'}\
j^{\mu_j'})$, in which case the coefficient is 1.
Certainly, $\bk$ must be a permutation of 
$((j+1)^{\mu_{j+1}'}\
j^{\mu_j'})$.
For any such $\bk$ and any $\bh$ that is a permutation of 
$(n^{a_{n,j}}\
  \cdots\ 1^{a_{1,j}})$, we
define the {\em height} of the monomial $x_{\bh,\bk}$
to be $\sum_{s} h_s$ where the sum is over all $1 \leq s \leq \mu_j$
such that $k_s=j$.
The monomial
$x_{i,j}^{a_{i,j}}
\cdots x_{1,j}^{a_{1,j}} x_{n,j+1}^{a_{n,j}} \cdots x_{i+1,j+1}^{a_{i+1,j}}$
  is of height $\sum_{k=1}^i k a_{k,j}$. Also $x_{(n^{a_{n,j}}\
  \cdots\ 1^{a_{1,j}}), \bk}$ 
  is of the same height if
$\bk = ((j+1)^{\mu'_{j+1}}\ j^{\mu'_j})$,
and otherwise its height is strictly bigger.
In order to straighten
$x_{(n^{a_{n,j}}\
  \cdots\ 1^{a_{1,j}}), \bk}$, we need to use the commutation relations
$x_{p,j+1} x_{p,j} = q^{-1} x_{p,j} x_{p,j+1}$ 
and
$x_{p,j+1} x_{q,j} = x_{q,j} x_{p,j+1} - (q-q^{-1}) x_{p,j} x_{q,j+1}$ 
for  $p > q$. Monomials arising from the ``error term'' $x_{p,j} x_{q,j+1}$ are of strictly greater height, so do not contribute to the coefficient, 
and the others have the same height.
The claim follows.
\end{proof}

Now take any $A \in \Mat{\lambda}{\mu}$
and define $\lambda^-, \mu^+$ as in \cref{doso}.
Note that $n:=\ell(\lambda^-)=\ell(\mu^+) = \ell(\lambda)\ell(\mu)$.
We can convert $A$ into a matrix
$A^\circ \in \Mat{\lambda^-}{\mu^+}$ with at most one non-zero entry in each row and column
by applying a sequence of the operations $A
\mapsto B$ described in \cref{ged}(a)--(b).
For example:
\begin{equation*}
\left[\begin{smallmatrix} 1&0&3\\2&2&1\end{smallmatrix}\right]
\mapsto
\left[\begin{smallmatrix} 1&0&0\\0&0&3\\
2&2&1\end{smallmatrix}\right]
\mapsto
\left[\begin{smallmatrix} 1&0&0\\0&0&0\\
0&0&3\\
2&2&1\end{smallmatrix}\right]
\mapsto
\left[\begin{smallmatrix} 1&0&0\\0&0&0\\
0&0&3\\
2&0&0\\
0&2&1\end{smallmatrix}\right]
\mapsto
\left[\begin{smallmatrix} 1&0&0\\0&0&0\\
0&0&3\\
2&0&0\\
0&2&0\\
0&0&1\end{smallmatrix}\right]
\mapsto
\left[\begin{smallmatrix} 1&0&0&0\\0&0&0&0\\
0&0&0&3\\
0&2&0&0\\
0&0&2&0\\
0&0&0&1\end{smallmatrix}\right]
\mapsto
\left[\begin{smallmatrix} 1&0&0&0&0\\0&0&0&0&0\\
0&0&0&0&3\\
0&2&0&0&0\\
0&0&0&2&0\\
0&0&0&0&1\end{smallmatrix}\right]
\mapsto
\left[\begin{smallmatrix} 1&0&0&0&0&0\\0&0&0&0&0&0\\
0&0&0&0&3&0\\
0&2&0&0&0&0\\
0&0&0&2&0&0\\
0&0&0&0&0&1\end{smallmatrix}\right].
\end{equation*}
The $n \times n$ matrix $A^\circ$ obtained in this way is uniquely determined. 
It corresponds to the permutation of $n$ strings arising from the middle part of the
double coset diagram of $A$.
\cref{ged} plus \cref{kitchen,cabinet}
gives us an explicit algorithm to express the standard basis element
$\xi_A$ as a composition 
\begin{equation}\label{ducks}
\xi_A = 
\xi_{A^-} \circ \xi_{A^\circ} \circ \xi_{A^+}
\end{equation}
where $\xi_{A^-}$ is a tensor product of $\ell(\lambda)$ merges of $\ell(\mu)$ strings and $\xi_{A^+}$ is a tensor product of $\ell(\mu)$ split of $\ell(\lambda)$ strings.
The double coset diagrams of $A^- \in
\Mat{\lambda}{\lambda^-}$ and $A^+
\in \Mat{\mu^+}{\mu}$ 
are given 
explicitly by the top part or the bottom part of the diagram of $A$, respectively.

\begin{lemma}\label{laura}
For $a,b \geq 0$, we have that 
$\xi_{\left[\begin{smallmatrix}a&b\end{smallmatrix}\right]} \circ
\xi_{\left[\begin{smallmatrix}a\\b\end{smallmatrix}\right]} =
\qbinom{a+b}{a}_{\!q}
\xi_{\left[\begin{smallmatrix}a+b\end{smallmatrix}\right]}$.
\end{lemma}

\begin{proof}
The $x_{1,1}^a x_{1,2}^b \otimes x_{2,1}^b x_{1,1}^a$-coefficient of
$\Delta(x_{1,1}^{a+b})$ is $\qbinom{a+b}{a}_{\!q}$ by \cref{Assad}.
\end{proof}

\begin{lemma}\label{rider}
For $a,b,c,d \geq 0$ with $a+b=c+d$, 
we have that
\begin{align*}
\theta_{\left[\begin{smallmatrix}0&c\\a&d-a\end{smallmatrix}\right]}&=
\xi_{\left[\begin{smallmatrix}c\\d\end{smallmatrix}\right]}
\circ
\xi_{\left[\begin{smallmatrix}a&b\end{smallmatrix}\right]}
=
\sum_{s=0}^{\min(a,c)}
q^{s(s+d-a)}
\xi_{\left[\begin{smallmatrix}s&c-s\\a-s&s+d-a\end{smallmatrix}\right]}
&&\text{if $a \leq d$ and $b \geq c$,}\\
\theta_{\left[\begin{smallmatrix}c-b&b\\d&0\end{smallmatrix}\right]}
&=\xi_{\left[\begin{smallmatrix}c\\d\end{smallmatrix}\right]}
\circ
\xi_{\left[\begin{smallmatrix}a&b\end{smallmatrix}\right]}
=
\sum_{t=0}^{\min(b,d)}
q^{t(t+c-b)}
\xi_{\left[\begin{smallmatrix}t+c-b&b-t\\d-t&t\end{smallmatrix}\right]}
&&\text{if $a \geq d$ and $b \leq c$.}
\end{align*}
\end{lemma}

\begin{proof}
We just prove this when $a \leq d$, the other case is similar.
Since the merge
$\xi_{\left[\begin{smallmatrix}c\\d\end{smallmatrix}\right]}$
and the split 
$\xi_{\left[\begin{smallmatrix}a&b\end{smallmatrix}\right]}$
are bar invariant,
and the canonical basis element $\theta_A$ is the unique bar invariant element equal to $\xi_{A}$ plus 
a $q \Z[q]$-linear combination of other $\xi_B$,
the first equality follows from the second one.
To prove the second equality, we must show that the
$\xi_{\left[\begin{smallmatrix}s&c-s\\a-s&s+d-a\end{smallmatrix}\right]}$-coefficient of
$\xi_{\left[\begin{smallmatrix}c\\d\end{smallmatrix}\right]}\circ \xi_{\left[\begin{smallmatrix}a&c+d-a\end{smallmatrix}\right]}$
is $q^{s(s+d-a)}$.
This is the $x_{2,1}^{d} x_{1,1}^{c} \otimes x_{1,1}^a
x_{1,2}^{c+d-a}$-coefficient in $\Delta(x_{2,1}^{c-s} x_{1,1}^s
x_{2,2}^{s+d-a} x_{1,2}^{c-s})$, which may be computed by the same
argument as was used in the proof of \cref{pigs}.
\end{proof}

%% file: s6-presentations-arxiv.tex
\section{Presentations}\label{s6-presentations}

We start now to represent morphisms in $\qSchur$ by string diagrams. 
Let $\one$ be the strict identity object, that is, the 
composition $()$ of length zero.
For $\lambda=(\lambda_1,\dots,\lambda_\ell) \in \Lambda$, the identity endomorphism $1_\lambda$ in $\qSchur$ will be represented by a sequence of strings labelled from left to right by $\lambda_1,\dots,\lambda_\ell$, which we think of as indicating the {\em thicknesses} of the strings.
We are including strings of zero thickness.
For $a,b \geq 0$,
we use 
the string diagrams
\begin{align}\label{ms}
\begin{tikzpicture}[baseline=-2.5mm]
\draw[line width=0pt] (0,-.3) to (0,0);
\spot{0,0};
\node at (0,-.42) {$\scriptstyle 0$};
\end{tikzpicture}
&:(0)\rightarrow \one,&
\begin{tikzpicture}[baseline=1mm]
\draw[line width=0pt] (0,.3) to (0,0);
\spot{0,0};
\node at (0,.42) {$\scriptstyle 0$};
\end{tikzpicture}
&:\one\rightarrow (0),&
\begin{tikzpicture}[anchorbase,scale=.8]
	\draw[-,line width=1pt] (0.28,-.3) to (0.08,0.04);
	\draw[-,line width=1pt] (-0.12,-.3) to (0.08,0.04);
	\draw[-,line width=2pt] (0.08,.4) to (0.08,0);
        \node at (-0.22,-.4) {$\scriptstyle a$};
        \node at (0.35,-.4) {$\scriptstyle b$};
        \node at (0.08,.55) {$\scriptstyle a+b$};
\end{tikzpicture} 
&:(a,b) \rightarrow (a+b),&
\begin{tikzpicture}[anchorbase,scale=.8]
	\draw[-,line width=2pt] (0.08,-.3) to (0.08,0.04);
	\draw[-,line width=1pt] (0.28,.4) to (0.08,0);
	\draw[-,line width=1pt] (-0.12,.4) to (0.08,0);
        \node at (-0.22,.5) {$\scriptstyle a$};
        \node at (0.36,.5) {$\scriptstyle b$};
              \node at (0.08,-.44) {$\scriptstyle a+b$};
\end{tikzpicture}
&:(a+b)\rightarrow (a,b)
\end{align}
to denote the standard basis vectors
$\xi_A$ where $A$ is the $0 \times 1$ matrix, the $1 \times 0$ matrix, the matrix 
$\left[\begin{smallmatrix}a&b\end{smallmatrix}\right]$ or the matrix 
$\left[\begin{smallmatrix}a\\b\end{smallmatrix}\right]$, respectively.
Henceforth, in string diagrams for morphisms in $\qSchur$,
we will omit thickness labels on strings when
they are implicitly determined by the other labels. 
We 
represent the positive crossing
$\xi_{\left[\begin{smallmatrix}0&b\\a&0\end{smallmatrix}\right]}$
by the string diagram
\begin{align}\label{gareth}
\begin{tikzpicture}[anchorbase,scale=.8]
	\draw[-,line width=1pt] (0.3,-.3) to (-.3,.4);
	\draw[-,line width=4pt,white] (-0.3,-.3) to (.3,.4);
	\draw[-,line width=1pt] (-0.3,-.3) to (.3,.4);
        \node at (0.3,-.42) {$\scriptstyle b$};
        \node at (-0.3,-.42) {$\scriptstyle a$};
\end{tikzpicture}
&:(a,b) \rightarrow (b,a).
\end{align}
This morphism is invertible by \cref{pigs},
so it makes sense to define
$\begin{tikzpicture}[anchorbase,scale=.6]
	\draw[-,line width=1pt] (-0.3,-.3) to (.3,.4);
	\draw[-,line width=4pt,white] (0.3,-.3) to (-.3,.4);
	\draw[-,line width=1pt] (0.3,-.3) to (-.3,.4);
        \node at (0.3,-.44) {$\scriptstyle a$};
        \node at (-0.3,-.43) {$\scriptstyle b$};
\end{tikzpicture}
:=
\Big(\begin{tikzpicture}[anchorbase,scale=.6]
	\draw[-,line width=1pt] (0.3,-.3) to (-.3,.4);
	\draw[-,line width=4pt,white] (-0.3,-.3) to (.3,.4);
	\draw[-,line width=1pt] (-0.3,-.3) to (.3,.4);
        \node at (0.3,-.43) {$\scriptstyle b$};
        \node at (-0.3,-.44) {$\scriptstyle a$};
\end{tikzpicture}\Big)^{-1}$.

\begin{theorem}\label{mainpres}
The $\Z[q,q^{-1}]$-linear monoidal category $\qSchur$ is generated
by the objects $(r)$ for $r \geq 0$ and the morphisms $\begin{tikzpicture}[anchorbase,scale=1]
\draw[line width=0pt] (0,-.3) to (0,0);
\spot{0,0};
%\node at (0,-.42) {$\scriptstyle 0$};
\end{tikzpicture}\ $,
$\begin{tikzpicture}[anchorbase,scale=1]
\draw[line width=0pt] (0,.3) to (0,0);
\spot{0,0};
%\node at (0,.42) {$\scriptstyle 0$};
\end{tikzpicture}\ $,
$\begin{tikzpicture}[baseline=-1mm,scale=.6]
	\draw[-,line width=1pt] (0.28,-.3) to (0.08,0.04);
	\draw[-,line width=1pt] (-0.12,-.3) to (0.08,0.04);
	\draw[-,line width=2pt] (0.08,.4) to (0.08,0);
        \node at (-0.22,-.43) {$\scriptstyle a$};
        \node at (0.35,-.43) {$\scriptstyle b$};
\end{tikzpicture}$,
$\begin{tikzpicture}[baseline=0mm,scale=.6]
	\draw[-,line width=2pt] (0.08,-.3) to (0.08,0.04);
	\draw[-,line width=1pt] (0.28,.4) to (0.08,0);
	\draw[-,line width=1pt] (-0.12,.4) to (0.08,0);
        \node at (-0.22,.53) {$\scriptstyle a$};
        \node at (0.36,.55) {$\scriptstyle b$};
\end{tikzpicture}$ and
$\begin{tikzpicture}[baseline=-1.5mm,scale=.6]
	\draw[-,thick] (0.3,-.35) to (-.3,.43);
	\draw[-,line width=5pt,white] (-0.3,-.35) to (.3,.43);
	\draw[-,thick] (-0.3,-.3) to (.3,.4);
       \node at (-0.32,-.53) {$\scriptstyle a$};
        \node at (0.32,-.54) {$\scriptstyle b$};
\end{tikzpicture}$
for $a,b \geq 0$, subject only to the following relations
for $a,b,c,d \geq 0$ with $a+b=c+d$:
\begin{align}
\begin{tikzpicture}[anchorbase]
\draw[line width=0pt] (0,-.15) to (0,0.15);
\spot{0,.15};
\spot{0,-.15};
%\node at (0.15,0) {$\scriptstyle 0$};
\end{tikzpicture}\ &=\ 1_\one,&
\begin{tikzpicture}[anchorbase]
\draw[line width=0pt] (0,-.1) to (0,-0.3);
\draw[line width=0pt] (0,.1) to (0,0.3);
\spot{0,.1};
\spot{0,-.1};
%\node at (0,0.42) {$\scriptstyle 0$};
%\node at (0,-0.42) {$\scriptstyle 0$};
\end{tikzpicture}\ &=
\begin{tikzpicture}[anchorbase]
\draw[line width=0pt] (0,-.25) to (0,0.25);
\node at (0,-0.37) {$\scriptstyle 0$};
\node at (0,0.37) {$\scriptstyle \phantom{0}$};
\end{tikzpicture},\label{secondone}\\
\begin{tikzpicture}[anchorbase,scale=.8]
	\draw[-,line width=0pt] (0.28,-.3) to (0.08,0.04);
	\draw[-,line width=1pt] (-0.12,-.3) to (0.08,0.04);
	\draw[-,line width=1pt] (0.08,.4) to (0.08,0);
        \node at (-0.22,-.42) {$\scriptstyle a$};
        \node at (0.35,-.42) {$\scriptstyle 0$};
\end{tikzpicture} 
=
\begin{tikzpicture}[anchorbase,scale=.8]
	\draw[-,line width=1pt] (0.08,.4) to (0.08,-.3);
 	\draw[-,line width=0pt] (0.5,.1) to (0.5,-.3);
        \node at (0.08,-.42) {$\scriptstyle a$};
%        \node at (0.5,-.42) {$\scriptstyle 0$};
\spot{.5,.1};
\end{tikzpicture}\ ,\hspace{15mm}
\begin{tikzpicture}[anchorbase,scale=.8]
	\draw[-,line width=1pt] (0.28,-.3) to (0.08,0.04);
	\draw[-,line width=0pt] (-0.12,-.3) to (0.08,0.04);
	\draw[-,line width=1pt] (0.08,.4) to (0.08,0);
        \node at (-0.22,-.42) {$\scriptstyle 0$};
        \node at (0.35,-.42) {$\scriptstyle b$};
\end{tikzpicture} 
& =\begin{tikzpicture}[anchorbase,scale=.8]
	\draw[-,line width=0pt] (0.08,.1) to (0.08,-.3);
 	\draw[-,line width=1pt] (0.5,.4) to (0.5,-.3);
        \node at (0.5,-.43) {$\scriptstyle b$};
%        \node at (0.08,-.42) {$\scriptstyle 0$};
\spot{.08,.1};
\end{tikzpicture},&
\begin{tikzpicture}[anchorbase,scale=.8]
	\draw[-,line width=1pt] (0.08,-.3) to (0.08,0.04);
	\draw[-,line width=0pt] (0.28,.4) to (0.08,0);
	\draw[-,line width=1pt] (-0.12,.4) to (0.08,0);
        \node at (-0.22,.52) {$\scriptstyle a$};
        \node at (0.36,.52) {$\scriptstyle 0$};
\end{tikzpicture}&=\begin{tikzpicture}[anchorbase,scale=.8]
	\draw[-,line width=1pt] (0.08,-.4) to (0.08,.3);
 	\draw[-,line width=0pt] (0.5,-.1) to (0.5,.3);
        \node at (0.08,.42) {$\scriptstyle a$};
%        \node at (0.5,.42) {$\scriptstyle 0$};
\spot{.5,-.1};
\end{tikzpicture}\ ,\hspace{15mm}
\begin{tikzpicture}[anchorbase,scale=.8]
	\draw[-,line width=1pt] (0.08,-.3) to (0.08,0.04);
	\draw[-,line width=1pt] (0.28,.4) to (0.08,0);
	\draw[-,line width=0pt] (-0.12,.4) to (0.08,0);
        \node at (-0.22,.52) {$\scriptstyle 0$};
        \node at (0.36,.53) {$\scriptstyle b$};
\end{tikzpicture}=\begin{tikzpicture}[anchorbase,scale=.8]
	\draw[-,line width=0pt] (0.08,-.1) to (0.08,.3);
 	\draw[-,line width=1pt] (0.5,-.4) to (0.5,.3);
        \node at (0.5,.45) {$\scriptstyle b$};
%        \node at (0.08,.42) {$\scriptstyle 0$};
\spot{.08,-.1};
\end{tikzpicture},\label{zeroforks}\\
\label{assrel}
\begin{tikzpicture}[baseline = 0]
	\draw[-,thick] (0.35,-.3) to (0.08,0.14);
	\draw[-,thick] (0.1,-.3) to (-0.04,-0.06);
	\draw[-,line width=1pt] (0.085,.14) to (-0.035,-0.06);
	\draw[-,thick] (-0.2,-.3) to (0.07,0.14);
	\draw[-,line width=2pt] (0.08,.45) to (0.08,.1);
        \node at (0.45,-.41) {$\scriptstyle c$};
        \node at (0.07,-.4) {$\scriptstyle b$};
        \node at (-0.28,-.41) {$\scriptstyle a$};
\end{tikzpicture}
&=
\begin{tikzpicture}[baseline = 0]
	\draw[-,thick] (0.36,-.3) to (0.09,0.14);
	\draw[-,thick] (0.06,-.3) to (0.2,-.05);
	\draw[-,line width=1pt] (0.07,.14) to (0.19,-.06);
	\draw[-,thick] (-0.19,-.3) to (0.08,0.14);
	\draw[-,line width=2pt] (0.08,.45) to (0.08,.1);
        \node at (0.45,-.41) {$\scriptstyle c$};
        \node at (0.07,-.4) {$\scriptstyle b$};
        \node at (-0.28,-.41) {$\scriptstyle a$};
\end{tikzpicture}\:,
&
\begin{tikzpicture}[baseline = -1mm]
	\draw[-,thick] (0.35,.3) to (0.08,-0.14);
	\draw[-,thick] (0.1,.3) to (-0.04,0.06);
	\draw[-,line width=1pt] (0.085,-.14) to (-0.035,0.06);
	\draw[-,thick] (-0.2,.3) to (0.07,-0.14);
	\draw[-,line width=2pt] (0.08,-.45) to (0.08,-.1);
        \node at (0.45,.4) {$\scriptstyle c$};
        \node at (0.07,.42) {$\scriptstyle b$};
        \node at (-0.28,.4) {$\scriptstyle a$};
\end{tikzpicture}
&=\begin{tikzpicture}[baseline = -1mm]
	\draw[-,thick] (0.36,.3) to (0.09,-0.14);
	\draw[-,thick] (0.06,.3) to (0.2,.05);
	\draw[-,line width=1pt] (0.07,-.14) to (0.19,.06);
	\draw[-,thick] (-0.19,.3) to (0.08,-0.14);
	\draw[-,line width=2pt] (0.08,-.45) to (0.08,-.1);
        \node at (0.45,.4) {$\scriptstyle c$};
        \node at (0.07,.42) {$\scriptstyle b$};
        \node at (-0.28,.4) {$\scriptstyle a$};
\end{tikzpicture}\:,\\
\label{mergesplit}
\begin{tikzpicture}[anchorbase,scale=.8]
	\draw[-,line width=2pt] (0.08,-.8) to (0.08,-.5);
	\draw[-,line width=2pt] (0.08,.3) to (0.08,.6);
\draw[-,thick] (0.1,-.51) to [out=45,in=-45] (0.1,.31);
\draw[-,thick] (0.06,-.51) to [out=135,in=-135] (0.06,.31);
        \node at (-.33,-.05) {$\scriptstyle a$};
        \node at (.45,-.05) {$\scriptstyle b$};
\end{tikzpicture}
&= 
\qbinom{a+b}{a}_{\!q}\:\:
\begin{tikzpicture}[anchorbase,scale=.8]
	\draw[-,line width=2pt] (0.08,-.8) to (0.08,.6);
        \node at (.62,-.05) {$\scriptstyle a+b$};
\end{tikzpicture},&
\begin{tikzpicture}[anchorbase,scale=1]
	\draw[-,line width=1.2pt] (0,0) to (.275,.3) to (.275,.7) to (0,1);
	\draw[-,line width=1.2pt] (.6,0) to (.315,.3) to (.315,.7) to (.6,1);
        \node at (0,1.13) {$\scriptstyle c$};
        \node at (0.63,1.13) {$\scriptstyle d$};
        \node at (0,-.1) {$\scriptstyle a$};
        \node at (0.63,-.1) {$\scriptstyle b$};
\end{tikzpicture}
&=
\sum_{\substack{0 \leq s \leq \min(a,c)\\0 \leq t \leq \min(b,d)\\t-s=d-a=b-c}}
q^{st}
\begin{tikzpicture}[anchorbase,scale=1]
	\draw[-,thick] (0.58,0) to (0.58,.2) to (.02,.8) to (.02,1);
	\draw[-,line width=4pt,white] (0.02,0) to (0.02,.2) to (.58,.8) to (.58,1);
	\draw[-,thick] (0.02,0) to (0.02,.2) to (.58,.8) to (.58,1);
	\draw[-,thin] (0,0) to (0,1);
	\draw[-,thin] (0.6,0) to (0.6,1);
        \node at (0,1.13) {$\scriptstyle c$};
        \node at (0.6,1.13) {$\scriptstyle d$};
        \node at (0,-.1) {$\scriptstyle a$};
        \node at (0.6,-.1) {$\scriptstyle b$};
        \node at (-0.1,.5) {$\scriptstyle s$};
        \node at (0.75,.5) {$\scriptstyle t$};
\end{tikzpicture}.
\end{align}
Positive and negative crossings can be written in terms of other generating morphisms since we have that
\begin{align}
\begin{tikzpicture}[baseline=-1mm]
	\draw[-,line width=1pt] (0.3,-.3) to (-.3,.4);
	\draw[-,line width=4pt,white] (-0.3,-.3) to (.3,.4);
	\draw[-,line width=1pt] (-0.3,-.3) to (.3,.4);
        \node at (-0.3,-.4) {$\scriptstyle a$};
        \node at (0.3,-.4) {$\scriptstyle b$};
\end{tikzpicture}
&=\sum_{s=0}^{\min(a,b)}
(-q)^{s}
\begin{tikzpicture}[anchorbase,scale=1]
	\draw[-,line width=1.2pt] (0,0) to (0,1);
	\draw[-,thick] (-0.8,0) to (-0.8,.2) to (-.03,.4) to (-.03,.6)
        to (-.8,.8) to (-.8,1);
	\draw[-,thin] (-0.82,0) to (-0.82,1);
        \node at (-0.81,-.1) {$\scriptstyle a$};
        \node at (0,-.1) {$\scriptstyle b$};
        \node at (-0.4,.9) {$\scriptstyle b-s$};
        \node at (-0.4,.13) {$\scriptstyle a-s$};
\end{tikzpicture}
=
\sum_{s=0}^{\min(a,b)}
(-q)^{s}
\begin{tikzpicture}[anchorbase,scale=1]
	\draw[-,line width=1.2pt] (0,0) to (0,1);
	\draw[-,thick] (0.8,0) to (0.8,.2) to (.03,.4) to (.03,.6)
        to (.8,.8) to (.8,1);
	\draw[-,thin] (0.82,0) to (0.82,1);
        \node at (0.81,-.1) {$\scriptstyle b$};
        \node at (0,-.1) {$\scriptstyle a$};
        \node at (0.4,.9) {$\scriptstyle a-s$};
        \node at (0.4,.13) {$\scriptstyle b-s$};
\end{tikzpicture},\label{tourists}\\
\begin{tikzpicture}[baseline=-1mm]
	\draw[-,line width=1pt] (-0.3,-.3) to (.3,.4);
	\draw[-,line width=4pt,white] (0.3,-.3) to (-.3,.4);
	\draw[-,line width=1pt] (0.3,-.3) to (-.3,.4);
        \node at (-0.3,-.4) {$\scriptstyle a$};
        \node at (0.3,-.4) {$\scriptstyle b$};
\end{tikzpicture}
&=\sum_{s=0}^{\min(a,b)}
(-q)^{-s}
\begin{tikzpicture}[anchorbase,scale=1]
	\draw[-,line width=1.2pt] (0,0) to (0,1);
	\draw[-,thick] (-0.8,0) to (-0.8,.2) to (-.03,.4) to (-.03,.6)
        to (-.8,.8) to (-.8,1);
	\draw[-,thin] (-0.82,0) to (-0.82,1);
        \node at (-0.81,-.1) {$\scriptstyle a$};
        \node at (0,-.1) {$\scriptstyle b$};
        \node at (-0.4,.9) {$\scriptstyle b-s$};
        \node at (-0.4,.13) {$\scriptstyle a-s$};
\end{tikzpicture}
=
\sum_{s=0}^{\min(a,b)}
(-q)^{-s}
\begin{tikzpicture}[anchorbase,scale=1]
	\draw[-,line width=1.2pt] (0,0) to (0,1);
	\draw[-,thick] (0.8,0) to (0.8,.2) to (.03,.4) to (.03,.6)
        to (.8,.8) to (.8,1);
	\draw[-,thin] (0.82,0) to (0.82,1);
        \node at (0.81,-.1) {$\scriptstyle b$};
        \node at (0,-.1) {$\scriptstyle a$};
        \node at (0.4,.9) {$\scriptstyle a-s$};
        \node at (0.4,.13) {$\scriptstyle b-s$};
\end{tikzpicture}.\label{visas}
\end{align}
Moreover, the following hold:
\begin{enumerate}
\item
There is a unique braiding $c:-\star- \stackrel{\sim}{\Rightarrow} -\star^\rev-$ making $\qSchur$ into a braided monoidal category such that $c_{(a),(b)}=\begin{tikzpicture}[baseline=-1.5mm,scale=.6]
	\draw[-,thick] (0.3,-.35) to (-.3,.43);
	\draw[-,line width=5pt,white] (-0.3,-.35) to (.3,.43);
	\draw[-,thick] (-0.3,-.3) to (.3,.4);
       \node at (-0.32,-.53) {$\scriptstyle a$};
        \node at (0.32,-.54) {$\scriptstyle b$};
\end{tikzpicture}$.
\item
For any $A \in \Mat{\lambda}{\mu}$, the standard basis element $\xi_A$
is represented as a string diagram by
the double coset diagram for $A$
with all crossings drawn as positive crossings.
\item The anti-linear involution $-:\qSchur\rightarrow \qSchur$ is defined on string diagrams by interchanging positive and negative crossings.
\item The linear isomorphism $\T:\qSchur \rightarrow \qSchur^\op$
maps a string diagram to its rotation through 180$^\circ$ around a horizontal axis.
\end{enumerate}
\end{theorem}

Before we prove this, some comments.
The relations \cref{secondone} imply that $(0) \cong \one$.
The relations \cref{zeroforks} mean that splits and merges with a string of thickness zero can be expressed in terms of the other generating morphisms, hence, can be eliminated from any string diagram. Using \cref{secondone}, \cref{zeroforks} and the definition of negative crossings, the second relation in \cref{mergesplit} implies that 
\begin{align}
\begin{tikzpicture}[anchorbase,scale=.8]
	\draw[-,line width=0pt] (0.3,-.3) to (-.3,.4);
	\draw[-,line width=4pt,white] (-0.3,-.3) to (.3,.4);
	\draw[-,line width=1pt] (-0.3,-.3) to (.3,.4);
        \node at (0.3,-.43) {$\scriptstyle 0$};
        \node at (-0.3,-.42) {$\scriptstyle a$};
\end{tikzpicture}
=
\begin{tikzpicture}[anchorbase,scale=.8]
	\draw[-,line width=1pt] (-0.3,-.3) to (.3,.4);
	\draw[-,line width=4pt,white] (0.3,-.3) to (-.3,.4);
	\draw[-,line width=0pt] (0.3,-.3) to (-.3,.4);
        \node at (0.3,-.43) {$\scriptstyle 0$};
        \node at (-0.3,-.42) {$\scriptstyle a$};
\end{tikzpicture}&=
\begin{tikzpicture}[anchorbase,scale=.8]
	\draw[-,line width=1pt] (0,-.3) to (0,.4);
        \node at (0,-.42) {$\scriptstyle a$};
	\draw[-,line width=0pt] (-0.3,.1) to (-0.3,.4);\spot{-.3,.1};
 	\draw[-,line width=0pt] (0.3,0) to (0.3,-.3);\spot{.3,0};
\end{tikzpicture}\ ,&
\begin{tikzpicture}[anchorbase,scale=.8]
	\draw[-,line width=1pt] (0.3,-.3) to (-.3,.4);
	\draw[-,line width=4pt,white] (-0.3,-.3) to (.3,.4);
	\draw[-,line width=0pt] (-0.3,-.3) to (.3,.4);
        \node at (0.3,-.44) {$\scriptstyle b$};
        \node at (-0.3,-.43) {$\scriptstyle 0$};
\end{tikzpicture}
=
\begin{tikzpicture}[anchorbase,scale=.8]
	\draw[-,line width=0pt] (-0.3,-.3) to (.3,.4);
	\draw[-,line width=4pt,white] (0.3,-.3) to (-.3,.4);
	\draw[-,line width=1pt] (0.3,-.3) to (-.3,.4);
       \node at (0.3,-.44) {$\scriptstyle b$};
        \node at (-0.3,-.43) {$\scriptstyle 0$};
\end{tikzpicture}&=\begin{tikzpicture}[anchorbase,scale=.8]
	\draw[-,line width=1pt] (0,-.3) to (0,.4);
        \node at (0,-.44) {$\scriptstyle b$};
	\draw[-,line width=0pt] (0.3,.1) to (0.3,.4);\spot{.3,.1};
 	\draw[-,line width=0pt] (-0.3,0) to (-0.3,-.3);\spot{-.3,0};
\end{tikzpicture}\ .
\end{align}
This means that crossings involving a string of thickness zero can also be expressed in terms of other morphisms, so these can be eliminated from string diagrams too.
Then all remaining
strings of thickness zero can be contracted to dots on the top and bottom boundaries. In this way,
any string diagram is equivalent to one without strings of thickness zero.
The relation \cref{assrel} means that we can introduce further diagrams as shorthands  more general 
splits 
and merges of $n$ strings.
For example, splits and merges of 3 strings
are
\begin{align}
\begin{tikzpicture}[baseline = 0]
	\draw[-,thick] (0.35,.3) to (0.09,-0.14);
	\draw[-,thick] (0.08,.3) to (.08,-0.1);
	\draw[-,thick] (-0.2,.3) to (0.07,-0.14);
	\draw[-,line width=2pt] (0.08,-.45) to (0.08,-.1);
        \node at (0.45,.41) {$\scriptstyle c$};
        \node at (0.07,.43) {$\scriptstyle b$};
        \node at (-0.28,.41) {$\scriptstyle a$};
\end{tikzpicture}
:=
\begin{tikzpicture}[baseline = 0]
	\draw[-,thick] (0.35,.3) to (0.08,-0.14);
	\draw[-,thick] (0.1,.3) to (-0.04,0.06);
	\draw[-,line width=1pt] (0.085,-.14) to (-0.035,0.06);
	\draw[-,thick] (-0.2,.3) to (0.07,-0.14);
	\draw[-,line width=2pt] (0.08,-.45) to (0.08,-.1);
        \node at (0.45,.41) {$\scriptstyle c$};
        \node at (0.07,.43) {$\scriptstyle b$};
        \node at (-0.28,.41) {$\scriptstyle a$};
\end{tikzpicture}
&=
\begin{tikzpicture}[baseline = 0]
	\draw[-,thick] (0.36,.3) to (0.09,-0.14);
	\draw[-,thick] (0.06,.3) to (0.2,.05);
	\draw[-,line width=1pt] (0.07,-.14) to (0.19,.06);
	\draw[-,thick] (-0.19,.3) to (0.08,-0.14);
	\draw[-,line width=2pt] (0.08,-.45) to (0.08,-.1);
        \node at (0.45,.41) {$\scriptstyle c$};
        \node at (0.07,.43) {$\scriptstyle b$};
        \node at (-0.28,.41) {$\scriptstyle a$};
\end{tikzpicture}\ ,&
\begin{tikzpicture}[baseline = 0]
	\draw[-,thick] (0.35,-.3) to (0.09,0.14);
	\draw[-,thick] (0.08,-.3) to (.08,0.1);
	\draw[-,thick] (-0.2,-.3) to (0.07,0.14);
	\draw[-,line width=2pt] (0.08,.45) to (0.08,.1);
        \node at (0.45,-.41) {$\scriptstyle c$};
        \node at (0.07,-.4) {$\scriptstyle b$};
        \node at (-0.28,-.41) {$\scriptstyle a$};
\end{tikzpicture}
:=
\begin{tikzpicture}[baseline = 0]
	\draw[-,thick] (0.35,-.3) to (0.08,0.14);
	\draw[-,thick] (0.1,-.3) to (-0.04,-0.06);
	\draw[-,line width=1pt] (0.085,.14) to (-0.035,-0.06);
	\draw[-,thick] (-0.2,-.3) to (0.07,0.14);
	\draw[-,line width=2pt] (0.08,.45) to (0.08,.1);
        \node at (0.45,-.41) {$\scriptstyle c$};
        \node at (0.07,-.4) {$\scriptstyle b$};
        \node at (-0.28,-.41) {$\scriptstyle a$};
\end{tikzpicture}
&=
\begin{tikzpicture}[baseline = 0]
	\draw[-,thick] (0.36,-.3) to (0.09,0.14);
	\draw[-,thick] (0.06,-.3) to (0.2,-.05);
	\draw[-,line width=1pt] (0.07,.14) to (0.19,-.06);
	\draw[-,thick] (-0.19,-.3) to (0.08,0.14);
	\draw[-,line width=2pt] (0.08,.45) to (0.08,.1);
        \node at (0.45,-.41) {$\scriptstyle c$};
        \node at (0.07,-.4) {$\scriptstyle b$};
        \node at (-0.28,-.41) {$\scriptstyle a$};
\end{tikzpicture}\:.
\end{align}
By \cref{kitchen,cabinet}, these are the standard basis vectors
$\xi_{\left[\begin{smallmatrix}a&b&c\end{smallmatrix}\right]}$
 and $\xi_{\left[\begin{smallmatrix}a\\b\\c\end{smallmatrix}\right]}$, respectively.

\iffalse
Finally, we note that the symmetry $\T$ from \cref{starinv}
can be computed on string diagrams by rotating around a horizontal axis. The anti-linear 
bar involution fixes the merge and split diagrams and interchanges positive and negative crossings.
\fi

\begin{proof}[Proof of \cref{mainpres}]
Let $\qSchur'$ be the strict $\Z[q,q^{-1}]$-linear monoidal category defined by the generators and relations in the statement of the theorem. 
We also {\em define} the negative crossings in $\qSchur'$ by setting
\begin{equation}\label{tireddog}
\begin{tikzpicture}[baseline=-1mm,scale=.9]
	\draw[-,line width=1pt] (-0.3,-.3) to (.3,.4);
	\draw[-,line width=4pt,white] (0.3,-.3) to (-.3,.4);
	\draw[-,line width=1pt] (0.3,-.3) to (-.3,.4);
        \node at (-0.3,-.4) {$\scriptstyle a$};
        \node at (0.3,-.4) {$\scriptstyle b$};
\end{tikzpicture}
:=\sum_{s=0}^{\min(a,b)}
(-q)^{-s}
\begin{tikzpicture}[anchorbase,scale=.9]
	\draw[-,line width=1.2pt] (0,0) to (0,1);
	\draw[-,thick] (-0.8,0) to (-0.8,.2) to (-.03,.4) to (-.03,.6)
        to (-.8,.8) to (-.8,1);
	\draw[-,thin] (-0.82,0) to (-0.82,1);
        \node at (-0.81,-.1) {$\scriptstyle a$};
        \node at (0,-.1) {$\scriptstyle b$};
        \node at (-0.4,.9) {$\scriptstyle b-s$};
        \node at (-0.4,.13) {$\scriptstyle a-s$};
\end{tikzpicture}\ .
\end{equation}
At this point, some calculations are needed to deduce the following additional relations from the defining relations in
$\qSchur'$ (for all $a,b,c,d \geq 0$ that make sense):
\begin{align}
\begin{tikzpicture}[anchorbase,scale=.9]
	\draw[-,thick] (0,0) to (0,1);
	\draw[-,thick] (0.015,0) to (0.015,.2) to (.57,.4) to (.57,.6)
        to (.015,.8) to (.015,1);
	\draw[-,line width=1.2pt] (0.6,0) to (0.6,1);
        \node at (0.6,-.1) {$\scriptstyle b$};
        \node at (0,-.1) {$\scriptstyle a$};
        \node at (0.3,.82) {$\scriptstyle c$};
        \node at (0.3,.19) {$\scriptstyle d$};
\end{tikzpicture}
&=
\sum_{s=\max(0,c-b)}^{\min(c,d)}
q^{s(b-c+s)}
\qbinom{a-d+s}{s}_{\!q}\displaystyle
\begin{tikzpicture}[anchorbase,scale=.9]
	\draw[-,thick] (0.88,0) to (0.88,.2) to (.02,.8) to (.02,1);
	\draw[-,line width=4pt, white] (0.02,0) to (0.02,.2) to (.88,.8) to (.88,1);
	\draw[-,thick] (0.02,0) to (0.02,.2) to (.88,.8) to (.88,1);
	\draw[-,thin] (0,0) to (0,1);
	\draw[-,thin] (0.9,0) to (0.9,1);
        \node at (0,-.1) {$\scriptstyle a$};
        \node at (0.9,-.1) {$\scriptstyle b$};
        \node at (.36,.16) {$\scriptstyle d-s$};
        \node at (.36,.86) {$\scriptstyle c-s$};
\end{tikzpicture}=
\sum_{s=\max(0,c-b)}^{\min(c,d)}
\qbinom{a-b+c-d}{s}_{\!q}
\begin{tikzpicture}[anchorbase,scale=.9]
	\draw[-,line width=1.2pt] (0,0) to (0,1);
	\draw[-,thick] (0.79,0) to (0.79,.2) to (.03,.4) to (.03,.6)
        to (.79,.8) to (.79,1);
	\draw[-,thin] (0.81,0) to (0.81,1);
        \node at (0.8,-.1) {$\scriptstyle b$};
        \node at (0,-.1) {$\scriptstyle a$};
        \node at (0.38,.9) {$\scriptstyle d-s$};
        \node at (0.38,.14) {$\scriptstyle c-s$};
\end{tikzpicture},
\label{jonsquare}
\\
\begin{tikzpicture}[anchorbase,scale=.9]
	\draw[-,thick] (0,0) to (0,1);
	\draw[-,thick] (-.015,0) to (-0.015,.2) to (-.57,.4) to (-.57,.6)
        to (-.015,.8) to (-.015,1);
	\draw[-,line width=1.2pt] (-0.6,0) to (-0.6,1);
        \node at (-0.6,-.1) {$\scriptstyle b$};
        \node at (0,-.1) {$\scriptstyle a$};
        \node at (-0.3,.84) {$\scriptstyle c$};
        \node at (-0.3,.19) {$\scriptstyle d$};
\end{tikzpicture}
&=
\sum_{s=\max(0,c-b)}^{\min(c,d)}
q^{s(b-c+s)}
\qbinom{a-d+s}{s}_{\!q}\displaystyle
\begin{tikzpicture}[anchorbase,scale=.9]
	\draw[-,thick] (0.88,0) to (0.88,.2) to (.02,.8) to (.02,1);
	\draw[-,line width=4pt, white] (0.02,0) to (0.02,.2) to (.88,.8) to (.88,1);
	\draw[-,thick] (0.02,0) to (0.02,.2) to (.88,.8) to (.88,1);
	\draw[-,thin] (0,0) to (0,1);
	\draw[-,thin] (0.9,0) to (0.9,1);
        \node at (0,-.1) {$\scriptstyle b$};
        \node at (.9,-.1) {$\scriptstyle a$};
        \node at (.36,.16) {$\scriptstyle c-s$};
        \node at (.36,.87) {$\scriptstyle d-s$};
\end{tikzpicture}=
\!\!\sum_{s=\max(0,c-b)}^{\min(c,d)}
\qbinom{a-b+c-d}{s}_{\!q}
\begin{tikzpicture}[anchorbase,scale=.9]
	\draw[-,line width=1.2pt] (0,0) to (0,1);
	\draw[-,thick] (-0.8,0) to (-0.8,.2) to (-.03,.4) to (-.03,.6)
        to (-.8,.8) to (-.8,1);
	\draw[-,thin] (-0.82,0) to (-0.82,1);
        \node at (-0.81,-.1) {$\scriptstyle b$};
        \node at (0,-.1) {$\scriptstyle a$};
        \node at (-0.4,.9) {$\scriptstyle d-s$};
        \node at (-0.4,.15) {$\scriptstyle c-s$};
\end{tikzpicture},
\label{jonsquare2}\\
\label{thickcrossing}
\begin{tikzpicture}[anchorbase,scale=1.1]
	\draw[-,line width=1pt] (0.3,-.3) to (-.3,.4);
	\draw[-,line width=4pt,white] (-0.3,-.3) to (.3,.4);
	\draw[-,line width=1pt] (-0.3,-.3) to (.3,.4);
        \node at (0.3,-.42) {$\scriptstyle b$};
        \node at (-0.3,-.42) {$\scriptstyle a$};
\end{tikzpicture}
&=
\begin{tikzpicture}[baseline=3mm,scale=.8]
	\draw[-,line width=1.2pt] (.6,0) to (.315,.3) to (.315,.7) to (.6,1);
	\draw[-,line width=1.2pt] (0,0) to (.285,.3) to (.285,.7) to (0,1);
        \node at (0,-.1) {$\scriptstyle a$};
        \node at (0.6,-.1) {$\scriptstyle b$};
    \node at (0,1.13) {$\scriptstyle b$};
        \node at (0.63,1.13) {$\scriptstyle a$};
        \end{tikzpicture}
-\sum_{s=1}^{\min(a,b)} \!\!\!q^{s^2}
\begin{tikzpicture}[anchorbase,scale=.9]
	\draw[-,thick] (0.58,0) to (0.58,.2) to (.02,.8) to (.02,1);
	\draw[-,line width=4pt,white] (0.02,0) to (0.02,.2) to (.58,.8) to (.58,1);
	\draw[-,thick] (0.02,0) to (0.02,.2) to (.58,.8) to (.58,1);
	\draw[-,thin] (0,0) to (0,1);
	\draw[-,thin] (0.6,0) to (0.6,1);
 %       \node at (0,1.13) {$\scriptstyle b$};
  %      \node at (.6,1.13) {$\scriptstyle a$};
        \node at (0,-.1) {$\scriptstyle a$};
        \node at (0.6,-.1) {$\scriptstyle b$};
        \node at (-0.1,.5) {$\scriptstyle s$};
        \node at (0.7,.5) {$\scriptstyle s$};
\end{tikzpicture}=
\sum_{s=0}^{\min(a,b)}(-q)^{s}
\begin{tikzpicture}[anchorbase,scale=.9]
	\draw[-,line width=1.2pt] (0,0) to (0,1);
	\draw[-,thick] (-0.8,0) to (-0.8,.2) to (-.03,.4) to (-.03,.6)
        to (-.8,.8) to (-.8,1);
	\draw[-,thin] (-0.82,0) to (-0.82,1);
        \node at (-0.81,-.1) {$\scriptstyle a$};
        \node at (0,-.1) {$\scriptstyle b$};
        \node at (-0.4,.9) {$\scriptstyle b-s$};
        \node at (-0.4,.13) {$\scriptstyle a-s$};
\end{tikzpicture}
=
\sum_{s=0}^{\min(a,b)}
(-q)^{s}
\begin{tikzpicture}[anchorbase,scale=.9]
	\draw[-,line width=1.2pt] (0,0) to (0,1);
	\draw[-,thick] (0.8,0) to (0.8,.2) to (.03,.4) to (.03,.6)
        to (.8,.8) to (.8,1);
	\draw[-,thin] (0.82,0) to (0.82,1);
        \node at (0.81,-.1) {$\scriptstyle b$};
        \node at (0,-.1) {$\scriptstyle a$};
        \node at (0.4,.9) {$\scriptstyle a-s$};
        \node at (0.4,.13) {$\scriptstyle b-s$};
\end{tikzpicture},\end{align}

\vspace{-4.5mm}

\begin{align}
\begin{tikzpicture}[anchorbase,scale=0.62]
	\draw[-,thick] (0.4,0) to (-0.6,1);
	\draw[-,line width=4pt,white] (0.1,0) to (0.1,.6) to (.5,1);
	\draw[-,thick] (0.1,0) to (0.1,.6) to (.5,1);
	\draw[-,thick] (0.08,0) to (0.08,1);
        \node at (0.6,1.13) {$\scriptstyle c$};
        \node at (0.1,1.16) {$\scriptstyle b$};
        \node at (-0.65,1.13) {$\scriptstyle a$};
\end{tikzpicture}
&=
\begin{tikzpicture}[anchorbase,scale=0.62]
	\draw[-,thick] (0.7,0) to (-0.3,1);
	\draw[-,line width=4pt,white] (0.1,0) to (0.1,.2) to (.9,1);
	\draw[-,line width=3pt,white] (0.08,0) to (0.08,1);
	\draw[-,thick] (0.1,0) to (0.1,.2) to (.9,1);
	\draw[-,thick] (0.08,0) to (0.08,1);
        \node at (0.9,1.13) {$\scriptstyle c$};
        \node at (0.1,1.16) {$\scriptstyle b$};
        \node at (-0.4,1.13) {$\scriptstyle a$};
\end{tikzpicture},&
\begin{tikzpicture}[anchorbase,scale=0.62]
	\draw[-,thick] (-0.08,0) to (-0.08,1);
	\draw[-,thick] (-0.1,0) to (-0.1,.6) to (-.5,1);
	\draw[-,line width=4pt,white] (-0.4,0) to (0.6,1);
	\draw[-,thick] (-0.4,0) to (0.6,1);
        \node at (0.7,1.13) {$\scriptstyle c$};
        \node at (-0.1,1.16) {$\scriptstyle b$};
        \node at (-0.6,1.13) {$\scriptstyle a$};
\end{tikzpicture}
&=
\begin{tikzpicture}[anchorbase,scale=0.62]
	\draw[-,thick] (-0.08,0) to (-0.08,1);
	\draw[-,thick] (-0.1,0) to (-0.1,.2) to (-.9,1);
	\draw[-,line width=4pt,white] (-0.7,0) to (0.3,1);
	\draw[-,thick] (-0.7,0) to (0.3,1);
        \node at (0.4,1.13) {$\scriptstyle c$};
        \node at (-0.1,1.16) {$\scriptstyle b$};
        \node at (-0.95,1.13) {$\scriptstyle a$};
\end{tikzpicture},&
\:\begin{tikzpicture}[baseline=-3.3mm,scale=0.62]
	\draw[-,thick] (0.08,0) to (0.08,-1);
	\draw[-,thick] (0.1,0) to (0.1,-.6) to (.5,-1);
	\draw[-,line width=4pt,white] (0.4,0) to (-0.6,-1);
	\draw[-,thick] (0.4,0) to (-0.6,-1);
        \node at (0.6,-1.13) {$\scriptstyle c$};
        \node at (0.07,-1.13) {$\scriptstyle b$};
        \node at (-0.6,-1.13) {$\scriptstyle a$};
\end{tikzpicture}
&=
\begin{tikzpicture}[baseline=-3.3mm,scale=0.62]
	\draw[-,thick] (0.08,0) to (0.08,-1);
	\draw[-,thick] (0.1,0) to (0.1,-.2) to (.9,-1);
	\draw[-,line width=4pt, white] (0.7,0) to (-0.3,-1);
	\draw[-,thick] (0.7,0) to (-0.3,-1);
        \node at (1,-1.13) {$\scriptstyle c$};
        \node at (0.1,-1.13) {$\scriptstyle b$};
        \node at (-0.4,-1.13) {$\scriptstyle a$};
\end{tikzpicture},\:&
\begin{tikzpicture}[baseline=-3.3mm,scale=0.62]
	\draw[-,thick] (-0.4,0) to (0.6,-1);
	\draw[-,line width=4pt,white] (-0.1,0) to (-0.1,-.6) to (-.5,-1);
	\draw[-,thick] (-0.1,0) to (-0.1,-.6) to (-.5,-1);
	\draw[-,thick] (-0.08,0) to (-0.08,-1);
        \node at (0.6,-1.13) {$\scriptstyle c$};
        \node at (-0.1,-1.13) {$\scriptstyle b$};
        \node at (-0.6,-1.13) {$\scriptstyle a$};
\end{tikzpicture}
&=
\begin{tikzpicture}[baseline=-3.3mm,scale=0.62]
	\draw[-,thick] (-0.7,0) to (0.3,-1);
	\draw[-,line width=4pt,white] (-0.1,0) to (-0.1,-.2) to (-.9,-1);
	\draw[-,line width=3pt,white] (-0.08,0) to (-0.08,-1);
	\draw[-,thick] (-0.1,0) to (-0.1,-.2) to (-.9,-1);
	\draw[-,thick] (-0.08,0) to (-0.08,-1);
        \node at (0.34,-1.13) {$\scriptstyle c$};
        \node at (-0.1,-1.13) {$\scriptstyle b$};
        \node at (-0.95,-1.13) {$\scriptstyle a$};
\end{tikzpicture},
\label{sliders}
\end{align}

\vspace{-4.5mm}

\begin{align}
\begin{tikzpicture}[anchorbase,scale=.62]
\draw[-,thick] (-.2,-.8) to [out=45,in=-45] (0.1,.31);
\draw[-,line width=4.3pt, white] (.36,-.8) to [out=135,in=-135] (0.06,.31);
\draw[-,thick] (.36,-.8) to [out=135,in=-135] (0.06,.31);
	\draw[-,line width=2pt] (0.08,.3) to (0.08,.5);
        \node at (-.3,-.95) {$\scriptstyle a$};
        \node at (.45,-.95) {$\scriptstyle b$};
\end{tikzpicture}
&=
q^{ab}\!\begin{tikzpicture}[anchorbase,scale=.62]
	\draw[-,line width=2pt] (0.08,.1) to (0.08,.5);
\draw[-,thick] (.46,-.8) to [out=100,in=-45] (0.1,.11);
\draw[-,thick] (-.3,-.8) to [out=80,in=-135] (0.06,.11);
        \node at (-.3,-.95) {$\scriptstyle a$};
        \node at (.43,-.95) {$\scriptstyle b$};
\end{tikzpicture},
&
\begin{tikzpicture}[baseline=.3mm,scale=.62]
\draw[-,thick] (.36,.8) to [out=-135,in=135] (0.06,-.31);
\draw[-,line width=4.3pt,white] (-.2,.8) to [out=-45,in=45] (0.1,-.31);
\draw[-,thick] (-.2,.8) to [out=-45,in=45] (0.1,-.31);
	\draw[-,line width=2pt] (0.08,-.3) to (0.08,-.5);
        \node at (-.3,.95) {$\scriptstyle a$};
        \node at (.45,.95) {$\scriptstyle b$};
\end{tikzpicture}
&=q^{ab}
\begin{tikzpicture}[baseline=.3mm,scale=.62]
	\draw[-,line width=2pt] (0.08,-.1) to (0.08,-.5);
\draw[-,thick] (.46,.8) to [out=-100,in=45] (0.1,-.11);
\draw[-,thick] (-.3,.8) to [out=-80,in=135] (0.06,-.11);
        \node at (-.3,.95) {$\scriptstyle a$};
        \node at (.43,.95) {$\scriptstyle b$};
\end{tikzpicture},
\label{swallows-symmetric}
&
\begin{tikzpicture}[baseline = -1mm,scale=0.71]
	\draw[-,thick] (-0.28,0) to[out=90,in=-90] (0.28,.6);
	\draw[-,thick] (0.28,-.6) to[out=90,in=-90] (-0.28,0);
	\draw[-,line width=4pt,white] (0.28,0) to[out=90,in=-90] (-0.28,.6);
	\draw[-,line width=4pt,white] (-0.28,-.6) to[out=90,in=-90] (0.28,0);
	\draw[-,thick] (0.28,0) to[out=90,in=-90] (-0.28,.6);
	\draw[-,thick] (-0.28,-.6) to[out=90,in=-90] (0.28,0);
        \node at (0.3,-.75) {$\scriptstyle b$};
        \node at (-0.3,-.75) {$\scriptstyle a$};
\end{tikzpicture}
&=
\begin{tikzpicture}[baseline = -1mm,scale=0.71]
	\draw[-,thick] (0.3,-.6) to (0.3,.6);
	\draw[-,thick] (-0.2,-.6) to (-0.2,.6);
        \node at (0.3,-.75) {$\scriptstyle b$};
        \node at (-0.2,-.75) {$\scriptstyle a$};
\end{tikzpicture}\ ,&
\begin{tikzpicture}[baseline = -1mm,scale=0.71]
	\draw[-,thick] (0.28,0) to[out=90,in=-90] (-0.28,.6);
	\draw[-,thick] (-0.28,-.6) to[out=90,in=-90] (0.28,0);
	\draw[-,line width=4pt,white] (-0.28,0) to[out=90,in=-90] (0.28,.6);
	\draw[-,line width=4pt,white] (0.28,-.6) to[out=90,in=-90] (-0.28,0);
	\draw[-,thick] (-0.28,0) to[out=90,in=-90] (0.28,.6);
	\draw[-,thick] (0.28,-.6) to[out=90,in=-90] (-0.28,0);
        \node at (0.3,-.75) {$\scriptstyle a$};
        \node at (-0.3,-.75) {$\scriptstyle b$};
\end{tikzpicture}
&=
\begin{tikzpicture}[baseline = -1mm,scale=0.71]
	\draw[-,thick] (0.2,-.6) to (0.2,.6);
	\draw[-,thick] (-0.3,-.6) to (-0.3,.6);
        \node at (0.2,-.75) {$\scriptstyle a$};
        \node at (-0.3,-.75) {$\scriptstyle b$};
\end{tikzpicture}\:,
\end{align}

\vspace{-4.5mm}

\begin{align}
\begin{tikzpicture}[baseline = -1mm,scale=0.71]
	\draw[-,thick] (0.45,-.6) to (-0.45,.6);
        \draw[-,thick] (0,-.6) to[out=90,in=-90] (-.45,0);
        \draw[-,line width=4pt,white] (-0.45,0) to[out=90,in=-90] (0,0.6);
        \draw[-,thick] (-0.45,0) to[out=90,in=-90] (0,0.6);
	\draw[-,line width=4pt,white] (0.45,.6) to (-0.45,-.6);
	\draw[-,thick] (0.45,.6) to (-0.45,-.6);
        \node at (0,-.77) {$\scriptstyle b$};
        \node at (0.5,-.77) {$\scriptstyle c$};
        \node at (-0.5,-.77) {$\scriptstyle a$};
\end{tikzpicture}
&=
\begin{tikzpicture}[baseline = -1mm,scale=0.71]
	\draw[-,thick] (0.45,-.6) to (-0.45,.6);
        \draw[-,thick] (0.45,0) to[out=90,in=-90] (0,0.6);
        \draw[-,line width=4pt,white] (0,-.6) to[out=90,in=-90] (.45,0);
        \draw[-,thick] (0,-.6) to[out=90,in=-90] (.45,0);
	\draw[-,line width=4pt,white] (0.45,.6) to (-0.45,-.6);
	\draw[-,thick] (0.45,.6) to (-0.45,-.6);
        \node at (0,-.77) {$\scriptstyle b$};
        \node at (0.5,-.77) {$\scriptstyle c$};
        \node at (-0.5,-.77) {$\scriptstyle a$};
\end{tikzpicture}.\label{braid}
\end{align}

\noindent
The derivations of these relations are similar to those in the appendix of \cite{BEEO} (which treats the $q=1$ case); \cref{vanishing,second,A} are needed along the way. The detailed calculations can be found in the appendix.

Now we prove (a) but for the presented category 
$\qSchur'$ rather than $\qSchur$ itself;
then (a) for $\qSchur$ follows at the end
when we have established that
$\qSchur'\cong \qSchur$.
We need natural isomorphisms
$c_{\lambda,\mu}:\lambda \star \mu\stackrel{\sim}{\rightarrow} \mu\star\lambda$ for all compositions $\lambda,\mu$. 
Given that $c_{(a),(b)}$ is the positive crossing, there is no choice for the definition of more general $c_{\lambda,\mu}$
in order for the hexagon axioms for a braided monoidal category to hold: it must be defined by composing
positive crossings according to a reduced expression
for the Grassmann permutation taking 
$1,\dots,\ell(\lambda)$ to $\ell(\mu)+1,\dots,\ell(\mu)+\ell(\lambda)$
and $\ell(\lambda)+1,\dots,\ell(\lambda)+\ell(\mu)$
to $1,\dots,\ell(\mu)$. 
As any two reduced expressions for a Grassmann permutation are equivalent by 
commuting braid relations,
the resulting morphism is well defined by the interchange law. 
The morphism $c_{\lambda,\mu}$ is an isomorphism
 since the positive crossing $\begin{tikzpicture}[baseline=-1.5mm,scale=.6]
	\draw[-,thick] (0.3,-.35) to (-.3,.43);
	\draw[-,line width=5pt,white] (-0.3,-.35) to (.3,.43);
	\draw[-,thick] (-0.3,-.3) to (.3,.4);
       \node at (-0.32,-.53) {$\scriptstyle a$};
        \node at (0.32,-.54) {$\scriptstyle b$};
\end{tikzpicture}$
is invertible; its two-sided inverse is
$\begin{tikzpicture}[baseline=-1.5mm,scale=.6]
	\draw[-,thick] (-0.3,-.35) to (.3,.43);
	\draw[-,line width=5pt,white] (0.3,-.35) to (-.3,.43);
	\draw[-,thick] (0.3,-.3) to (-.3,.4);
       \node at (-0.32,-.53) {$\scriptstyle b$};
        \node at (0.32,-.54) {$\scriptstyle a$};
\end{tikzpicture}$
 according to the last two relations in \cref{swallows-symmetric}. Naturality follows from \cref{sliders}.

Next, we show that there a strict $\Z[q,q^{-1}]$-linear monoidal 
functor $F:\qSchur' \rightarrow \qSchur$
taking $(r) \mapsto (r)$ and the generating morphisms 
of $\qSchur'$ to the morphisms in $\qSchur$ represented by the
same
diagrams. To prove this, we just need to check relations:
the relations \cref{secondone,zeroforks} are trivial to check in $\qSchur$,
\cref{assrel} follows from \cref{kitchen,cabinet}, and \cref{mergesplit}
follows from \cref{laura,rider}. 
By definition, the functor $F$ takes the positive crossing in $\qSchur'$ to the positive crossing in $\qSchur$, so the identity \cref{tourists} in $\qSchur$ follows by applying $F$ to \cref{thickcrossing}.
We have observed already that the negative crossing in $\qSchur'$ is 
the two-sided inverse of the positive crossing 
in $\qSchur'$, hence,
$$
F\Big(\begin{tikzpicture}[baseline=-1.5mm,scale=.6]
	\draw[-,thick] (-0.3,-.35) to (.3,.43);
	\draw[-,line width=5pt,white] (0.3,-.35) to (-.3,.43);
	\draw[-,thick] (0.3,-.3) to (-.3,.4);
       \node at (-0.32,-.53) {$\scriptstyle a$};
        \node at (0.32,-.54) {$\scriptstyle b$};
\end{tikzpicture}\Big)=
\begin{tikzpicture}[baseline=-1.5mm,scale=.6]
	\draw[-,thick] (-0.3,-.35) to (.3,.43);
	\draw[-,line width=5pt,white] (0.3,-.35) to (-.3,.43);
	\draw[-,thick] (0.3,-.3) to (-.3,.4);
       \node at (-0.32,-.53) {$\scriptstyle a$};
        \node at (0.32,-.54) {$\scriptstyle b$};
\end{tikzpicture}
$$
since the negative crossing in $\qSchur$ is also the inverse of the positive crossing by the original definition.
To prove that \cref{visas} holds in $\qSchur$,
the first equality follows by applying $F$ to 
\cref{tireddog}. The second equality follows by applying the bar involution to the second equality of
\cref{tourists}, remembering that this fixes splits and merges in $\qSchur$ 
thanks to \cref{pigs}.

For any $A \in \Mat{\lambda}{\mu}$, let $\xi_A'$ be the morphism in $\qSchur'$ obtained by taking the (reduced) double coset diagram for $A$, replacing all crossings by positive crossings, and
 interpreting the result as a morphism by composing generators as the diagram suggests.
 The resulting morphism is well defined independent of the choices
 made when doing this. 
For the split of $\ell(\mu)$ strings at the bottom and the merge of $\ell(\lambda)$ strings at the top, this depends on \cref{assrel} as explained in the comments 
after the statement of the theorem.
For the permutation of 
$\ell(\lambda)\ell(\mu)$ strings in the middle, one needs to draw the diagram according to a choice of a reduced expression, but the resulting morphism is independent of this by \cref{braid}.
We are ready to prove (b) by showing that $F(\xi_A') = \xi_A$. This follows from the factorization of $\xi_A$ explained in \cref{ducks}, together with \cref{chambers,kitchen,cabinet}, since these results show
$\xi_A$ can be obtained from merges, splits and positive crossings in exactly the same way as $\xi_A'$ is obtained from the correspnding generating morphisms for $\qSchur'$.

Now we can prove that $F$ is an isomorphism. 
It is clear that it defines a bijection between the object sets of $\qSchur'$ and $\qSchur$ (both are identified with $\Lambda$).
Since the morphisms $\xi_A (A \in \Mat{\lambda}{\mu})$ 
form a basis for $\Hom_{\qSchur}(\mu,\lambda)$ by the definition of $\qSchur$, we deduce using the previous paragraph 
that $F$ is full. 
It just remains to show that it is faithful, which we do by proving that the morphisms
$\xi_A'$ $(A \in \Mat{\lambda}{\mu})$ span
$\Hom_{\qSchur'}(\mu,\lambda)$ as a $\Z[q,q^{-1}]$-module.
This follows from our next claim,
since the merge and split morphisms $f$ described in the claim
for all $\lambda,\lambda'$ generate $\qSchur'$
as a $\Z[q,q^{-1}]$-linear category by \cref{tourists}.

\vspace{1mm}

\noindent
{\bf Claim.} {\em For any $\lambda,\lambda',\mu \in \Lambda$, $A \in \Mat{\lambda}{\mu}$
and $f:\lambda\rightarrow \lambda'$ that consists of a merge or split of 2 strings tensored on the left and/or right by some identity morphisms,
the composition $f \circ \xi_A'$ is a $\Z[q,q^{-1}]$-linear combination of the
morphisms $\xi_A'\:(A \in \Mat{\lambda}{\mu})$.}

\vspace{1mm}
\noindent
To prove the claim, there are two cases:
\begin{itemize}
\item
Suppose first that $f$ has a merge of two strings 
connecting to the $i$th
and $(i+1)$th thick strings at the top of $\xi_A'$. The double coset diagram of $A$ has a merge of $r$ strings
at its $i$th vertex and merge of $s$ strings at its $(i+1)$th vertex.
We use \cref{assrel} to convert $f \circ \xi_A'$ into a  diagram which has a merge of $(r+s)$ strings at its $i$th vertex.
For example:
\begin{equation}\label{spots}
\begin{tikzpicture}[anchorbase]
	\draw[-,line width=2pt] (0.08,.3) to (0.08,-0.04);
	\draw[-,line width=1pt] (0.28,-.4) to (0.08,0);
	\draw[-,line width=1pt] (-0.12,-.4) to (0.08,0);
	\draw[-,thin] (-0.127,.-.39) to (-0.28,-.7);
	\draw[-,thin] (-0.118,.-.39) to (-0.118,-.7);
	\draw[-,thin] (-0.109,.-.39) to (0.04,-.7);
	\draw[-,line width=.5pt] (.28,.-.39) to (.15,-.7);
	\draw[-,line width=.5pt] (0.285,.-.39) to (0.4,-.7);
\end{tikzpicture}=
\begin{tikzpicture}[anchorbase]
	\draw[-,line width=2pt] (0.08,.3) to (0.08,-0.06);
	\draw[-,thin] (.053,.-.04) to (-0.32,-.7);
	\draw[-,thin] (.067,.-.04) to (-0.14,-.7);
	\draw[-,thin] (.081,.-.04) to (0.08,-.7);
	\draw[-,thin] (.095,.-.04) to (0.3,-.7);
	\draw[-,thin] (.108,.-.04) to (0.52,-.7);
\end{tikzpicture}.
\end{equation}
The permutation arising in the middle section of the resulting diagram is not necessarily reduced, but it can be
converted to a scalar multiple of some $\xi_B'$ 
using
the relations \cref{assrel,mergesplit,swallows-symmetric,braid}.
\item
Now suppose that $f$ has a split connecting to the $i$th
vertex at the top of the double coset diagram of $A$. 
Say this vertex in the double coset diagram is part of 
an $n$-fold merge.
Using \cref{assrel,mergesplit,sliders}, we
rewrite the composition of the split in $f$ and this merge in $\xi_A'$ as a
sum of other $\xi_B'$. For example:
\begin{equation}\label{pots}
\begin{tikzpicture}[anchorbase,scale=1.2]
	\draw[-,line width=1.2pt] (-0.118,-.4) to (-0.118,-.18);
	\draw[-,thin] (-0.127,.-.39) to (-0.28,-.7);
	\draw[-,thin] (-0.118,.-.39) to (-0.118,-.7);
	\draw[-,thin] (-0.109,.-.39) to (0.04,-.7);
	\draw[-,line width=.8pt] (-.116,-.21) to (-.3,.1);
	\draw[-,line width=.8pt] (-0.12,-.21) to (0.05,.1);
\node at (-0.7,-.1) {$\scriptstyle f$};
\node at (-0.7,-0.54) {$\scriptstyle g$};
\end{tikzpicture}=
\sum\:
\begin{tikzpicture}[anchorbase,scale=1.3]
	\draw[-,thin] (-0.09,.03) to (0.24,-.7);
	\draw[-,line width=1.5pt,white] (0.09,.03) to (0,-.7);
	\draw[-,thin] (-0.09,.03) to (0,-.7);
	\draw[-,line width=1.5pt,white] (0.09,.03) to (-0.24,-.7);
	\draw[-,thin] (0.09,.03) to (-0.24,-.7);
	\draw[-,thin] (0.09,.03) to (0,-.7);
	\draw[-,thin] (0.09,.03) to (0.24,-.7);
	\draw[-,thin] (-0.09,.03) to (-0.24,-.7);
	\draw[-,line width=.5pt] (0.089,.01) to (0.089,.04);
	\draw[-,line width=.5pt] (-0.089,.01) to (-0.089,.04);
\end{tikzpicture}.
\end{equation}
Then compose these diagrams with the remainder of the diagram,
using \cref{sliders} then \cref{assrel} again
to commute the splits at the bottom of this part of the
resulting diagrams downwards past the positive crossings in $\xi_A'$. 
\end{itemize}

All that is left is to prove (c) and (d). Part (c) follows because the bar involution on $\qSchur$ fixes merges and splits
and interchanges positive and negative crossings by \cref{pigs}; it obviusly fixes the other two generating morphisms
$\begin{tikzpicture}[anchorbase]
\draw[line width=0pt] (0,-.3) to (0,0);
\spot{0,0};
%\node at (0,-.42) {$\scriptstyle 0$};
\end{tikzpicture}$ and
$\begin{tikzpicture}[anchorbase]
\draw[line width=0pt] (0,.3) to (0,0);
\spot{0,0};
%\node at (0,.42) {$\scriptstyle 0$};
\end{tikzpicture}$.
Part (d) follows using (b) because $\T$ takes $\xi_A$ to $\xi_{A^\tr}$.
\end{proof}

\begin{corollary}\label{1984}
In $\qSchur$, we have that
\begin{equation*}
\begin{tikzpicture}[anchorbase,scale=0.8]
	\draw[-,thick] (0.28,0) to[out=90,in=-90] (-0.28,.6);
	\draw[-,line width=4pt,white] (-0.28,0) to[out=90,in=-90] (0.28,.6);
	\draw[-,thick] (0.28,-.6) to[out=90,in=-90] (-0.28,0);
	\draw[-,line width=4pt,white] (-0.28,-.6) to[out=90,in=-90] (0.28,0);
	\draw[-,thick] (-0.28,-.6) to[out=90,in=-90] (0.28,0);
	\draw[-,thick] (-0.28,0) to[out=90,in=-90] (0.28,.6);
        \node at (0.3,-.75) {$\scriptstyle b$};
        \node at (-0.3,-.75) {$\scriptstyle a$};
\end{tikzpicture}
=
\sum_{s=0}^{\min(a,b)}
q^{-s(s-1)/2+r(a+b-2s)}(q^{-1}-q)^s [s]_q^!
\begin{tikzpicture}[anchorbase,scale=1]
	\draw[-,thick] (0.58,0) to (0.58,.2) to (.02,.8) to (.02,1);
	\draw[-,line width=4pt,white] (0.02,0) to (0.02,.2) to (.58,.8) to (.58,1);
	\draw[-,thick] (0.02,0) to (0.02,.2) to (.58,.8) to (.58,1);
	\draw[-,thin] (0,0) to (0,1);
	\draw[-,thin] (0.6,0) to (0.6,1);
        \node at (0,-.1) {$\scriptstyle a$};
        \node at (0.6,-.1) {$\scriptstyle b$};
        \node at (-0.35,.5) {$\scriptstyle a-s$};
        \node at (.95,.5) {$\scriptstyle b-s$};
\end{tikzpicture}.
\end{equation*}
\end{corollary}

\begin{proof}
This is a translation of \cref{tolerate} into the graphical
description of $\qSchur$ 
provided by the theorem.
\end{proof}

We also have the following theorem, which gives an alternative presentation for $\qSchur$ with fewer generators and relations.

\begin{theorem}\label{altpres}
The strict $\Z[q,q^{-1}]$-monoidal category $\qSchur$ is
generated by the objects $(r)$ for $r \geq 0$ 
and the morphisms
$\begin{tikzpicture}[anchorbase,scale=1]
\draw[line width=0pt] (0,-.3) to (0,0);
\spot{0,0};
%\node at (0,-.42) {$\scriptstyle 0$};
\end{tikzpicture}\ $,
$\begin{tikzpicture}[anchorbase,scale=1]
\draw[line width=0pt] (0,.3) to (0,0);
\spot{0,0};
%\node at (0,.42) {$\scriptstyle 0$};
\end{tikzpicture}\ $,
$\begin{tikzpicture}[anchorbase,scale=.6]
	\draw[-,line width=1pt] (0.28,-.3) to (0.08,0.04);
	\draw[-,line width=1pt] (-0.12,-.3) to (0.08,0.04);
	\draw[-,line width=2pt] (0.08,.4) to (0.08,0);
        \node at (-0.22,-.43) {$\scriptstyle a$};
        \node at (0.35,-.43) {$\scriptstyle b$};
\end{tikzpicture}$ and
$\begin{tikzpicture}[baseline=0mm,scale=.6]
	\draw[-,line width=2pt] (0.08,-.3) to (0.08,0.04);
	\draw[-,line width=1pt] (0.28,.4) to (0.08,0);
	\draw[-,line width=1pt] (-0.12,.4) to (0.08,0);
        \node at (-0.22,.53) {$\scriptstyle a$};
        \node at (0.36,.55) {$\scriptstyle b$};
\end{tikzpicture}$
subject only to the relations \cref{secondone,zeroforks}
for $a,b \geq 0$, \cref{assrel}
for $a,b,c > 0$,
and one of the two square-switch relations
from \cref{squareswitch} for all $a,b,c,d\geq 0$
with $d\leq a$ and $c \leq b+d$. 
\end{theorem}

\begin{proof}
Let $\qSchur'$ be the strict $\Z[q,q^{-1}]$-linear monoidal category defined by the new 
presentation in the statement of the theorem, assuming for clarity that the {\em second} relation in \cref{squareswitch} is the chosen one.
All of the relations of $\qSchur'$ hold in $\qSchur$ 
thanks to \cref{jonsquare2}.
So there is a strict $\Z[q,q^{-1}]$-linear monoidal functor
$F:\qSchur'\rightarrow \qSchur$ taking $(r) \mapsto (r)$
and the generating morphisms for $\qSchur'$ to the morphisms represented by the same diagrams in $\qSchur$. In the next paragraph, we show that $F$ is an isomorphism, proving the theorem for this choice
of square-switch.
The proof of the theorem if one instead chooses the first square-switch relation from
\cref{squareswitch}, i.e., the one that is known to hold in $\qSchur$ by \cref{jonsquare}, is very similar---one simply needs to rotate all calculations in a vertical axis.

To prove that $F$ is an isomorphism, we use the presentation from \cref{mainpres} to
construct a two-sided inverse $G:\qSchur\rightarrow \qSchur'$.
This is defined on objects so that $(r) \mapsto (r)$
and, on generating morphisms, it maps the positive crossing to
\begin{equation}\label{dogsandcats}
\begin{tikzpicture}[anchorbase,scale=1.2]
	\draw[-,line width=1pt] (0.3,-.3) to (-.3,.4);
	\draw[-,line width=4pt,white] (-0.3,-.3) to (.3,.4);
	\draw[-,line width=1pt] (-0.3,-.3) to (.3,.4);
        \node at (0.3,-.42) {$\scriptstyle b$};
        \node at (-0.3,-.42) {$\scriptstyle a$};
\end{tikzpicture} :=
\sum_{s=0}^{\min(a,b)}
(-q)^{s}
\begin{tikzpicture}[anchorbase,scale=1]
	\draw[-,line width=1.2pt] (0,0) to (0,1);
	\draw[-,thick] (-0.8,0) to (-0.8,.2) to (-.03,.4) to (-.03,.6)
        to (-.8,.8) to (-.8,1);
	\draw[-,thin] (-0.82,0) to (-0.82,1);
        \node at (-0.81,-.1) {$\scriptstyle a$};
        \node at (0,-.1) {$\scriptstyle b$};
        \node at (-0.4,.9) {$\scriptstyle b-s$};
        \node at (-0.4,.13) {$\scriptstyle a-s$};
\end{tikzpicture} \in \Hom_{\qSchur'}((a)\star (b) \rightarrow (b)\star (a))
\end{equation}
and the other generating morphisms for $\qSchur$ to the morphisms represented by the same diagrams in $\qSchur'$.
That $G$ is indeed a two-sided inverse of $F$ follows using \cref{tourists}.
It remains to show that $G$ is well defined, which is another relations check.
The relations \cref{secondone,zeroforks} hold in $\qSchur'$ by its definition.
If one or more of $a,b,c$ is zero, the relations \cref{assrel} 
follow easily from \cref{secondone,zeroforks}, so the relations \cref{assrel} also hold in $\qSchur'$ for all $a,b,c\geq 0$.
The first relation from \cref{mergesplit} follows from the chosen
square-switch relation taking $b=0$ and $c=d$.
It remains to show that the second relation from \cref{mergesplit} holds in $\qSchur'$
using only \cref{secondone,zeroforks,assrel} and square-switch.
This is explained in the appendix; see part (a) of the corollary there.
\end{proof}

%% file: s7-basis.tex
\section{A straightening formula for codeterminants}\label{s7-basis}

\begin{definition}\label{bqh}
Let $\mathcal{O}$ be a commutative Noetherian ring
and $K = \bigoplus_{\lambda,\mu \in \Comp} 1_\lambda K 1_\mu$
be a locally unital $\mathcal{O}$-algebra
with (mutually orthogonal) distinguished idempotents
$1_\lambda\:(\lambda \in \Comp)$ for some index
set $\Comp$.
We say that $K$ is a {\em based quasi-hereditary algebra} with {\em weight poset} $\Par$
if we are given 
a subset $\Par \subseteq \Comp$,
an upper finite partial order $\leq$ on $\Par$,
and finite sets $X(\lambda,\kappa) \subset 1_\lambda K 1_\kappa$ and $Y(\kappa,\lambda) \subset 1_\kappa K 1_\lambda$ for $\lambda \in \Comp,\kappa \in \Par$, such that the following axioms hold:
\begin{itemize}
%\item 
%The
%sets $\bigcup_{\kappa \in \Par} X(\lambda,\kappa)$
%and $\bigcup_{\kappa \in \Comp} Y(\kappa,\lambda)$
%are finite
%for each 
%$\lambda \in \Comp$. 
\item The products $xy$ for $(x,y) \in \bigcup_{\lambda, \mu \in \Comp}
\bigcup_{\kappa \in \Par}
X(\lambda,\kappa)\times Y(\kappa,\mu)$
give a basis for $K$ as a free $\mathcal{O}$-module.
We refer to this as the {\em triangular basis}.
\item For $\lambda,\mu \in \Par$, 
we have that 
$X(\lambda,\mu)\neq\varnothing\Rightarrow \lambda \leq \mu$,
$Y(\lambda,\mu)\neq\varnothing\Rightarrow \lambda \geq \mu$, and
$X(\lambda,\lambda) = Y(\lambda,\lambda) = \{1_\lambda\}$.
\end{itemize}
We say that it is a {\em symmetrically-based quasi-hereditary algebra} if in addition there is an algebra anti-involution $\T:K \rightarrow K$
such that 
$Y(\kappa,\lambda) = \T (X(\lambda,\kappa))$
for all $\lambda \in \Comp$ and $\kappa\in\Par$
(in this case, there is no need to specify $Y(\kappa,\lambda)$ in the first place).
\end{definition}

\begin{remark}
When $\mathcal{O}$ is a field, \cref{bqh} is \cite[Def.~5.1]{BS}.
When the set $\Comp$ is finite, it
is a simplified version of the definition of based quasi-hereditary algebra given in \cite{KM}.
In that case, as explained in detail in \cite{KM}, 
$K$ is also a standardly full-based algebra in the sense of \cite{DR}, and a split quasi-hereditary algebra in the sense of \cite{CPS2}.
In the symmetrically-based case, 
$K$ is a cellular algebra in the sense of \cite{GL},
and
when $K$ is the path algebra of an $\mathcal{O}$-linear category
$\cC$ with object set $\Comp$,
\cref{bqh} is equivalent to
$\cC$ being a strictly object-adapted cellular category
in the sense of \cite[Def.~2.1]{EL} (the opposite partial order is used there). The
far-reaching consequences for the representation theory of $K$
 are well known, and are discussed in these references.
\end{remark}

For the remainder of the section, $\schur$ is the path algebra
\begin{equation}\label{norfolk}
\schur := \bigoplus_{\lambda,\mu \in \Lambda} 
\Hom_{\qSchur}(\mu,\lambda)
\end{equation}
of the $q$-Schur category with $0$-strings. This is a locally unital
$\Z[q,q^{-1}]$-algebra with the distinguished system $\{1_\lambda\:|\:\lambda \in \Lambda\}$
of mutually orthogonal idempotents coming from the identity endomorphisms of the objects of $\qSchur$.
Recall the set $\Row(\lambda,\mu)$ of {\em row tableaux} of shape $\mu$ and 
content $\lambda$ from \cref{s2-combinatorics}, and the bijection
$A:\Row(\lambda,\mu)\stackrel{\sim}{\rightarrow} \Mat{\lambda}{\mu}$ from \cref{Amap}.
We start now to index the standard and canonical bases by the sets $\Row(\lambda,\mu)$ instead of $\Mat{\lambda}{\mu}$,
introducing the shorthands
\begin{align}\label{shorthands}
\varphi_P &:= \xi_{A(P)},&
\beta_P &:= \theta_{A(P)}
\end{align}
for $P \in \Row(\lambda,\mu)$.
For a partition $\kappa$,
let $\Std(\lambda,\kappa)$ be the usual set of
{\em semistandard tableau of
shape $\kappa$ and content $\lambda$}, that is, is the subset of $\Row(\lambda,\kappa)$
consisting of the row tableaux of shape $\kappa$ and 
content $\lambda$
whose entries are also strictly increasing down columns.

\begin{lemma}\label{cella} 
For $\lambda,\mu \vDash r$,
the $\Z[q,q^{-1}]$-module $1_\lambda \schur 1_\mu$ 
is spanned by the
products $\varphi_{P}\T(\varphi_{Q})$
for $P \in \Row(\lambda,\kappa), Q
\in \Row(\mu,\kappa)$, where $\kappa$ is the dominant conjugate of $\mu$.
\end{lemma}

\begin{proof}
The dominant conjugate $\kappa$ of $\mu$ 
is the unique partition
whose parts are a permutation of 
the non-zero
parts of $\mu$.
Using a 
morphism of the form 
$\tau_{w;\mu}$ from \cref{genperm}, we deduce 
$\mu \cong \kappa$ in $\qSchur$.
Consequently, any 
element of $1_\lambda \schur 1_\mu = \Hom_{\qSchur}(\mu,\lambda)$ is a morphism
which factors through
$\kappa$.
Since the morphisms $\varphi_{P}$ for $P \in \Row(\lambda,\kappa)$ 
give the standard basis for $1_\lambda \schur 1_\kappa = \Hom_{\qSchur}(\kappa,\lambda)$
and the morphisms $\T(\varphi_Q)$ for $Q \in \Row(\mu,\kappa)$
give
the standard basis for $1_\kappa \schur 1_\mu = \Hom_{\qSchur}(\mu,\kappa)$, we deduce
that the products $\varphi_{P}\T(\varphi_Q)$ span
$1_\lambda \schur 1_\mu$.
\end{proof}

Now we come to the main combinatorial lemma.
To formulate it, we use certain lexicographic total orders on 
tableaux and partitions.
On partitions, $\geq_{\operatorname{lex}}$ is just the usual lexicographical ordering; it is
a refinement of the dominance ordering on partitions into a
total order. To define the required ordering 
$\leq_{\operatorname{lex}}$ 
on tableaux of the same shape, given any tableau $T$, we let $\ZigzagDownRev(T)$
be the sequence obtained by reading its entries in order from right
to left along rows, starting with the top row.
Then we declare that $S \leq_{\operatorname{lex}} T$ if and only if
$\ZigzagDownRev(S) \leq_{\operatorname{lex}} \ZigzagDownRev(T)$ in the lexicographic
ordering on sequences.

\begin{lemma}\label{cellb}
For $\lambda \vDash r$, $\kappa \vdash r$ and $P \in
\Row(\lambda,\kappa)$ which is {\em not} semistandard,
$\varphi_{P}$ can be written as a $\Z[q,q^{-1}]$-linear combination of the
following elements:
\begin{itemize}
\item
$\varphi_{S}$ for $S \in
\Row(\lambda,\kappa)$
with $S <_{\operatorname{lex}} P$;
\item 
$\varphi_{P'}\T(\varphi_{Q'})$
for $P' \in \Row(\lambda,\kappa')$ and $Q' \in \Row(\kappa,\kappa')$
of shape $\kappa' \vdash r$ with $\kappa' >_{\operatorname{lex}} \kappa$.
\end{itemize}
\end{lemma}

\begin{proof}
Take $P$ as in the statement.
Since $P$ is not semistandard,
we may choose $a\geq 1$ and
$0 \leq m < n \leq \kappa_{a+1}$
so that the entries of $P$ in rows $a$ and $(a+1)$ look like
$$
\begin{array}{ccccccccccccccccc}
i_1 &\leq& \cdots&\leq &i_{m}& < &i_{m+1} &\leq &\cdots &\leq &i_{n}&\leq &i_{n+1}&\leq&\cdots&\leq &i_{\kappa_a}\\
&&&&&&\veebar&&&&\\
j_1 &\leq& \cdots&\leq &j_{m}& \leq &j_{m+1} &
=&\cdots &= &j_{n} &< &j_{n+1} &\leq&\cdots &\leq &j_{\kappa_{a+1}}.\\
\end{array}
$$
%We assume for ease of exposition that $m>0$ and $n < \kappa_{a+1}$.
Let $U$ be the row tableau which is
identical to $P$ everywhere except in rows $a$ and $(a+1)$,
which are replaced by {\em three} (possibly empty) rows
%\footnote{
%If $m=0$ and/or $n =
%\kappa_{a+1}$, the tableau $U$ is constructed similarly but one only inserts
%one or two rows, not three, and the subsequent definition of the tableau $V$ needs to be modified accordingly.
\iffalse
The tableau $V$ needs to be modified accordingly: 
\begin{itemize}
\item
If $m=0$ and $n <
\kappa_{a+1}$ then $V$ has all entries $b$ on rows $b \neq a+1$, 
and entries $a^n\ (a+1)^{\kappa_{a+1}-n}$ on row $(a+1)$.
\item if $m > 0$ and $n = \kappa_{a+1}$ then $V$ has all entries $b$
  on rows $b \neq a$, and entries $a^m\  (a+1)^{\kappa_a-m}$ on row $a$.
\item
If both
$m=0$ and $n=\kappa_{a+1}$, then $V$ has all entries $b$ on rows
$b < a+1$ and all entries $b-1$ on rows $b > a$.
\end{itemize}
One also needs to adjust the definition of $T_{v,w}$ in a similar way.\fi
%}
as in the diagram:
$$
\begin{array}{lllllllllll}
i_1 &\leq& \cdots&\leq &i_{m}\\
j_1 &\leq &\cdots &\leq &j_n &\leq &i_{m+1} &\leq& \cdots& \leq i_{\kappa_a}\\
j_{n+1} &\leq &\cdots &\leq &j_{\kappa_{a+1}}.
\end{array}
$$
Let $\mu$ be the shape of the tableau $U$.
Let $V$
be the row tableau of shape $\kappa$ and content $\mu$ with all entries on row $b$ equal to $b$ for $b < a$, 
entries $a^{m}\ (a+1)^{\kappa_{a}-m}$ 
on row $a$,
entries $(a+1)^{n}\ (a+2)^{\kappa_{a+1}-n}$ 
on row
$a+1$, and all entries on
row $b$ equal to $b+1$ for $b > a+1$.
Expanding in terms of the standard basis, we have that
\begin{equation}\label{atlantic}
\varphi_{U} \varphi_{V}
=
\sum_{S \in \Row(\lambda,\kappa)} g_S \varphi_{S}
\end{equation}
for coefficients $g_S \in \Z[q,q^{-1}]$. 
We claim that $g_S = 0$
unless $S \leq_{\operatorname{lex}} P$ and that
$g_P=1$.
This suffices to prove the lemma. 
Indeed, assuming the claim, 
we rearrange \cref{atlantic} to obtain
$$
\varphi_{P} = \varphi_{U} \varphi_{V}
- 
\sum_{S <_{\operatorname{lex}} P} g_S \varphi_{S}.
$$
The second term on the right hand side is already of the desired form.
To understand the first term, note that
the first $(a-1)$ rows of $U$ are of lengths
$\kappa_1,\dots,\kappa_{a-1}$, 
and it also has a row of length
$\kappa_a+n-m > \kappa_a$.
Consequently, the dominant conjugate of the shape $\mu$ of $U$ is greater than $\kappa$
in the ordering $>_{\mathrm{lex}}$.
So, by \cref{cella}, the first term can be rewritten as a sum $\varphi_{P'}\T(\varphi_{Q'})$
for row tableaux $P', Q'$ of dominant shape $\kappa' >_{\operatorname{lex}} \kappa$.
This is also of the desired form.

It just remains to prove the claim. 
Take $S \in \Row(\lambda,\kappa)$.
Recalling that 
$x_S = x_{\zigzag(S), \bi^\kappa}$ and $x_U = x_{\zigzag(U), \bi^\mu}$, the definition of multiplication in $\schur$ gives that
$g_S$ is the $x_{\zigzag(U), \bi^\mu} \otimes x_{\zigzag(V),
  \bi^\kappa}$-coefficient
of 
$$
\Delta(x_{\zigzag(S), \bi^\kappa})
= 
\sum_{\bk \in \I_\mu} x_{\zigzag(S), \bk} \otimes x_{\bk,\bi^\kappa}
$$
when expanded in terms of the normally-ordered monomial basis.
To straighten $x_{\bk, \bi^\kappa}$ into normal order, we only need
the fourth relation from \cref{r1}, and see that this coefficient is non-zero if and only if 
$\bk = \ZigzagDownRev(R)$ for a 
tableau $R$ of shape $\kappa$
(not necessarily a row tableau)
that is obtained from $V$ by shuffling
entries within rows $a$ and $(a+1)$.
Moreover, the coefficient is $1$ in the case that $R = V$.
To complete the proof, we show for such a tableau $R$ that
the $x_{\zigzag(U), \bi^\mu}$-coefficient
of $x_{\zigzag(S), \zigzag(R)}$ is zero unless $S
\leq_{\operatorname{lex}} P$, it is 1 if $S = P$ and $R
= V$, and it is zero if $S = P$ and $R \neq V$.
Suppose the entries in rows $a$ and $(a+1)$ 
of $S$ are
$$
\begin{array}{lllllll}
i_1' &\leq&\cdots&\leq& i_{\kappa_a}'\\
j_1' &\leq&\cdots&\leq& j_{\kappa_{a+1}}'.
\end{array}
$$
In order to convert the monomial $x_{\zigzag(S), \zigzag(R)}$ into
normal order, we must apply the relations to commute products
of the form $x_{i_c',a+1} x_{i_b',a}$ for $1 \leq b < c
\leq \kappa_a$ or $x_{j_c',a+2} x_{j_b',a+1}$ for $1 \leq b < c \leq
\kappa_{a+1}$.
This can be done using the second and third relations from \cref{r1}.
We deduce that $$
x_{\zigzag(S), \zigzag(R)}
= 
\sum_{\substack{v \in (S_{\kappa_r} / S_m \times
S_{\kappa_a-m})_{\operatorname{min}}\\
w \in (S_{\kappa_{a+1}} / S_n \times
S_{\kappa_{a+1}-n})_{\operatorname{min}}}}
g_{v,w} 
x_{\zigzag(T_{v,w}),\bi^\mu}
$$
for some scalars $g_{v,w} \in \Z[q,q^{-1}]$
with $g_{1,1} = \delta_{R,V}$,
where 
$T_{v,w}$ is the tableau of shape $\mu$
obtained from $S$ by replacing its rows $a$ and $(a+1)$
by three
rows according to the diagram:
$$
\begin{array}{lllllllllll}
i'_{v(1)} &\leq& \cdots&\leq &i'_{v(m)}\\
j'_{w(1)} &\leq &\cdots &\leq &j'_{w(n)} &i'_{v(m+1)} &\leq& \cdots& \leq i'_{v(\kappa_a)}\\
j'_{w(n+1)} &\leq &\cdots &\leq &j'_{w(\kappa_{a+1})}.
\end{array}
$$
In particular, if $S = P$ then $T_{1,1}=U$.
Using the fourth relation, the
$x_{\zigzag(U),\bi^\mu}$-coefficient of $x_{\zigzag(T_{v,w}),\bi^\mu}$
is non-zero if and only if 
$T_{v,w}\sim_{\operatorname{row}} U$, i.e., they have the same entries in
each row counted with multiplicity, and the coefficient is $1$ if $T_{v,w}=U$.
Now it remains to check that 
\begin{itemize}
\item
$T_{v,w} \sim_{\operatorname{row}}U\Rightarrow S
\leq_{\operatorname{lex}} P$;
\item
$T_{v,w} \sim_{\operatorname{row}}U$ and $S = P
\Rightarrow (v,w)=(1,1)$.
\end{itemize}
To see this, suppose that $T_{v,w} \sim_{\operatorname{row}} U$. All rows of $S$
are clearly equal to the corresponding rows of $P$ except perhaps for rows $a$ and $(a+1)$.
Also the sequences $i'_{v(1)} \leq \cdots \leq i'_{v(m)}$ 
and $j_{w(n+1)}' \leq \cdots \leq j'_{w(\kappa_{a+1})}$
are equal to $i_1 \leq \cdots\leq i_m$ and $j_{n+1} \leq \cdots \leq
j_{\kappa_{a+1}}$, respectively.
So the $a$th row of $S$ is obtained by taking all of the entries in
the $a$th row
of $U$ together with $\kappa_a - m$ entries from row $(a+1)$,
and row $(a+1)$ of $S$ is obtained by taking all of the
remaining entries from row $(a+1)$ of $U$ plus all of the
entries in row $(a+2)$. It follows that $S
\leq_{\operatorname{lex}} P$. Moreover, if $S = P$, then 
we must have that $v = w = 1$ due to the assumptions that $i_m < i_{m+1}$ and
$j_n < j_{n+1}$.
\end{proof}

\begin{theorem}\label{secrets}
The path algebra $\schur = \bigoplus_{\lambda,\mu \in \Lambda} 1_\lambda \schur 1_\mu$ of $\qSchur$ is a symmetrically-based quasi-hereditary algebra.
The required data from \cref{bqh} is as follows:
\begin{itemize}
\item
The weight poset is the set $\Par \subset \Comp$ of  partitions ordered by the dominance ordering.
\item
The anti-involution
$\T:\schur\rightarrow\schur$ is the transposition map arising from \cref{starinv}.
\item
$X(\lambda,\kappa)
=\{\varphi_P\:|\: P \in \Std(\lambda,\kappa)\}$.
\end{itemize}
In particular, for $\lambda,\mu \vDash r$, the {\em codeterminants}
\begin{equation}\label{greenbasis}
\left\{\varphi_{P}\T(\varphi_Q)\:\big|\:\kappa \vdash r,
P \in \Std(\lambda,\kappa),
Q \in \Std(\mu,\kappa)
\right\}
\end{equation}
give a basis for $1_\lambda \schur 1_\mu$ as a free $\Z[q,q^{-1}]$-module.
\end{theorem}

\begin{proof}
The second axiom follows because there is a unique
semistandard tableau of shape and content $\kappa$, and there only exist
semistandard tableaux of shape $\kappa'$ and content $\kappa$ if $\kappa \leq \kappa'$.
It just remains to check for $\lambda,\mu\vDash r$ that the set \cref{greenbasis},
is a basis for $1_\lambda \schur 1_\mu$ as a free $\Z[q,q^{-1}]$-module.
By the original definition,
$1_\lambda \schur 1_\mu$ is a free
$\Z[q,q^{-1}]$-module with basis labelled by $\Mat{\lambda}{\mu}$.
It is well known that $|\Mat{\lambda}{\mu}| = \sum_{\kappa \vdash r}|\Std(\lambda,\kappa)\times\Std(\mu,\kappa)|$, e.g., this follows from the Robinson-Schensted-Knuth-type correspondence in \cref{birds} below.
So the set \cref{greenbasis} is of size $ \leq \operatorname{rank} 1_\lambda \schur 1_\mu$.
It remains to show that the set \cref{greenbasis} spans
$1_\lambda \schur 1_\mu$ as a $\Z[q,q^{-1}]$-module.

By \cref{cella},
the elements
$\varphi_P\T(\varphi_Q)$ 
for $P \in \Row(\lambda,\kappa), Q \in
\Row(\mu,\kappa)$ and $\kappa \vdash r$ span $1_\lambda \schur 1_\mu$.
To complete the proof, we show by induction on the
lexicographic orderings that any such 
$\varphi_P\T(\varphi_Q)$ can be written as a $\Z[q,q^{-1}]$-linear combination
of $\varphi_{P'}\T(\varphi_{Q'})$ such that either $P' \in \Std(\lambda,\kappa),
Q' \in \Std(\mu,\kappa)$ with $P' \leq_{\operatorname{lex}} P, Q'
\leq_{\operatorname{lex}} Q$, or $P' \in \Std(\lambda,\kappa'), Q' \in
\Std(\mu,\kappa')$ for $\kappa' >_{\operatorname{lex}} \kappa$.
Applying $\T$ if necessary, we may assume that $P$ is not semistandard.
Applying \cref{cellb}, we see that $\varphi_P\T(\varphi_Q)$
is a linear combination of elements
$\varphi_S\T(\varphi_Q)$ for $S \in \Row(\lambda,\kappa)$ with $S <_{\operatorname{lex}} P$,
and $\varphi_{P'} \T(\varphi_{Q'}) \T(\varphi_{Q})
= \varphi_{P'}\T(\varphi_Q \varphi_{Q'})$ 
with $P'$ of shape $\kappa' >_{\operatorname{lex}} \kappa$.
Both types of elements can then be expanded into the
required form by induction;
for the second type, one first expands
$\varphi_Q \varphi_{Q'}$ as a sum of terms $\varphi_R$ for
$R \in \Row(\mu,\kappa')$, then applies $\T$ to obtain a linear combination of
$\varphi_{P'}\T(\varphi_R)$'s, before invoking the induction hypothesis.
\end{proof}

\begin{remark}\label{fromsmallerones}
Let us explain how the canonical basis fits into this picture.
  In \cite[$\S$5.3]{DR}, one finds a Robinson-Schensted-Knuth-type
  correspondence giving a bijection
  \begin{align}\label{birds}
    \Mat{\lambda}{\mu} &\stackrel{\sim}{\rightarrow} \bigcup_{\kappa \in \Par} \Std(\lambda,\kappa)\times\Std(\mu,\kappa),
    &
    A &\mapsto (P(A), Q(A)),
  \end{align}
  which we explain more fully shortly.
  Also let $\kappa(A)$ be the common shape of the tableaux $P(A)$ and $Q(A)$ and recall \cref{shorthands}.
  Then
\cite[Th.~5.3.3]{DR} can be reformulated as follows:

\vspace{2mm}
\noindent
{\bf Theorem.} {\em
  The path algebra $\schur$ of $\qSchur$ has another triangular basis
  \begin{equation}\label{otherbasis}
  \big\{\beta_P \T(\beta_Q)
\:  \big|\:\textstyle(P,Q) \in \bigcup_{\lambda,\mu\in\Comp,\kappa \in \Par} \Std(\lambda,\kappa)\times\Std(\mu,\kappa)
\big\}\end{equation}
making it a symmetrically-based quasi-hereditary algebra
with $X(\lambda,\kappa) = \{\beta_P\:|\:P \in \Std(\lambda,\kappa)\}$ and all other data as in \cref{secrets}.
  Moreover, for $A \in \Mat{\lambda}{\mu}$ 
  we have that \begin{equation}
  \theta_A \equiv
  \beta_{P(A)}\T(\beta_{Q(A)})
  \pmod{\textstyle\sum_{B \in \Mat{\lambda}{\mu}\text{ with }\kappa(B) > \kappa(A)} \Z[q,q^{-1}] \theta_B}.
  \end{equation} 
  So the canonical basis is a cellular basis 
  which is equivalent to the triangular basis
  \cref{otherbasis}, i.e., it defines the same two-sided cell ideals and induces the same basis in each two-sided cell.}

  \vspace{2mm}

% NEW HERE
  
  \noindent
 To define the map \cref{birds} explicitly,
 take $A \in \Mat{\lambda}{\mu}$
 corresponding to $R \in \Row(\lambda,\mu)$ under the bijection \cref{Amap}.
 Let $\bi = (i_1,\dots,i_r) \in I_\lambda$ be the
 sequence $\ZigzagDownRev(R)$.
 Then we use {\em column insertion}\footnote{We mean the following algorithm to insert $i$ into a semistandard tableau:
 start with the first column; 
 if $i$ is bigger than all entries in the column then we add $i$ to the bottom of that column and stop; otherwise, we find the smallest entry $j$ in the column that is greater than or equal to $i$, replace that entry by $i$, then repeat to insert $j$ into the next column to the right.}
 to insert
 $i_1, \dots, i_r$ in order into the empty tableau,
 to end up with a semistandard tableau $P(A) \in \Std(\lambda,\kappa)$ for some $\kappa\vdash r$.
 We also obtain
 another semistandard tableau $Q(A) \in \Std(\mu,\kappa)$, namely, the {\em recording tableau}
 defined so that the entry of the 
 box that gets added  
 at the $r$th step of the algorithm is $i^\mu_r$.
 This concise description of the map \cref{birds} is equivalent to the more complicated description in \cite[$\S$5.3]{DR}. It takes some combinatorial work (omitted here) to establish the equivalence. 
For example, suppose that $A = \left[\begin{smallmatrix}1&0&2\\0&1&0\end{smallmatrix}\right]$,
\iffalse
so 
  $R=
\begin{tikzpicture}[anchorbase,scale=1.2]
  \draw (0,.75) to (.25,.75);
  \draw (0,.5) to (.25,.5);
  \draw (0,.25) to (.5,.25);
  \draw (0,0) to (.5,0);
  \draw (0,0) to (0,.75);
  \draw (0.25,0) to (0.25,.75);
  \draw (.5,0) to (.5,.25);
  \node at (0.125,0.125) {$\scriptstyle 1$};
  \node at (0.375,0.125) {$\scriptstyle 1$};
    \node at (0.125,0.375) {$\scriptstyle 2$};
  \node at (0.125,0.625) {$\scriptstyle 1$};
\end{tikzpicture}$,
\fi
  $\lambda = (3,1)$ and $\mu = (1,1,2)$.
  Then
$\bi = (1,2,1,1)$ and $\bi^\mu = (1,2,3,3)$.
Column insertion of the sequence $\bi$ gives
$\varnothing\stackrel{1}{\rightarrow} \begin{tikzpicture}[anchorbase,scale=1.2]
   \draw (0,.5) to (.25,.5);
    \draw (0,.25) to (.25,.25);
 \draw (0.25,.25) to (.25,.5);
 \draw (0,.25) to (0,.5);
   \node at (0.125,0.375) {$\scriptstyle 1$}; 
\end{tikzpicture}\stackrel{2}{\rightarrow}\begin{tikzpicture}[baseline=3.5mm,scale=1.2]
   \draw (0,.25) to (.25,.25);
   \draw (0,.5) to (.25,.5);
       \draw (0,0) to (.25,0);
 \draw (0.25,0) to (0.25,.5);
 \draw (0,0) to (0,.5);
   \node at (0.125,0.375) {$\scriptstyle 1$}; 
  \node at (0.125,0.125) {$\scriptstyle 2$}; 
 \end{tikzpicture}
\stackrel{1}{\rightarrow}\begin{tikzpicture}[baseline=3.5mm,scale=1.2]
   \draw (0,.5) to (.5,.5);
    \draw (0,.25) to (.5,.25);
 \draw (0,0) to (.25,0);
 \draw (0,0) to (0,.5);
 \draw (0.25,0) to (.25,.5);
 \draw (0.5,.25) to (0.5,.5);
   \node at (0.125,0.375) {$\scriptstyle 1$}; 
  \node at (0.375,0.375) {$\scriptstyle 1$}; 
  \node at (0.125,0.125) {$\scriptstyle 2$}; 
 \end{tikzpicture}\stackrel{1}{\rightarrow}\begin{tikzpicture}[baseline=3.5mm,scale=1.2]
   \draw (0,.5) to (.75,.5);
    \draw (0,.25) to (.75,.25);
 \draw (0,0) to (.25,0);
 \draw (0,0) to (0,.5);
 \draw (0.25,0) to (.25,.5);
 \draw (0.5,.25) to (0.5,.5);
 \draw (.75,.25) to (.75,.5);
   \node at (0.125,0.375) {$\scriptstyle 1$}; 
  \node at (0.375,0.375) {$\scriptstyle 1$}; 
  \node at (0.625,0.375) {$\scriptstyle 1$}; 
  \node at (0.125,0.125) {$\scriptstyle 2$}; 
\end{tikzpicture}$.
So we get that 
\begin{align*}
P(A)&= \begin{tikzpicture}[anchorbase,scale=1.2]
   \draw (0,.5) to (.75,.5);
    \draw (0,.25) to (.75,.25);
 \draw (0,0) to (.25,0);
 \draw (0,0) to (0,.5);
 \draw (0.25,0) to (.25,.5);
 \draw (0.5,.25) to (0.5,.5);
 \draw (.75,.25) to (.75,.5);
   \node at (0.125,0.375) {$\scriptstyle 1$}; 
  \node at (0.375,0.375) {$\scriptstyle 1$}; 
  \node at (0.625,0.375) {$\scriptstyle 1$}; 
  \node at (0.125,0.125) {$\scriptstyle 2$}; 
 \end{tikzpicture},&Q(A) &= \begin{tikzpicture}[anchorbase,scale=1.2]
   \draw (0,.5) to (.75,.5);
    \draw (0,.25) to (.75,.25);
 \draw (0,0) to (.25,0);
 \draw (0,0) to (0,.5);
 \draw (0.25,0) to (.25,.5);
 \draw (0.5,.25) to (0.5,.5);
 \draw (.75,.25) to (.75,.5);
   \node at (0.125,0.375) {$\scriptstyle 1$}; 
  \node at (0.375,0.375) {$\scriptstyle 3$}; 
  \node at (0.625,0.375) {$\scriptstyle 3$}; 
  \node at (0.125,0.125) {$\scriptstyle 2$}; 
 \end{tikzpicture}, & \kappa(A) &= (3,1).
 \end{align*}
\end{remark}

%% file: s8-tilting.tex
\section{Tilting modules}\label{s8-tilting}

For $n \geq 0$, let $\cI_n$ be the two-sided tensor ideal of $\qSchur$ generated 
by the identity morphisms $1_{(r)}$ for all $r > n$, then set
$\qSchur_{n} :=
\qSchur / \cI_n$.
This is a strict $\Z[q,q^{-1}]$-linear monoidal category.

\begin{theorem}\label{sfo}
  The path algebra $\schur_n$ of $\qSchur_{n}$ is a
  symmetrically-based quasi-hereditary
  algebra, with one possible triangular basis 
arising from the images of the
codeterminants from \cref{greenbasis}
for all $\kappa \in \Par$ 
satisfying $\kappa_1 \leq n$,
and another one given by the images of the
canonical basis products from \cref{otherbasis} for the same $\kappa$.
Also the images of the canonical basis
elements $\theta_A$ for $A \in \bigcup_{\lambda,\mu\in\Comp} \Mat{\lambda}{\mu}$ such that $\kappa(A)_1 \leq n$ give a cellular basis for $K_n$.
\end{theorem}

\begin{proof}
  The two-sided tensor ideal $\cI_n$ is equal to the ordinary two-sided ideal of
  $\qSchur$ generated by the morphisms $1_\kappa$ for all partitions
  $\kappa \in \Par$ with $\kappa_1 > n$.
  This follows because every object $\lambda \in \Comp$ which has some part $r > n$ is isomorphic to such a partition $\kappa$.
  Hence, $\cI_n$ corresponds to the two-sided ideal $I_n \lhd \schur$ of the path algebra
$K$ of $\qSchur$  generated by the idempotents $1_\kappa$
  for all $\kappa \in \Par$ with $\kappa_1 > n$,
  and $\schur_n = \schur / I_n$.
  The set $\{\kappa \in \Par\:|\:\kappa_1 > n\}$
  is an upper set in the poset $\Par$, hence, $I_n$ is a {\em cell ideal} in the based quasi-hereditary algebra $\schur$.
  Consequently, by \cite[Cor.~5.6]{BS}, the quotient algebra $\schur_n$ is also a symmetrically-based quasi-hereditary algebra with bases as described in the statement of the theorem.
  \end{proof}

Now let $\k$ be a field
viewed as a $\Z[q,q^{-1}]$-algebra in some way, and consider the $\k$-linear monoidal categories
$\qSchur(\k): = \k \otimes_{\Z[q,q^{-1}]}\qSchur$
and $\qSchur_n(\k) := \k \otimes_{\Z[q,q^{-1}]} \qSchur_n$.
From
 the bases as free $\Z[q,q^{-1}]$-modules
 discussed in the proof of \cref{sfo}, it follows that $\qSchur_n(\k)$
 may identified with the quotient of $\qSchur(\k)$ by the two-sided tensor ideal
 $\cI_n(\k)$ generated by the morphisms $1_{(r)}$ for $r > n$.
 
  Let $\qTilt_n(\k)$ be the monoidal
  category of polynomial tilting modules for $\qG_n(\k)$, that is, the full additive Karoubian monoidal subcategory of the category of polynomial representations of $\qG_n(\k)$ generated by the exterior powers $\bigwedge^r V$ for $1 \leq r \leq n$.
  Here, to avoid too much more notation,
  we are re-using $\bigwedge^r V$ to denote the specializations
 of the $\Z[q,q^{-1}]$-modules from before. Note also that we defined
the braided monoidal category $\qTilt_n(\k)$ in the introduction in a different way in terms modules over the algebra
 $U_n(\k)$, but the two definitions are equivalent. This identification requires the specific choice of comultiplication $\Delta$ described in the introduction in order for the induced homomorphism $\dot U_n \rightarrow \schur$ to map 
\begin{align}\label{chickencurry}
E_i^{(r)} 1_\lambda &\mapsto 
\begin{tikzpicture}[anchorbase,scale=1.8]
\draw [thick] (0,0) to (0,.8);
\draw [thick] (.4,0) to (.4,.8);
\draw [thick] (.6,0) to (.6,.8);
\draw [line width=.8pt] (.59,0) to (.59,.3) to (.41,.5) to (.41,.8);
\draw [thick] (1,0) to (1,.8);
\node at (.2,.4) {$\scriptstyle\cdots$};
\node at (.8,.4) {$\scriptstyle\cdots$};
\node at (0,-.1) {$\scriptstyle \lambda_1$};
\node at (.4,-.1) {$\scriptstyle \lambda_i$};
\node at (.67,-.1) {$\scriptstyle \lambda_{i+1}$};
\node at (1,-.1) {$\scriptstyle \lambda_n$};
\node at (.51,.48) {$\scriptstyle r$};
\end{tikzpicture},&
F_i^{(r)} 1_\lambda &\mapsto 
\begin{tikzpicture}[anchorbase,scale=1.8]
\draw [thick] (0,0) to (0,.8);
\draw [thick] (.4,0) to (.4,.8);
\draw [thick] (.6,0) to (.6,.8);
\draw [line width=.8pt] (.41,0) to (.41,.3) to (.59,.5) to (.59,.8);
\draw [thick] (1,0) to (1,.8);
\node at (.2,.4) {$\scriptstyle\cdots$};
\node at (.8,.4) {$\scriptstyle\cdots$};
\node at (0,-.1) {$\scriptstyle \lambda_1$};
\node at (.4,-.1) {$\scriptstyle \lambda_i$};
\node at (1,-.1) {$\scriptstyle \lambda_n$};
\node at (.67,-.1) {$\scriptstyle \lambda_{i+1}$};
\node at (.49,.48) {$\scriptstyle r$};
\end{tikzpicture}
\end{align}
for $1 \leq i < n, r \geq 0$ and $\lambda \in \N^n$ with $\lambda_{i+1} \geq r$ or $\lambda_i \geq r$, respectively 
(they map to zero for all other $\lambda$).
To see that the defining relations of $\dot U_n$ hold in $\schur$, most of them are easy, indeed, this is the origin of the square-switch relation.
The Serre relation is deduced from the other relations in \cite[Lem.~2.2.1]{CKM}.

\begin{remark}
When $0 \leq a-d\leq b-c$, the expressions in \cref{jonsquare,jonsquare2}
are the canonical basis elements
$\theta_{\left[\begin{smallmatrix}a-d&c\\d&b-c\end{smallmatrix}\right]}$ and $\theta_{\left[\begin{smallmatrix}
b-c&d\\c&a-d
\end{smallmatrix}
\right]}$ from \cref{rank2canonical}. They
are also the images under the homomorphism \cref{chickencurry}
of the canonical basis elements
$E^{(c)}F^{(d)}1_{(a,b)}$ and
$F^{(c)}E^{(d)}1_{(b,a)}$ of $\dot U_2$. 
\end{remark}
 
 The monoidal functor $\Sigma_n$ from \cref{dthm} extends to define a $\k$-linear
 monoidal functor $\qSchur(\k) \rightarrow \qTilt_n(\k)$.
 Since $\bigwedge^r V = \{0\}$ for $r > n$,
 this factors through the quotient $\qSchur_n(\k)$
 to induce a $\k$-linear monoidal functor
$ \bar{\Sigma}_n:\qSchur_n(\k) \rightarrow \qTilt_n(\k)$.
 
 \begin{theorem}\label{sanbruno}
For any field $\k$, the functor $\bar{\Sigma}_n :\qSchur(\k) \rightarrow \qTilt_n(\k)$
   induces a $\k$-linear monoidal equivalence
   between the additive Karoubi envelope of 
   $\qSchur_n(\k)$ and $\qTilt_n(\k)$.
 \end{theorem}

 \begin{proof}
   We saw already in \cref{ginandtonic}(2) that $\bar{\Sigma}_n$ is full. It is dense by the definition of $\qTilt_n(\k)$. It just remains to show that it is faithful.
   Thus,
   we must show that the surjective
  $\k$-linear map
  $  \textstyle
  \Hom_{\qSchur_n(\k)}(\mu,\lambda) \twoheadrightarrow \Hom_{\qG_n(\k)}(\bigwedge^\mu V, \bigwedge^\lambda V)$
  induced by the functor is also injective for any $\lambda,\mu\vDash r$.
By \cref{sfo}, we know that the morphism space on the left is of dimension
$\sum_{\kappa \vdash r} |\Std(\lambda,\kappa) \times \Std(\mu,\kappa)|$.
This is also the dimension of
$\Hom_{\qG_n(\k)}(\bigwedge^\mu V, \bigwedge^\lambda V)$. Indeed, in the highest weight category of polynomial representations of $\qG_n(\k)$, the 
tilting module $\bigwedge^\mu V$ has a filtration
with sections that are standard modules $\Delta(\kappa')$ for partions $\kappa$ with $\kappa_1 \leq n$,
and $\bigwedge^\lambda V$ has a filtration with sections that are costandard modules
$\nabla(\kappa')$ for the same $\kappa$.
By the Littlewood-Richardson rule, the multiplicities
$(\bigwedge^\mu V: \Delta(\kappa'))$ and $(\bigwedge^\mu V: \nabla(\kappa'))$ are
$|\Row(\mu,\kappa)|$ and $|\Row(\lambda,\kappa)|$.
Since $\dim \Ext^i_{\qG_n(\k)}(\Delta(\lambda), \nabla(\mu)) = \delta_{\lambda,\mu} \delta_{i,0}$,
this is enough to prove that
$\Hom_{\qG_n(\k)}(\bigwedge^\mu V, \bigwedge^\lambda V)$ has the same dimension
as $\Hom_{\qSchur_n(\k)}(\mu,\lambda)$.
 \end{proof}

 \begin{corollary}\label{webn}
  The kernel of the full monoidal functor
  $\Sigma_n$ from \cref{dthm} is equal to $\cI_n$.
 \end{corollary}

 \begin{proof}
Let $\cJ_n$ be the kernel of $\Sigma_n$.
  Since $\bigwedge^r V = \{0\}$ for $r > n$, it is clear that $\cI_n \subseteq \cJ_n$.
  Hence, $\Sigma_n$ factors through the quotient to induce a full $\Z[q,q^{-1}]$-linear monoidal functor from
  $\qSchur_n$ to the category of polynomial representations of $\qG_n$.
  To prove that $\cJ_n = \cI_n$, thereby proving the corollary, it remains to show that this induced functor is also faithful.
  This follows because it remains an isomorphism on base change to $\Q(q)$ by a special case of \cref{sanbruno}.
\end{proof}

\vspace{2mm}
\noindent
    {\em Proofs of results in the introduction.}
    Recall that in the introduction we were discussing the $q$-Schur category without $0$-strings.
This is the full subcategory 
of the $q$-Schur category with $0$-strings
generated by the objects $\sComp$. The path algebra $\sch$ of the category without $0$-strings
from \cref{pathalgebra} is the idempotent truncation
$\sch = \bigoplus_{\lambda,\mu \in \sComp} 1_\lambda \schur 1_\mu$ of the path algebra $\schur$
of the category with $0$-strings from \cref{norfolk}.
The set $\Par$ indexing the special idempotents is a subset of $\sComp \subset \Comp$.
In view of this, \cref{th2} follows immediately from \cref{secrets}.
Every object of the $q$-Schur category with $0$-strings is isomorphic to an object of the $q$-Schur category without $0$-strings. So the two path algebras $\schur$ and $\sch$ are Morita equivalent,
and the restriction of the equivalence from \cref{sanbruno} remains an equivalence.
\cref{th3} follows.
Finally, we explain how to establish the presentations in \cref{th1alt,th1}.
These are similar to the ones in \cref{mainpres} and \cref{altpres}, respectively, but we have omitted the relations involving 
the generators
$\begin{tikzpicture}[anchorbase]
\draw[line width=0pt] (0,-.3) to (0,0);
\spot{0,0};
\end{tikzpicture}$
and
$\begin{tikzpicture}[anchorbase]
\draw[line width=0pt] (0,.3) to (0,0);
\spot{0,0};
\end{tikzpicture}$. Instead, the relations \cref{squareswitch,altrels} need to be interpreted 
in a different way when strings labelled by 0 are present---simply omit those strings so that the splits and merges become identity morphisms.
That these relations hold follows from the ones in \cref{mainpres,altpres} by contracting $0$-strings. To complete the proof of \cref{th1alt}, one needs to show that we have a full set of relations. This follows by a straightening argument which is the same as the one used in the proof of \cref{mainpres}. Then \cref{th1} follows from \cref{th1alt} 
by the same argument that was used to deduce \cref{altpres} from \cref{mainpres}.

%% file: s9-appendix.tex
\renewcommand{\theequation}{\Alph{equation}}

\vspace{4mm}

\begin{center}
{\sc
Appendix: Relations}
\end{center}

\vspace{2mm}

The remaining pages contain some elementary calculations involving generators and relations
needed in the proofs of \cref{mainpres,altpres}.
%They are intended as an appendix included in the {\tt arxiv} version of the paper but not in the published version. 

\vspace{2mm}

\noindent{\bf Lemma.}
{\em
Assume
the relations \cref{secondone,zeroforks,assrel},
the first relation from \cref{mergesplit}, 
and the second square-switch relation from \cref{squareswitch}.
The following hold for all $a,b,c,d\geq 0$ with $a+b=c+d$:
\begin{enumerate}
\item
$\displaystyle\begin{tikzpicture}[anchorbase,scale=.9]
	\draw[-,line width=1.25pt] (0,0) to (.285,.3) to (.285,.7) to (0,1);
	\draw[-,line width=1.25pt] (.61,0) to (.325,.3) to (.325,.7) to (.61,1);
        \node at (0,1.13) {$\scriptstyle c$};
        \node at (0.6,1.13) {$\scriptstyle d$};
        \node at (0,-.1) {$\scriptstyle a$};
        \node at (0.6,-.1) {$\scriptstyle b$};
\end{tikzpicture}
=
\sum_{s=0}^{\min(a,c)}
q^{s(d-a+s)}
\sum_{t=0}^{\min(a,c)-s}
(-q)^{t}
\begin{tikzpicture}[anchorbase,scale=.9]
	\draw[-,thin] (0,0) to (0,1);
	\draw[-,thick] (0.02,0) to (0.02,.2) to (.7,.3) to (.7,.7) to (1.08,.8) to (1.08,1);
	\draw[-,thick] (1.08,0) to (1.08,.2) to (.7,.3) to (.7,.7) to (.02,.8) to (.02,1);
	\draw[-,thin] (.01,0) to (0.01,.22) to (.35,.27) to (.35,.73) to
        (0.01,.78) to (.01,1);
	\draw[-,thin] (1.1,0) to (1.1,1);
        \node at (0,1.13) {$\scriptstyle c$};
        \node at (1.13,1.13) {$\scriptstyle d$};
        \node at (0,-.1) {$\scriptstyle a$};
        \node at (1.13,-.1) {$\scriptstyle b$};
        \node at (-0.1,.5) {$\scriptstyle s$};
        \node at (1.6,.5) {$\scriptstyle d-a+s$};
        \node at (0.25,.5) {$\scriptstyle t$};
\end{tikzpicture}$.
\item
$\displaystyle\begin{tikzpicture}[anchorbase,scale=.9]
	\draw[-,line width=1.25pt] (0,0) to (.285,.3) to (.285,.7) to (0,1);
	\draw[-,line width=1.25pt] (.61,0) to (.325,.3) to (.325,.7) to (.61,1);
        \node at (0,1.13) {$\scriptstyle c$};
        \node at (0.6,1.13) {$\scriptstyle d$};
        \node at (0,-.1) {$\scriptstyle a$};
        \node at (0.6,-.1) {$\scriptstyle b$};
\end{tikzpicture}
=
\sum_{s=0}^{\min(a,c)}
q^{-s(d-a+s)}
\sum_{t=0}^{\min(a,c)-s}
(-q)^{-t}
\begin{tikzpicture}[anchorbase,scale=.9]
	\draw[-,thin] (0,0) to (0,1);
	\draw[-,thick] (0.02,0) to (0.02,.2) to (.7,.3) to (.7,.7) to (1.08,.8) to (1.08,1);
	\draw[-,thick] (1.08,0) to (1.08,.2) to (.7,.3) to (.7,.7) to (.02,.8) to (.02,1);
	\draw[-,thin] (.01,0) to (0.01,.22) to (.35,.27) to (.35,.73) to
        (0.01,.78) to (.01,1);
	\draw[-,thin] (1.1,0) to (1.1,1);
        \node at (0,1.13) {$\scriptstyle c$};
        \node at (1.13,1.13) {$\scriptstyle d$};
        \node at (0,-.1) {$\scriptstyle a$};
        \node at (1.13,-.1) {$\scriptstyle b$};
        \node at (-0.1,.5) {$\scriptstyle s$};
        \node at (1.6,.5) {$\scriptstyle d-a+s$};
        \node at (0.25,.5) {$\scriptstyle t$};
\end{tikzpicture}$.
\end{enumerate}
}

\begin{proof}
(a) Using \cref{assrel} and the first relation from \cref{mergesplit} to combine the strings of thickness $s$ and $t$, the right hand side of the identity to be
proved
simplifies to
$$
\sum_{s=0}^{\min(a,c)} 
\sum_{u=s}^{\min(a,c)}
q^{s(d-a+s)} (-q)^{u-s}
\qbinom{u}{s}_{\!q}
\begin{tikzpicture}[anchorbase,scale=1.35]
	\draw[-,thin] (0,0) to (0,1);
	\draw[-,thick] (0.02,0) to (0.02,.2) to (.56,.35) to (.56,.65) to (1.08,.8) to (1.08,1);
	\draw[-,thick] (1.08,0) to (1.08,.2) to (.56,.35) to (.56,.65) to (.02,.8) to (.02,1);
	\draw[-,thin] (1.1,0) to (1.1,1);
        \node at (0,1.13) {$\scriptstyle c$};
        \node at (1.13,1.13) {$\scriptstyle d$};
        \node at (0,-.1) {$\scriptstyle a$};
        \node at (1.13,-.1) {$\scriptstyle b$};
        \node at (-0.11,.5) {$\scriptstyle u$};
        \node at (1.5,.5) {$\scriptstyle d-a+s$};
        \node at (.3,.86) {$\scriptstyle c-u$};
        \node at (.28,.41) {$\scriptstyle a-u$};
        \node at (.83,.62) {$\scriptstyle a-s$};
        \node at (.83,.15) {$\scriptstyle c-s$};
        \end{tikzpicture}.
$$
Then we use the second square-switch relation from \cref{squareswitch}
to see that this equals
$$
\sum_{s=0}^{\min(a,c)}
\sum_{u=s}^{\min(a,c)}
\sum_{v=u}^{\min(a,c)}
q^{s(d-a+s)} (-q)^{u-s}
\qbinom{u}{s}_{\!q}
\qbinom{d-a+u}{v-s}_{\!q}
\begin{tikzpicture}[anchorbase,scale=1]
	\draw[-,thin] (0,0) to (0,1);
	\draw[-,thick] (0.02,0) to (0.02,.2) to (.86,.35) to
        (.86,.65) to (.02,.8) to (.02,1);
        \draw[-,line width=.4pt] (.86,.655) to (1.473,.575) to (1.473,.425) to (.86,.345);
	\draw[-,line width=1pt] (1.493,0) to (1.493,1);
        \node at (0,1.13) {$\scriptstyle c$};
        \node at (1.51,1.13) {$\scriptstyle d$};
        \node at (0,-.1) {$\scriptstyle a$};
        \node at (1.51,-.1) {$\scriptstyle b$};
        \node at (-0.14,.5) {$\scriptstyle u$};
        \node at (.55,.5) {$\scriptstyle v-u$};
\end{tikzpicture}.
$$
Using \cref{assrel} and the first relation from \cref{mergesplit} again, 
this simplifies to
\begin{equation*}
\sum_{s=0}^{\min(a,c)}
\sum_{u=s}^{\min(a,c)}
\sum_{v=u}^{\min(a,c)}
q^{s(d-a+s)}(-q)^{u-s}
\qbinom{u}{s}_{\!q}
\qbinom{d-a+u}{v-s}_{\!q}
\qbinom{v}{u}_{\!q}
\begin{tikzpicture}[anchorbase,scale=.9]
	\draw[-,thin] (0,0) to (0,1);
	\draw[-,thick] (0.02,0) to (0.02,.2) to (.66,.35) to
        (.66,.65) to (.02,.8) to (.02,1);
	\draw[-,line width=1pt] (.688,0) to (.688,1);
        \node at (0,1.13) {$\scriptstyle c$};
        \node at (.69,1.13) {$\scriptstyle d$};
        \node at (0,-.1) {$\scriptstyle a$};
        \node at (.69,-.1) {$\scriptstyle b$};
        \node at (-0.11,.5) {$\scriptstyle v$};
\end{tikzpicture}.
\end{equation*}
Next, switch the orders of the summations 
and rearrange binomial coefficients
to get
\begin{multline*}
\sum_{v=0}^{\min(a,c)}
\sum_{s=0}^v
\sum_{u=s}^{v}
q^{s(d-a+s)} 
(-q)^{u-s}
\qbinom{u}{s}_{\!q}
\qbinom{d-a+u}{v-s}_{\!q}
\qbinom{v}{u}_{\!q}
 \begin{tikzpicture}[anchorbase,scale=.9]
	\draw[-,thin] (0,0) to (0,1);
	\draw[-,thick] (0.02,0) to (0.02,.2) to (.66,.35) to
        (.66,.65) to (.02,.8) to (.02,1);
	\draw[-,line width=1pt] (.688,0) to (.688,1);
        \node at (0,1.13) {$\scriptstyle c$};
        \node at (.69,1.13) {$\scriptstyle d$};
        \node at (0,-.1) {$\scriptstyle a$};
        \node at (.69,-.1) {$\scriptstyle b$};
        \node at (-0.11,.5) {$\scriptstyle v$};
\end{tikzpicture}\\
=
\sum_{v=0}^{\min(a,c)}
\sum_{s=0}^v
q^{s(d-a+s)} 
(-q)^{v-s}
\qbinom{v}{s}_{\!q}
\left(
\displaystyle
\sum_{u=s}^{v}
(-q)^{u-v}
\qbinom{d-a+u}{u-s}_{\!q}
\qbinom{d-a+s}{v-u}_{\!q}\right)
 \begin{tikzpicture}[anchorbase,scale=.9]
	\draw[-,thin] (0,0) to (0,1);
	\draw[-,thick] (0.02,0) to (0.02,.2) to (.66,.35) to
        (.66,.65) to (.02,.8) to (.02,1);
	\draw[-,line width=1pt] (.688,0) to (.688,1);
        \node at (0,1.13) {$\scriptstyle c$};
        \node at (.69,1.13) {$\scriptstyle d$};
        \node at (0,-.1) {$\scriptstyle a$};
        \node at (.69,-.1) {$\scriptstyle b$};
        \node at (-0.11,.5) {$\scriptstyle v$};
\end{tikzpicture}.
\end{multline*}
Applying \cref{A}, taking $a,b,m$ and $s$ there to be
$u-s$, $v-u$, $b-c+s$ and
$v-s$ in the present setup, shows that 
the term in parentheses is equal to $q^{(b-c+s)(v-s)}$.
Hence, we have
$$
\sum_{v=0}^{\min(a,c)}
(-q)^{v} q^{(b-c)v}\left(\sum_{s=0}^v
(-1)^s q^{s(v-1)}\qbinom{v}{s}_{\!q}
\right)
\begin{tikzpicture}[anchorbase,scale=.9]
	\draw[-,thin] (0,0) to (0,1);
	\draw[-,thick] (0.02,0) to (0.02,.2) to (.66,.35) to
        (.66,.65) to (.02,.8) to (.02,1);
	\draw[-,line width=1pt] (.688,0) to (.688,1);
        \node at (0,1.13) {$\scriptstyle c$};
        \node at (.69,1.13) {$\scriptstyle d$};
        \node at (0,-.1) {$\scriptstyle a$};
        \node at (.69,-.1) {$\scriptstyle b$};
        \node at (-0.11,.5) {$\scriptstyle v$};
\end{tikzpicture}.
$$
By \cref{vanishing}, the term in parentheses is zero unless $v=0$,
leaving us just with the $v=0$ term which, after contracting the string of thickness zero, is equal to the desired
left hand side.

\vspace{2mm}
\noindent
(b)
This is just the same argument as (a) with $q$ replaced by $q^{-1}$ throughout (including in
\cref{A,vanishing}).
\end{proof}

\vspace{2mm}

\noindent{\bf Corollary.}
{\em
Assume the relations \cref{secondone,zeroforks,assrel},
the first relation from \cref{mergesplit}, 
and the second square-switch relation from \cref{squareswitch}.
\begin{enumerate}
\item
If 
$\displaystyle\begin{tikzpicture}[anchorbase,scale=1.1]
	\draw[-,line width=1pt] (0.3,-.3) to (-.3,.4);
	\draw[-,line width=4pt,white] (-0.3,-.3) to (.3,.4);
	\draw[-,line width=1pt] (-0.3,-.3) to (.3,.4);
        \node at (0.3,-.42) {$\scriptstyle b$};
        \node at (-0.3,-.42) {$\scriptstyle a$};
\end{tikzpicture} \!\!=\!\!
\sum_{s=0}^{\min(a,b)}\!\!
(-q)^{s}
\begin{tikzpicture}[anchorbase,scale=.9]
	\draw[-,line width=1.2pt] (0,0) to (0,1);
	\draw[-,thick] (-0.8,0) to (-0.8,.2) to (-.03,.4) to (-.03,.6)
        to (-.8,.8) to (-.8,1);
	\draw[-,thin] (-0.82,0) to (-0.82,1);
        \node at (-0.81,-.1) {$\scriptstyle a$};
        \node at (0,-.1) {$\scriptstyle b$};
        \node at (-0.4,.9) {$\scriptstyle b-s$};
        \node at (-0.4,.13) {$\scriptstyle a-s$};
\end{tikzpicture}$
for all $a,b \geq 0$ then
$\displaystyle
\begin{tikzpicture}[anchorbase,scale=.9]
	\draw[-,line width=1.2pt] (0,0) to (.275,.3) to (.275,.7) to (0,1);
	\draw[-,line width=1.2pt] (.6,0) to (.315,.3) to (.315,.7) to (.6,1);
        \node at (0,1.13) {$\scriptstyle c$};
        \node at (0.63,1.13) {$\scriptstyle d$};
        \node at (0,-.1) {$\scriptstyle a$};
        \node at (0.63,-.1) {$\scriptstyle b$};
\end{tikzpicture}\!\!\!
=\!
\sum_{s=0}^{\min(a,c)}\!\!
q^{s(d-a+s)}
\begin{tikzpicture}[anchorbase,scale=.9]
	\draw[-,thick] (0.58,0) to (0.58,.2) to (.02,.8) to (.02,1);
	\draw[-,line width=4pt,white] (0.02,0) to (0.02,.2) to (.58,.8) to (.58,1);
	\draw[-,thick] (0.02,0) to (0.02,.2) to (.58,.8) to (.58,1);
	\draw[-,thin] (0,0) to (0,1);
	\draw[-,line width=1pt] (0.59,0) to (0.59,1);
        \node at (0,1.13) {$\scriptstyle c$};
        \node at (0.6,1.13) {$\scriptstyle d$};
        \node at (0,-.1) {$\scriptstyle a$};
        \node at (0.6,-.1) {$\scriptstyle b$};
        \node at (-0.1,.5) {$\scriptstyle s$};
        \node at (1.15,.5) {$\scriptstyle d-a+s$};
\end{tikzpicture}
$
for all $a,b,c,d \geq 0$ with $a+b=c+d$.
\item
If
$
\displaystyle\begin{tikzpicture}[anchorbase,scale=1.1]
	\draw[-,line width=1pt] (-0.3,-.3) to (.3,.4);
 \draw[-,line width=4pt,white] (0.3,-.3) to (-.3,.4);
 \draw[-,line width=1pt] (0.3,-.3) to (-.3,.4);
        \node at (0.3,-.42) {$\scriptstyle b$};
        \node at (-0.3,-.42) {$\scriptstyle a$};
\end{tikzpicture}
\!\!=\!\!\sum_{s=0}^{\min(a,b)}\!\!
(-q)^{-s}
\begin{tikzpicture}[anchorbase,scale=.9]
	\draw[-,line width=1.2pt] (0,0) to (0,1);
	\draw[-,thick] (-0.8,0) to (-0.8,.2) to (-.03,.4) to (-.03,.6)
        to (-.8,.8) to (-.8,1);
	\draw[-,thin] (-0.82,0) to (-0.82,1);
        \node at (-0.81,-.1) {$\scriptstyle a$};
        \node at (0,-.1) {$\scriptstyle b$};
        \node at (-0.4,.9) {$\scriptstyle b-s$};
        \node at (-0.4,.13) {$\scriptstyle a-s$};
\end{tikzpicture}
$
for all $a,b \geq 0$
then
$\displaystyle
\begin{tikzpicture}[anchorbase,scale=.9]
	\draw[-,line width=1.2pt] (0,0) to (.275,.3) to (.275,.7) to (0,1);
	\draw[-,line width=1.2pt] (.6,0) to (.315,.3) to (.315,.7) to (.6,1);
        \node at (0,1.13) {$\scriptstyle c$};
        \node at (0.63,1.13) {$\scriptstyle d$};
        \node at (0,-.1) {$\scriptstyle a$};
        \node at (0.63,-.1) {$\scriptstyle b$};
\end{tikzpicture}
\!\!\!=\!
\sum_{s=0}^{\min(a,c)}
\!q^{-s(d-a+s)}
\begin{tikzpicture}[anchorbase,scale=.9]
	\draw[-,thick] (0.02,0) to (0.02,.2) to (.58,.8) to (.58,1);
 \draw[-,line width=4pt,white] (0.58,0) to (0.58,.2) to (.02,.8) to (.02,1);
 \draw[-,thick] (0.58,0) to (0.58,.2) to (.02,.8) to (.02,1);
	\draw[-,thin] (0,0) to (0,1);
	\draw[-,line width=1pt] (0.59,0) to (0.59,1);
        \node at (0,1.13) {$\scriptstyle c$};
        \node at (0.6,1.13) {$\scriptstyle d$};
        \node at (0,-.1) {$\scriptstyle a$};
        \node at (0.6,-.1) {$\scriptstyle b$};
        \node at (-0.1,.5) {$\scriptstyle s$};
        \node at (1.15,.5) {$\scriptstyle d-a+s$};
\end{tikzpicture}
$
for all $a,b,c,d \geq 0$ with $a+b=c+d$.
\end{enumerate}}

\begin{proof}
(a)
Using the assumed formula for positive crossings, the right hand side of the identity to be proved is equal to \begin{align*}
\sum_{s=0}^{\min(a,b)}
q^{s(d-a+s)}
&\sum_{t=0}^{\min(a,b)-s}
(-q)^{t}
\begin{tikzpicture}[anchorbase,scale=1]
	\draw[-,thin] (0,0) to (0,1);
	\draw[-,thick] (0.02,0) to (0.02,.2) to (.7,.3) to (.7,.7) to (1.08,.8) to (1.08,1);
	\draw[-,thick] (1.08,0) to (1.08,.2) to (.7,.3) to (.7,.7) to (.02,.8) to (.02,1);
	\draw[-,thin] (.01,0) to (0.01,.22) to (.35,.27) to (.35,.73) to
        (0.01,.78) to (.01,1);
	\draw[-,thin] (1.1,0) to (1.1,1);
        \node at (0,1.13) {$\scriptstyle c$};
        \node at (1.13,1.13) {$\scriptstyle d$};
        \node at (0,-.1) {$\scriptstyle a$};
        \node at (1.13,-.1) {$\scriptstyle b$};
        \node at (-0.1,.5) {$\scriptstyle s$};
        \node at (1.6,.5) {$\scriptstyle d-a+s$};
        \node at (0.25,.5) {$\scriptstyle t$};
\end{tikzpicture}.
\end{align*}
Then we apply the identity from part (a) of the lemma.

\vspace{2mm}
\noindent
(b) Similar using part (b) of the lemma.
\end{proof}

\vspace{2mm}
Now let $\qSchur'$ be as in the proof of \cref{mainpres}.
We are going to prove the relations 
\cref{jonsquare,jonsquare2,thickcrossing,sliders,swallows-symmetric,braid} needed in that proof. 
We must deduce these from
the defining relations \cref{secondone,zeroforks,assrel,mergesplit}
for $\qSchur'$ and the definition \cref{tireddog}.
We will use the strict $\Z[q,q^{-1}]$-linear monoidal isomorphisms
\begin{align*}
\mathtt{R}:\qSchur' &\rightarrow (\qSchur')^{\operatorname{rev}},&
\mathtt{T}:\qSchur'&\rightarrow (\qSchur')^\op.
\end{align*}
Both are defined on objects by $(r) \mapsto (r)$.
On generating 
morphisms, $\mathtt{R}$
rotates the diagrams around a vertical axis, i.e.,
$\begin{tikzpicture}[anchorbase,scale=1]
\draw[line width=0pt] (0,-.3) to (0,0);
\spot{0,0};
%\node at (0,-.42) {$\scriptstyle 0$};
\end{tikzpicture}\mapsto 
\begin{tikzpicture}[anchorbase,scale=1]
\draw[line width=0pt] (0,-.3) to (0,0);
\spot{0,0};
%\node at (0,-.42) {$\scriptstyle 0$};
\end{tikzpicture}\ $,
$\begin{tikzpicture}[anchorbase,scale=1]
\draw[line width=0pt] (0,.3) to (0,0);
\spot{0,0};
%\node at (0,.42) {$\scriptstyle 0$};
\end{tikzpicture}\mapsto \begin{tikzpicture}[anchorbase,scale=1]
\draw[line width=0pt] (0,.3) to (0,0);
\spot{0,0};
%\node at (0,.42) {$\scriptstyle 0$};
\end{tikzpicture}\ $,
$\begin{tikzpicture}[baseline=-1mm,scale=.6]
	\draw[-,line width=1pt] (0.28,-.3) to (0.08,0.04);
	\draw[-,line width=1pt] (-0.12,-.3) to (0.08,0.04);
	\draw[-,line width=2pt] (0.08,.4) to (0.08,0);
        \node at (-0.22,-.43) {$\scriptstyle a$};
        \node at (0.35,-.43) {$\scriptstyle b$};
\end{tikzpicture}\mapsto \begin{tikzpicture}[baseline=-1mm,scale=.6]
	\draw[-,line width=1pt] (0.28,-.3) to (0.08,0.04);
	\draw[-,line width=1pt] (-0.12,-.3) to (0.08,0.04);
	\draw[-,line width=2pt] (0.08,.4) to (0.08,0);
        \node at (-0.22,-.43) {$\scriptstyle b$};
        \node at (0.35,-.43) {$\scriptstyle a$};
\end{tikzpicture}$,
$\begin{tikzpicture}[baseline=0mm,scale=.6]
	\draw[-,line width=2pt] (0.08,-.3) to (0.08,0.04);
	\draw[-,line width=1pt] (0.28,.4) to (0.08,0);
	\draw[-,line width=1pt] (-0.12,.4) to (0.08,0);
        \node at (-0.22,.53) {$\scriptstyle a$};
        \node at (0.36,.55) {$\scriptstyle b$};
\end{tikzpicture}\mapsto 
\begin{tikzpicture}[baseline=0mm,scale=.6]
	\draw[-,line width=2pt] (0.08,-.3) to (0.08,0.04);
	\draw[-,line width=1pt] (0.28,.4) to (0.08,0);
	\draw[-,line width=1pt] (-0.12,.4) to (0.08,0);
        \node at (-0.22,.53) {$\scriptstyle b$};
        \node at (0.36,.55) {$\scriptstyle a$};
\end{tikzpicture}$ and
$\begin{tikzpicture}[baseline=-1.5mm,scale=.6]
	\draw[-,thick] (0.3,-.35) to (-.3,.43);
	\draw[-,line width=5pt,white] (-0.3,-.35) to (.3,.43);
	\draw[-,thick] (-0.3,-.3) to (.3,.4);
       \node at (-0.32,-.53) {$\scriptstyle a$};
        \node at (0.32,-.54) {$\scriptstyle b$};
\end{tikzpicture}
\mapsto \begin{tikzpicture}[baseline=-1.5mm,scale=.6]
	\draw[-,thick] (0.3,-.35) to (-.3,.43);
	\draw[-,line width=5pt,white] (-0.3,-.35) to (.3,.43);
	\draw[-,thick] (-0.3,-.3) to (.3,.4);
       \node at (-0.32,-.53) {$\scriptstyle b$};
        \node at (0.32,-.54) {$\scriptstyle a$};
\end{tikzpicture}$.
Similarly, $\mathtt{T}$ rotates diagrams around a horizontal axis.
Existence of $\mathtt{R}$ and $\mathtt{T}$ follows easily from the symmetry of the defining relations
of $\qSchur'$.

\vspace{2mm}

\noindent
{\bf \cref{jonsquare}--\cref{jonsquare2}}.
Note $a \geq d$.
To prove the first equality in \cref{jonsquare}, we expand the left hand side 
as a sum of diagrams involving a crossing
using \cref{mergesplit}, to see that
\begin{align*}
\begin{tikzpicture}[anchorbase,scale=1]
	\draw[-,thick] (0,0) to (0,1);
	\draw[-,line width=1.2pt] (0.6,0) to (0.6,1);
	\draw[-,thick] (0.015,0) to (0.015,.2) to (.57,.4) to (.57,.6)
        to (.015,.8) to (.015,1);
        \node at (0.6,-.1) {$\scriptstyle b$};
        \node at (0,-.1) {$\scriptstyle a$};
        \node at (0.3,.82) {$\scriptstyle c$};
        \node at (0.3,.19) {$\scriptstyle d$};
\end{tikzpicture}
&=
\sum_{s=\max(0,c-b)}^{\min(c,d)}
q^{s(b-c+s)}
\begin{tikzpicture}[anchorbase,scale=1]
	\draw[-,thick] (1.08,0) to (1.08,.3) to (.65,.5) to (.02,.8) to (.02,1);
	\draw[-,line width=3.5pt,white] (0.02,0) to (0.02,.2) to (.65,.5) to (1.08,.7) to (1.08,1);
	\draw[-,thick] (0.02,0) to (0.02,.2) to (.65,.5) to (1.08,.7) to (1.08,1);
	\draw[-,line width=.7pt] (.01,0) to (0.01,.21) to (.35,.37) to (.35,.63) to
        (0.01,.79) to (.01,1);
	\draw[-,thin] (0,0) to (0,1);
	\draw[-,thin] (1.1,0) to (1.1,1);
        \node at (0,-.1) {$\scriptstyle a$};
        \node at (1.13,-.1) {$\scriptstyle b$};
        \node at (-0.3,.5) {$\scriptstyle a-d$};
        \node at (0.26,.5) {$\scriptstyle s$};
        \node at (.2,.18) {$\scriptstyle d$};
        \node at (.2,.84) {$\scriptstyle c$};
\end{tikzpicture}.
\end{align*}
Then use \cref{assrel} and the first relation from \cref{mergesplit} to establish the result.
The first equality in \cref{jonsquare2} follows now by applying $\mathtt{R}$.
Then to prove the second equality in \cref{jonsquare}, we 
use the first equality from \cref{jonsquare2}
to expand the
right hand side to see that it equals
$$
\sum_{s=\max(0,c-b)}^{\min(c,d)}
\sum_{t=s}^{\min(c,d)}
q^{(t-s)(a-d+t)}\qbinom{a-b+c-d}{s}_{\!q}
\qbinom{b-c+t}{t-s}_{\!q}
\begin{tikzpicture}[anchorbase,scale=1]
	\draw[-,thick] (0.88,0) to (0.88,.2) to (.02,.8) to (.02,1);
	\draw[-,line width=4pt,white] (0.02,0) to (0.02,.2) to (.88,.8) to (.88,1);
	\draw[-,thick] (0.02,0) to (0.02,.2) to (.88,.8) to (.88,1);
	\draw[-,thin] (0,0) to (0,1);
	\draw[-,thin] (0.9,0) to (0.9,1);
        \node at (0,-.1) {$\scriptstyle a$};
        \node at (.9,-.1) {$\scriptstyle b$};
        \node at (.33,.18) {$\scriptstyle d-t$};
        \node at (.34,.85) {$\scriptstyle c-t$};
\end{tikzpicture}.
$$
Now switch the summations to get
$$
\sum_{t=\max(0,c-b)}^{\min(c,d)}
\!\!\!\!\!\!q^{t(b-c+t)}
\!\!\!\!\!\!\!\sum_{s=\max(0,c-b)}^{t}\!\!\!\!\!
q^{(a-b+c-d)(t-s)-s(b-c+t)}
 \qbinom{a-b+c-d}{s}_{\!q}
\qbinom{b-c+t}{t-s}_{\!q}
\begin{tikzpicture}[anchorbase,scale=1]
	\draw[-,thick] (0.88,0) to (0.88,.2) to (.02,.8) to (.02,1);
	\draw[-,line width=4pt,white] (0.02,0) to (0.02,.2) to (.88,.8) to (.88,1);
	\draw[-,thick] (0.02,0) to (0.02,.2) to (.88,.8) to (.88,1);
	\draw[-,thin] (0,0) to (0,1);
	\draw[-,thin] (0.9,0) to (0.9,1);
        \node at (0,-.1) {$\scriptstyle a$};
        \node at (.9,-.1) {$\scriptstyle b$};
        \node at (.33,.18) {$\scriptstyle d-t$};
        \node at (.34,.85) {$\scriptstyle c-t$};
\end{tikzpicture}.
$$
In the second summation, if $c > b$, we can add extra terms with $0
\leq s < c-b$ since the second binomial coefficient is zero for all
such $s$.
So we are in a position to apply \cref{second} to get
$$
\sum_{t=\max(0,c-b)}^{\min(c,d)}
q^{t(b-c+t)}
 \qbinom{a-d+t}{t}_{\!q}
\begin{tikzpicture}[anchorbase,scale=1]
	\draw[-,thick] (0.88,0) to (0.88,.2) to (.02,.8) to (.02,1);
	\draw[-,line width=4pt,white] (0.02,0) to (0.02,.2) to (.88,.8) to (.88,1);
	\draw[-,thick] (0.02,0) to (0.02,.2) to (.88,.8) to (.88,1);
	\draw[-,thin] (0,0) to (0,1);
	\draw[-,thin] (0.9,0) to (0.9,1);
        \node at (0,-.1) {$\scriptstyle a$};
        \node at (.9,-.1) {$\scriptstyle b$};
        \node at (.33,.18) {$\scriptstyle d-t$};
        \node at (.34,.85) {$\scriptstyle c-t$};
\end{tikzpicture}.
$$
The second equality in 
\cref{jonsquare2} follows by applying $\mathtt{R}$.

\vspace{2mm}
\noindent
{\bf\cref{thickcrossing}.}
Set $d := a$ and $c := b$ in the second identity from \cref{mergesplit}.
The $s=t=0$
term on the right hand side of the resulting identity 
is equal to the crossing
$\begin{tikzpicture}[baseline=-1.5mm,scale=.6]
	\draw[-,thick] (0.3,-.35) to (-.3,.43);
	\draw[-,line width=5pt,white] (-0.3,-.35) to (.3,.43);
	\draw[-,thick] (-0.3,-.3) to (.3,.4);
       \node at (-0.32,-.53) {$\scriptstyle a$};
        \node at (0.32,-.54) {$\scriptstyle b$};
\end{tikzpicture}$,
as follows by contracting
strings of thickness zero in the manner explained in the comments 
after the statement of the theorem.
Rearranging to make this term the subject proves
the first equality in \cref{thickcrossing}.
Also the final equality may be deduced from from the others by
applying $\mathtt{R}$. 

It remains to establish the middle equality. If $a=0$ or $b=0$, it follows by
contracting the strings of thickness zero.
So we may assume that $a,b \geq 1$.
Then we proceed by induction on $a+b$. The base case $a=b=1$ follows by contracting strings of thickness zero once again.
For the induction step, we have by first the equality and the induction hypothesis that
$$
\begin{tikzpicture}[anchorbase,scale=1.2]
	\draw[-,line width=1pt] (0.3,-.3) to (-.3,.4);
	\draw[-,line width=4pt,white] (-0.3,-.3) to (.3,.4);
	\draw[-,line width=1pt] (-0.3,-.3) to (.3,.4);
        \node at (0.3,-.42) {$\scriptstyle b$};
        \node at (-0.3,-.42) {$\scriptstyle a$};
\end{tikzpicture}
=
\begin{tikzpicture}[baseline=3mm,scale=.9]
	\draw[-,line width=1.2pt] (.6,0) to (.315,.3) to (.315,.7) to (.6,1);
	\draw[-,line width=1.2pt] (0,0) to (.285,.3) to (.285,.7) to (0,1);
        \node at (0,-.1) {$\scriptstyle a$};
        \node at (0.6,-.1) {$\scriptstyle b$};
    \node at (0,1.13) {$\scriptstyle b$};
        \node at (0.63,1.13) {$\scriptstyle a$};
        \end{tikzpicture}
-\sum_{s=1}^{\min(a,b)}q^{s^2}
\begin{tikzpicture}[anchorbase,scale=1]
	\draw[-,thick] (0.58,0) to (0.58,.2) to (.02,.8) to (.02,1);
	\draw[-,line width=4pt,white] (0.02,0) to (0.02,.2) to (.58,.8) to (.58,1);
	\draw[-,thick] (0.02,0) to (0.02,.2) to (.58,.8) to (.58,1);
	\draw[-,thin] (0,0) to (0,1);
	\draw[-,thin] (0.6,0) to (0.6,1);
        \node at (0,-.1) {$\scriptstyle a$};
        \node at (0.6,-.1) {$\scriptstyle b$};
        \node at (-0.1,.5) {$\scriptstyle s$};
        \node at (0.7,.5) {$\scriptstyle s$};
\end{tikzpicture}
=
\begin{tikzpicture}[baseline=3mm,scale=.9]
	\draw[-,line width=1.2pt] (.6,0) to (.315,.3) to (.315,.7) to (.6,1);
	\draw[-,line width=1.2pt] (0,0) to (.285,.3) to (.285,.7) to (0,1);
        \node at (0,-.1) {$\scriptstyle a$};
        \node at (0.6,-.1) {$\scriptstyle b$};
    \node at (0,1.13) {$\scriptstyle b$};
        \node at (0.63,1.13) {$\scriptstyle a$};
        \end{tikzpicture}
- \sum_{s=1}^{\min(a,b)}
q^{s^2}\sum_{t=0}^{\min(a,b)-s}
(-q)^{t}
\begin{tikzpicture}[anchorbase,scale=1]
	\draw[-,thin] (0,0) to (0,1);
	\draw[-,thick] (0.02,0) to (0.02,.2) to (.7,.3) to (.7,.7) to (1.08,.8) to (1.08,1);
	\draw[-,thick] (1.08,0) to (1.08,.2) to (.7,.3) to (.7,.7) to (.02,.8) to (.02,1);
	\draw[-,thin] (.01,0) to (0.01,.22) to (.35,.27) to (.35,.73) to
        (0.01,.78) to (.01,1);
	\draw[-,thin] (1.1,0) to (1.1,1);
 \draw[-,thin] (.355,.268) to (.355,.732);
        \node at (0,-.1) {$\scriptstyle a$};
        \node at (1.13,-.1) {$\scriptstyle b$};
        \node at (-0.1,.5) {$\scriptstyle s$};
        \node at (1.25,.5) {$\scriptstyle s$};
        \node at (0.25,.5) {$\scriptstyle t$};
\end{tikzpicture}\!.
$$
By (a) from 
the lemma at the start of this appendix (applicable since \cref{jonsquare2}
implies the second square-switch relation), we have that
\begin{equation*}
\begin{tikzpicture}[anchorbase,scale=1]
	\draw[-,line width=1.2pt] (0,0) to (.285,.3) to (.285,.7) to (0,1);
	\draw[-,line width=1.2pt] (.6,0) to (.325,.3) to (.325,.7) to (.6,1);
        \node at (0,1.13) {$\scriptstyle b$};
        \node at (0.6,1.13) {$\scriptstyle a$};
        \node at (0,-.1) {$\scriptstyle a$};
        \node at (0.6,-.1) {$\scriptstyle b$};
\end{tikzpicture}
=
\sum_{s=0}^{\min(a,b)}
q^{s^2}
\sum_{t=0}^{\min(a,b)-s}
(-q)^{t}
\begin{tikzpicture}[anchorbase,scale=1]
	\draw[-,thin] (0,0) to (0,1);
	\draw[-,thick] (0.02,0) to (0.02,.2) to (.7,.3) to (.7,.7) to (1.08,.8) to (1.08,1);
	\draw[-,thick] (1.08,0) to (1.08,.2) to (.7,.3) to (.7,.7) to (.02,.8) to (.02,1);
	\draw[-,thin] (.01,0) to (0.01,.22) to (.35,.27) to (.35,.73) to
        (0.01,.78) to (.01,1);
	\draw[-,thin] (1.1,0) to (1.1,1);
        \node at (0,1.13) {$\scriptstyle b$};
        \node at (1.13,1.13) {$\scriptstyle a$};
        \node at (0,-.1) {$\scriptstyle a$};
        \node at (1.13,-.1) {$\scriptstyle b$};
        \node at (-0.1,.5) {$\scriptstyle s$};
        \node at (1.2,.5) {$\scriptstyle s$};
        \node at (0.25,.5) {$\scriptstyle t$};
\end{tikzpicture}.
\end{equation*}
Using this, the previous equation simplifies to give
$$
\begin{tikzpicture}[anchorbase,scale=1.2]
	\draw[-,line width=1pt] (0.3,-.3) to (-.3,.4);
	\draw[-,line width=4pt,white] (-0.3,-.3) to (.3,.4);
	\draw[-,line width=1pt] (-0.3,-.3) to (.3,.4);
        \node at (0.3,-.42) {$\scriptstyle b$};
        \node at (-0.3,-.42) {$\scriptstyle a$};
\end{tikzpicture}
=
\sum_{t=0}^{\min(a,b)}
(-q)^{t}
\begin{tikzpicture}[anchorbase,scale=1]
	\draw[-,line width=0pt] (0,0) to (0,1);
	\draw[-,thick] (0.02,0) to (0.02,.2) to (.7,.3) to (.7,.7) to (1.08,.8) to (1.08,1);
	\draw[-,thick] (1.08,0) to (1.08,.2) to (.7,.3) to (.7,.7) to (.02,.8) to (.02,1);
	\draw[-,thin] (.01,0) to (0.01,.22) to (.35,.27) to (.35,.73) to
        (0.01,.78) to (.01,1);
        \draw[-,thin] (.355,.268) to (.355,.732);
	\draw[-,line width=0pt] (1.1,0) to (1.1,1);
        \node at (0,-.1) {$\scriptstyle a$};
        \node at (1.13,-.1) {$\scriptstyle b$};
        \node at (-0.1,.5) {$\scriptstyle 0$};
        \node at (1.25,.5) {$\scriptstyle 0$};
        \node at (0.25,.5) {$\scriptstyle t$};
\end{tikzpicture}
=\sum_{t=0}^{\min(a,b)}
(-q)^{t}
\begin{tikzpicture}[anchorbase,scale=1]
	\draw[-,thick] (0.02,0) to (0.02,.2) to (.7,.3) to (.7,.7) to (1.08,.8) to (1.08,1);
	\draw[-,thick] (1.08,0) to (1.08,.2) to (.7,.3) to (.7,.7) to (.02,.8) to (.02,1);
	\draw[-,thin] (.01,0) to (0.01,.22) to (.35,.27) to (.35,.73) to
        (0.01,.78) to (.01,1);
        \node at (0,-.1) {$\scriptstyle a$};
        \node at (1.13,-.1) {$\scriptstyle b$};
        \node at (0.25,.5) {$\scriptstyle t$};
\end{tikzpicture}=
\sum_{t=0}^{\min(a,b)}
(-q)^{t}
\begin{tikzpicture}[anchorbase,scale=1]
	\draw[-,thin] (0,0) to (0,1);
	\draw[-,thick] (0.015,0) to (0.015,.2) to (.57,.4) to (.57,.6)
        to (.015,.8) to (.015,1);
	\draw[-,line width=1.2pt] (0.59,0) to (0.59,1);
        \node at (0.59,-.1) {$\scriptstyle b$};
        \node at (0,-.1) {$\scriptstyle a$};
        \node at (-0.1,.5) {$\scriptstyle t$};
\end{tikzpicture},
$$
as required for the induction step.

\iffalse
\vspace{2mm}
\noindent
{\bf\cref{serre}.}
This is explained in the proof of \cite[Lemma 2.2.1]{CKM}.
\fi

\vspace{2mm}
\noindent
{\bf\cref{sliders}.}
Since it will be needed shortly,  we next argue that 
 $\mathtt{R}$ and $\T$
both map 
$\begin{tikzpicture}[baseline=-1.5mm,scale=.6]
	\draw[-,thick] (-0.3,-.35) to (.3,.43);
	\draw[-,line width=5pt,white] (0.3,-.35) to (-.3,.43);
	\draw[-,thick] (0.3,-.3) to (-.3,.4);
       \node at (-0.32,-.53) {$\scriptstyle a$};
        \node at (0.32,-.54) {$\scriptstyle b$};
\end{tikzpicture}
\mapsto \begin{tikzpicture}[baseline=-1.5mm,scale=.6]
	\draw[-,thick] (-0.3,-.35) to (.3,.43);
	\draw[-,line width=5pt,white] (0.3,-.35) to (-.3,.43);
	\draw[-,thick] (0.3,-.3) to (-.3,.4);
       \node at (-0.32,-.53) {$\scriptstyle b$};
        \node at (0.32,-.54) {$\scriptstyle a$};
\end{tikzpicture}$.
For $\T$, 
this is obvious from the definition \cref{tireddog}.
To see it for $\mathtt{R}$, we need to show that
$$
\sum_{s=0}^{\min(a,b)}
(-q)^{-s}
\begin{tikzpicture}[anchorbase,scale=1]
	\draw[-,line width=1.2pt] (0,0) to (0,1);
	\draw[-,thick] (0.8,0) to (0.8,.2) to (.03,.4) to (.03,.6)
        to (.8,.8) to (.8,1);
	\draw[-,thin] (0.82,0) to (0.82,1);
        \node at (0.81,-.1) {$\scriptstyle a$};
        \node at (0,-.1) {$\scriptstyle b$};
        \node at (0.4,.9) {$\scriptstyle b-s$};
        \node at (0.4,.13) {$\scriptstyle a-s$};
\end{tikzpicture}
=
\sum_{s=0}^{\min(a,b)}
(-q)^{-s}
\begin{tikzpicture}[anchorbase,scale=1]
	\draw[-,line width=1.2pt] (0,0) to (0,1);
	\draw[-,thick] (-0.8,0) to (-0.8,.2) to (-.03,.4) to (-.03,.6)
        to (-.8,.8) to (-.8,1);
	\draw[-,thin] (-0.82,0) to (-0.82,1);
        \node at (-0.81,-.1) {$\scriptstyle b$};
        \node at (0,-.1) {$\scriptstyle a$};
        \node at (-0.4,.9) {$\scriptstyle a-s$};
        \node at (-0.4,.13) {$\scriptstyle b-s$};
\end{tikzpicture}\ .
$$
This follows because we have simply that
$$
\begin{tikzpicture}[anchorbase,scale=1]
	\draw[-,line width=1.2pt] (0,0) to (0,1);
	\draw[-,thick] (0.8,0) to (0.8,.2) to (.03,.4) to (.03,.6)
        to (.8,.8) to (.8,1);
	\draw[-,thin] (0.82,0) to (0.82,1);
        \node at (0.81,-.1) {$\scriptstyle a$};
        \node at (0,-.1) {$\scriptstyle b$};
        \node at (0.4,.9) {$\scriptstyle b-s$};
        \node at (0.4,.13) {$\scriptstyle a-s$};
\end{tikzpicture}
=
\sum_{t=0}^{\min(a,b)-s}
\qbinom{0}{t}_{\!q}
\begin{tikzpicture}[anchorbase,scale=1.3]
	\draw[-,line width=1.2pt] (-0.02,0) to (-0.02,1);
	\draw[-,thick] (-0.8,0) to (-0.8,.2) to (-.03,.4) to (-.03,.6)
        to (-.8,.8) to (-.8,1);
	\draw[-,thin] (-0.81,0) to (-0.81,1);
        \node at (-0.81,-.1) {$\scriptstyle b$};
        \node at (0,-.1) {$\scriptstyle a$};
        \node at (-0.4,.9) {$\scriptstyle a-s-t$};
        \node at (-0.4,.13) {$\scriptstyle b-s-t$};
\end{tikzpicture}
=\begin{tikzpicture}[anchorbase,scale=1]
	\draw[-,line width=1.2pt] (0,0) to (0,1);
	\draw[-,thick] (-0.8,0) to (-0.8,.2) to (-.03,.4) to (-.03,.6)
        to (-.8,.8) to (-.8,1);
	\draw[-,thin] (-0.82,0) to (-0.82,1);
        \node at (-0.81,-.1) {$\scriptstyle b$};
        \node at (0,-.1) {$\scriptstyle a$};
        \node at (-0.4,.9) {$\scriptstyle a-s$};
        \node at (-0.4,.13) {$\scriptstyle b-s$};
\end{tikzpicture}
$$
by \cref{jonsquare2}.
Now for \cref{sliders}, the 
four identities are all equivalent since we can
rotate in
horizontal and/or vertical axes using $\mathtt{R}$ and $\T$. So it suffices to prove the first one:
$$
\begin{tikzpicture}[anchorbase,scale=0.7]
	\draw[-,thick] (0.4,0) to (-0.6,1);
	\draw[-,line width=4pt,white] (0.1,0) to (0.1,.6) to (.5,1);
	\draw[-,thick] (0.1,0) to (0.1,.6) to (.5,1);
	\draw[-,thick] (0.08,0) to (0.08,1);
        \node at (0.6,1.13) {$\scriptstyle c$};
        \node at (0.1,1.16) {$\scriptstyle b$};
        \node at (-0.65,1.13) {$\scriptstyle a$};
\end{tikzpicture}
=
\begin{tikzpicture}[anchorbase,scale=0.7]
	\draw[-,thick] (0.7,0) to (-0.3,1);
	\draw[-,line width=4pt,white] (0.1,0) to (0.1,.2) to (.9,1);
	\draw[-,line width=3pt,white] (0.08,0) to (0.08,1);
	\draw[-,thick] (0.1,0) to (0.1,.2) to (.9,1);
	\draw[-,thick] (0.08,0) to (0.08,1);
        \node at (0.9,1.13) {$\scriptstyle c$};
        \node at (0.1,1.16) {$\scriptstyle b$};
        \node at (-0.4,1.13) {$\scriptstyle a$};
\end{tikzpicture}.
$$
We proceed by induction on $a+b+c$. The base
case is when $a=0$, which is obvious since we can contract the string of thickness zero. 
For the induction step, 
we rewrite the crossing at the bottom
of the diagram on the left hand side using \cref{thickcrossing}:
$$
\begin{tikzpicture}[anchorbase,scale=0.7]
	\draw[-,thick] (0.4,0) to (-0.6,1);
	\draw[-,line width=4pt,white] (0.1,0) to (0.1,.6) to (.5,1);
	\draw[-,thick] (0.1,0) to (0.1,.6) to (.5,1);
	\draw[-,thick] (0.08,0) to (0.08,1);
        \node at (0.6,1.13) {$\scriptstyle c$};
        \node at (0.1,1.16) {$\scriptstyle b$};
        \node at (-0.65,1.13) {$\scriptstyle a$};
\end{tikzpicture}
 =
\sum_{s=0}^{\min(a,b+c)} (-q)^{s}
\begin{tikzpicture}[anchorbase,scale=.8]
	\draw[-,thin] (0.01,0) to (0.01,1);
	\draw[-,thick] (0.02,0) to (0.02,0.2) to (.88,0.4) to (.88,0.6) to
        (0.02,.8) to (0.02,1);
	\draw[-,thick] (0.03,0) to (0.03,0.2) to (.91,0.4) to (.91,1);
	\draw[-,thick] (.92,0) to (.92,0.8) to (1.3,1);
        \node at (0,-.15) {$\scriptstyle b+c$};
        \node at (0,1.15) {$\scriptstyle a$};
        \node at (-.1,.5) {$\scriptstyle s$};
        \node at (.9,1.15) {$\scriptstyle b$};
        \node at (1.4,1.15) {$\scriptstyle c$};
        \node at (.9,-.15) {$\scriptstyle a$};
\end{tikzpicture}
 =
\sum_{s=0}^{\min(a,b+c)} (-q)^{s}
\begin{tikzpicture}[anchorbase,scale=.8]
	\draw[-,thin] (0.01,0) to (0.01,1);
	\draw[-,thick] (0.02,0) to (0.02,0.2) to (.695,0.3) to (.695,0.7) to
        (0.02,.8) to (0.02,1);
	\draw[-,thick] (0.03,0) to (0.03,0.2) to (.71,0.3) to (.71,1);
	\draw[-,thick] (.73,0) to (.73,0.5) to (1.4,1);
        \node at (0,-.15) {$\scriptstyle b+c$};
        \node at (0,1.15) {$\scriptstyle a$};
        \node at (-.1,.5) {$\scriptstyle s$};
        \node at (.7,1.15) {$\scriptstyle b$};
        \node at (1.4,1.15) {$\scriptstyle c$};
        \node at (.73,-.15) {$\scriptstyle a$};
\end{tikzpicture}.
$$
By the second equation from 
\cref{mergesplit}, we have that
$$
\begin{tikzpicture}[anchorbase,scale=.8]
	\draw[-,line width=1.2pt] (0,0) to (.275,.3) to (.275,.7) to (0,1);
	\draw[-,line width=1.2pt] (.6,0) to (.315,.3) to (.315,.7) to (.6,1);
        \node at (-0.2,1.13) {$\scriptstyle a+b-s$};
        \node at (0.63,1.13) {$\scriptstyle c$};
        \node at (-0.2,-.1) {$\scriptstyle b+c-s$};
        \node at (0.63,-.1) {$\scriptstyle a$};
\end{tikzpicture}
=
\sum_{t=\max(0,s-b)}^{\min(a,c)}
q^{t(b+t-s)}\begin{tikzpicture}[anchorbase,scale=.8]
	\draw[-,thick] (1.08,0) to (1.08,.2) to (.02,.8) to (.02,1);
	\draw[-,line width=4pt,white] (0.02,0) to (0.02,.2) to (1.08,.8) to (1.08,1);
	\draw[-,thick] (0.02,0) to (0.02,.2) to (1.08,.8) to (1.08,1);
	\draw[-,thin] (0,0) to (0,1);
	\draw[-,thin] (1.1,0) to (1.1,1);
        \node at (-.2,1.13) {$\scriptstyle a+b-s$};
        \node at (1.13,1.13) {$\scriptstyle c$};
        \node at (-.2,-.1) {$\scriptstyle b+c-s$};
        \node at (1.13,-.1) {$\scriptstyle a$};
        \node at (1.23,.5) {$\scriptstyle t$};
        \node at (-.6,.5) {$\scriptstyle b+t-s$};
\end{tikzpicture}.
$$
We substitute this into our formula to obtain
\begin{equation*}
\begin{tikzpicture}[anchorbase,scale=0.7]
	\draw[-,thick] (0.4,0) to (-0.6,1);
	\draw[-,line width=4pt,white] (0.1,0) to (0.1,.6) to (.5,1);
	\draw[-,thick] (0.1,0) to (0.1,.6) to (.5,1);
	\draw[-,thick] (0.08,0) to (0.08,1);
        \node at (0.6,1.13) {$\scriptstyle c$};
        \node at (0.1,1.16) {$\scriptstyle b$};
        \node at (-0.65,1.13) {$\scriptstyle a$};
\end{tikzpicture}
 =
\sum_{s=0}^{\min(a,b+c)} 
\sum_{t=\max(0,s-b)}^{\min(a,c)}(-q)^{s}
q^{t(b+t-s)}
\begin{tikzpicture}[anchorbase,scale=.8]
	\draw[-,thick] (1.05,0)  to (1.05,.05) to (.345,.5) to (.465,.7)
        to (0.03,.85) to (0.03,1);
	\draw[-,line width=3pt,white] (0.04,0) to (0.04,.05) to (1.4,.96) to (1.4,1);
	\draw[-,thin] (0.04,0) to (0.04,.05) to (1.4,.96) to (1.4,1);
	\draw[-,thick] (0.03,0) to (0.03,.05) to (.64,1);
	\draw[-,thin] (1.07,0) to (1.07,.05) to (1.42,.96) to (1.42,1);
	\draw[-,thin] (0,0) to (0,1);
	\draw[-,thick] (0.01,0) to (0.01,.05) to (.42,.7) to (0.01,.85)
        to (0.01,1);
        \node at (0,-.15) {$\scriptstyle b+c$};
        \node at (0,1.15) {$\scriptstyle a$};
        \node at (-.1,.5) {$\scriptstyle s$};
        \node at (1.4,.5) {$\scriptstyle t$};
        \node at (.7,1.18) {$\scriptstyle b$};
        \node at (1.4,1.15) {$\scriptstyle c$};
        \node at (1.05,-.15) {$\scriptstyle a$};
\end{tikzpicture}.
\end{equation*}
By the second equation in \cref{mergesplit} again, we have that
$$
\begin{tikzpicture}[anchorbase,scale=.8]
	\draw[-,line width=1.2pt] (0,0) to (.275,.3) to (.275,.7) to (0,1);
	\draw[-,line width=1.2pt] (.6,0) to (.315,.3) to (.315,.7) to (.6,1);
        \node at (-0.1,1.15) {$\scriptstyle a-s$};
        \node at (0.63,1.15) {$\scriptstyle b$};
        \node at (-0.3,-.1) {$\scriptstyle b+t-s$};
        \node at (0.7,-.1) {$\scriptstyle a-t$};
\end{tikzpicture}
=
\sum_{u=\max(s,t)}^{\min(a,b+t)}
q^{(u-s)(u-t)}
\begin{tikzpicture}[anchorbase,scale=.8]
	\draw[-,thick] (1.08,0) to (1.08,.2) to (.02,.8) to (.02,1);
	\draw[-,line width=4pt,white] (0.02,0) to (0.02,.2) to (1.08,.8) to (1.08,1);
	\draw[-,thick] (0.02,0) to (0.02,.2) to (1.08,.8) to (1.08,1);
	\draw[-,thin] (0,0) to (0,1);
	\draw[-,thin] (1.1,0) to (1.1,1);
        \node at (-.2,1.15) {$\scriptstyle a-s$};
        \node at (1.13,1.15) {$\scriptstyle b$};
        \node at (-.2,-.1) {$\scriptstyle b+t-s$};
        \node at (1.13,-.1) {$\scriptstyle a-t$};
        \node at (1.43,.5) {$\scriptstyle u-t$};
        \node at (-.4,.5) {$\scriptstyle u-s$};
\end{tikzpicture}.
$$
Using this, \cref{assrel} the first equation in \cref{mergesplit},
and the induction hypothesis to pull a 
split past the string of thickness $c-t$,
we simplify further to get
\begin{align*}
\begin{tikzpicture}[anchorbase,scale=0.7]
	\draw[-,thick] (0.4,0) to (-0.6,1);
	\draw[-,line width=4pt,white] (0.1,0) to (0.1,.6) to (.5,1);
	\draw[-,thick] (0.1,0) to (0.1,.6) to (.5,1);
	\draw[-,thick] (0.08,0) to (0.08,1);
        \node at (0.6,1.13) {$\scriptstyle c$};
        \node at (0.1,1.16) {$\scriptstyle b$};
        \node at (-0.65,1.13) {$\scriptstyle a$};
\end{tikzpicture}
& =
\sum_{s=0}^{\min(a,b+c)} \sum_{t=\max(0,s-b)}^{\min(a,c)}
\sum_{u=\max(s,t)}^{\min(a,b+t)}
(-q)^{s} q^{t(b+t-s)} q^{(u-s)(u-t)}\qbinom{u}{s}_{\!q}
\begin{tikzpicture}[anchorbase,scale=.8]
	\draw[-,thin] (0.01,1) to (.01,.95) to (1.035,.05) to (1.035,0);
	\draw[-,thin] (1.03,0) to (1.03,.05) to (.705,.95) to (.705,1);
	\draw[-,line width=3pt,white] (0.35,0) to (.35,.05) to (.695,.95) to (.695,1);
	\draw[-,thin] (1.03,0) to (1.03,.05) to (.705,.95) to (.705,1);
	\draw[-,line width=3pt,white] (0.36,0) to (.36,.05) to (1.39,.95) to (1.39,1);
	\draw[-,thin] (0.36,0) to (.36,.05) to (1.39,.95) to (1.39,1);
	\draw[-,thin] (0.35,0) to (.35,.05) to (.695,.95) to (.695,1);
	\draw[-,thin] (1.04,0) to (1.04, .05) to (1.4,.95) to (1.4,1);
	\draw[-,thin] (0.34,0) to (.34,.05) to (0,.95) to (0,1);
        \node at (0.35,-.15) {$\scriptstyle b+c$};
        \node at (0,1.15) {$\scriptstyle a$};
        \node at (-.07,.5) {$\scriptstyle u$};
        \node at (1.4,.5) {$\scriptstyle t$};
        \node at (.7,1.18) {$\scriptstyle b$};
        \node at (1.4,1.15) {$\scriptstyle c$};
        \node at (1.05,-.15) {$\scriptstyle a$};
\end{tikzpicture}\\
&=
\sum_{t=0}^{\min(a,c)}
\sum_{u=t}^{\min(a,b+t)}
q^{-u(t-u)+t(b+t)}
\left(\sum_{s=0}^u 
(-1)^sq^{-s(u-1)} \qbinom{u}{s}_{\!q}\right)
\begin{tikzpicture}[anchorbase,scale=.8]
	\draw[-,thin] (0.01,1) to (.01,.95) to (1.035,.05) to (1.035,0);
	\draw[-,thin] (1.03,0) to (1.03,.05) to (.705,.95) to (.705,1);
	\draw[-,line width=3pt,white] (0.35,0) to (.35,.05) to (.695,.95) to (.695,1);
	\draw[-,thin] (1.03,0) to (1.03,.05) to (.705,.95) to (.705,1);
	\draw[-,line width=3pt,white] (0.36,0) to (.36,.05) to (1.39,.95) to (1.39,1);
	\draw[-,thin] (0.36,0) to (.36,.05) to (1.39,.95) to (1.39,1);
	\draw[-,thin] (0.35,0) to (.35,.05) to (.695,.95) to (.695,1);
	\draw[-,thin] (1.04,0) to (1.04, .05) to (1.4,.95) to (1.4,1);
	\draw[-,thin] (0.34,0) to (.34,.05) to (0,.95) to (0,1);
        \node at (0.35,-.15) {$\scriptstyle b+c$};
        \node at (0,1.15) {$\scriptstyle a$};
        \node at (-.07,.5) {$\scriptstyle u$};
        \node at (1.4,.5) {$\scriptstyle t$};
        \node at (.7,1.18) {$\scriptstyle b$};
        \node at (1.4,1.15) {$\scriptstyle c$};
        \node at (1.05,-.15) {$\scriptstyle a$};
\end{tikzpicture}.
\end{align*}
By \cref{vanishing}, the expression in parentheses is $\delta_{u,0}$,
so the only non-zero term arises when $u=t=0$, and we get 
exactly the right hand side we were after.

\vspace{2mm}
\noindent
{\bf\cref{swallows-symmetric}.}
Consider the first two relations.
Since we can apply $\mathtt{R}$ and $\T$, 
it suffices to prove the first equality
in the special case that $a \geq b$.
Replacing the crossing using \cref{tireddog},
then applying \cref{assrel} and the first relation from \cref{mergesplit} as usual, we have that
$$
\begin{tikzpicture}[anchorbase,scale=.7]
\draw[-,thick] (-.2,-.8) to [out=45,in=-45] (0.1,.31);
\draw[-,line width=4.3pt,white] (.36,-.8) to [out=135,in=-135] (0.06,.31);
\draw[-,thick] (.36,-.8) to [out=135,in=-135] (0.06,.31);
	\draw[-,line width=2pt] (0.08,.3) to (0.08,.5);
        \node at (-.3,-.95) {$\scriptstyle a$};
        \node at (.45,-.95) {$\scriptstyle b$};
\end{tikzpicture}
=
\sum_{s=0}^{b} (-q)^{-s}
\begin{tikzpicture}[anchorbase,scale=.7]
	\draw[-,line width=1.8pt] (0.08,.1) to (0.08,.5);
\draw[-,thick] (.47,-.8) to [out=100,in=-45] (0.09,.115);
\draw[-,thin] (-.29,-.8) to [out=80,in=-135] (0.05,.115);
\draw[-,thick] (-.275,-.8) to (-.26,-.7) to (.425,-.6) to (.39,-.45) to (-.18,-.3) to [out=70,in=-135] (0.07,.115);
        \node at (-.3,-.95) {$\scriptstyle a$};
        \node at (.43,-.95) {$\scriptstyle b$};
        \node at (-.26,-.02) {$\scriptstyle b$};
        \node at (.41,-.06) {$\scriptstyle a$};
        \node at (-.4,-.45) {$\scriptstyle s$};
\end{tikzpicture}
=
\left(\sum_{s=0}^{b}(-q)^{-s} \qbinom{a+b-s}{b-s}_{\!q}
\qbinom{a}{s}_{\!q} 
\right)\begin{tikzpicture}[anchorbase,scale=.7]
	\draw[-,line width=2pt] (0.08,.1) to (0.08,.5);
\draw[-,thick] (.46,-.8) to [out=100,in=-45] (0.1,.11);
\draw[-,thick] (-.3,-.8) to [out=80,in=-135] (0.06,.11);
        \node at (-.3,-.95) {$\scriptstyle a$};
        \node at (.43,-.95) {$\scriptstyle b$};
\end{tikzpicture}.
$$
It remains to observe that the coefficient here equals $q^{ab}$. 
This follows from \cref{A},
taking $a,b,m$ and $s$ there to be
$b-s,s,a$ and $b$ in the present setup.

Now consider the third and fourth relations in \cref{swallows-symmetric}.
Since we can apply $\mathtt{R}$ and $\mathtt{T}$,
it suffices just to prove the third one, and we can assume moreover that $a \leq b$.
We proceed by induction on $a$. The base case $a=0$
follows by contracting the string of thickness zero.
For the induction step, suppose that $a \geq 1$.
We claim for $0 \leq s < a$ that 
\begin{equation*}
\begin{tikzpicture}[anchorbase,scale=.9]
  \draw[-,thin] (.274,-.8) to (.274,-.65) to (-.279,-.29);
  \draw[-,line width=3pt,white] (-.274,-.8) to (-.274,-.65) to (.279,-.29);  \draw[-,thin] (-.274,-.8) to (-.274,-.65) to (.279,-.29);
  \draw[-,thin] (-.289,-.8) to (-.289,-.29);
  \draw[-,thin] (.289,-.8) to (.289,-.29);
	\draw[-,thick] (-0.28,-.3) to [out=90,in=-130] (0,.2) to [out=50,in=-90] (0.28,.6);
	\draw[-,white,line width=4pt] (0.28,-.3) to [out=90,in=-50] (0,.2) to [out=130,in=-90] (-0.28,.6);
  \draw[-,thick] (0.28,-.3) to [out=90,in=-50] (0,.2) to [out=130,in=-90] (-0.28,.6);
        \node at (0.3,-.95) {$\scriptstyle b$};
        \node at (-0.3,-.95) {$\scriptstyle a$};
        \node at (-.66,-.45) {$\scriptstyle a-s$};
         \node at (.66,-.45) {$\scriptstyle a-s$};      \node at (0.3,.77) {$\scriptstyle b$};
        \node at (-0.3,.75) {$\scriptstyle a$};
\end{tikzpicture}
=
q^{(a-s)(b-a+2s)}
\begin{tikzpicture}[anchorbase,scale=.9]
	\draw[-,thick] (0.02,.2) to (0.02,.3) to (.58,.9) to (.58,1);
 	\draw[-,line width=4pt,white] (.58,0.2) to (.58,.3) to (.02,.9) to (.02,1);
	\draw[-,thick] (.58,0.2) to (.58,.3) to (.02,.9) to (.02,1);
 \draw[-,thin] (0,0.2) to (0,1);
	\draw[-,thin] (.6,0.2) to (.6,1);
        \node at (0,.05) {$\scriptstyle a$};
        \node at (.58,.05) {$\scriptstyle b$};
        \node at (-0.16,.6) {$\scriptstyle s$};
       \node at (0.6,1.17) {$\scriptstyle b$};
        \node at (0,1.15) {$\scriptstyle a$};
\end{tikzpicture}.
\end{equation*}
To see this, one uses \cref{sliders},
the second relation from \cref{swallows-symmetric} and the induction hypothesis to
move the merges up past the crossing.
Using \cref{thickcrossing} and the second relation from \cref{swallows-symmetric},
we have that
\begin{align*}
\begin{tikzpicture}[anchorbase,scale=0.8]
	\draw[-,thick] (0.28,-.8) to[out=90,in=-90] (-0.28,-0.1);
	\draw[-,line width=4pt,white] (-0.28,-.8) to[out=90,in=-90] (0.28,-0.1);
	\draw[-,thick] (-0.28,-.8) to[out=90,in=-90] (0.28,-0.1);
	\draw[-,thick] (-0.28,-0.1) to[out=90,in=-90] (0.28,.6);
	\draw[-,line width=4pt,white] (0.28,-0.1) to[out=90,in=-90] (-0.28,.6);
	\draw[-,thick] (0.28,-0.1) to[out=90,in=-90] (-0.28,.6);
        \node at (0.3,-.95) {$\scriptstyle b$};
        \node at (-0.3,-.95) {$\scriptstyle a$};
       \node at (0.3,.77) {$\scriptstyle b$};
        \node at (-0.3,.75) {$\scriptstyle a$};
\end{tikzpicture}
&=
\begin{tikzpicture}[anchorbase,scale=0.8]
	\draw[-,thick] (-0.28,0) to[out=90,in=-90] (0.28,.6);	
  \draw[-,line width=4pt,white] (0.28,0) to[out=90,in=-90] (-0.28,.6);
\draw[-,thick] (0.28,0) to[out=90,in=-90] (-0.28,.6);
	\draw[-,thick] (-.3,-.8) to (-.01,-.5) to (-0.01,-.3) to[out=135,in=-90] (-0.28,0);
	\draw[-,thick] (0.3,-.8) to (.01,-.5) to (0.01,-.3) to[out=45,in=-90] (0.28,0);
        \node at (0.3,-.95) {$\scriptstyle b$};
        \node at (-0.3,-.95) {$\scriptstyle a$};
       \node at (0.3,.77) {$\scriptstyle b$};
        \node at (-0.3,.75) {$\scriptstyle a$};
\end{tikzpicture}
-\sum_{t=1}^{a}
q^{t^2}
\begin{tikzpicture}[anchorbase,scale=0.8]
  \draw[-,thin] (.274,-.8) to (.274,-.65) to (-.279,-.29);
  \draw[-,line width=3pt,white] (-.274,-.8) to (-.274,-.65) to (.279,-.29);
 \draw[-,thin] (-.274,-.8) to (-.274,-.65) to (.279,-.29);
	\draw[-,thick] (-0.28,-.3) to [out=90,in=-130] (0,.2) to [out=50,in=-90] (0.28,.6);
	\draw[-,line width=4pt,white] (0.28,-.3) to [out=90,in=-50] (0,.2) to [out=130,in=-90] (-0.28,.6);
  	\draw[-,thick] (0.28,-.3) to [out=90,in=-50] (0,.2) to [out=130,in=-90] (-0.28,.6);
\draw[-,thin] (.289,-.8) to (.289,-.29);
  \draw[-,thin] (-.289,-.8) to (-.289,-.29);
        \node at (0.3,-.95) {$\scriptstyle b$};
        \node at (-0.3,-.95) {$\scriptstyle a$};
       \node at (0.3,.77) {$\scriptstyle b$};
        \node at (-0.3,.75) {$\scriptstyle a$};
        \node at (-0.43,-.45) {$\scriptstyle t$};
\end{tikzpicture}=
q^{ab}
\begin{tikzpicture}[anchorbase,scale=0.8]
	\draw[-,thick] (-.3,-.8) to (-.01,-.3) to (-0.01,-.3) to (0.01,.1) to (-.3,.6);
	\draw[-,thick] (0.3,-.8) to (.01,-.3) to (0.01,-.3) to
        (0.01,.1) to (.3,.6);
        \node at (0.3,-.95) {$\scriptstyle b$};
        \node at (-0.3,-.95) {$\scriptstyle a$};
       \node at (0.3,.77) {$\scriptstyle b$};
        \node at (-0.3,.75) {$\scriptstyle a$};
\end{tikzpicture}
-\sum_{s=0}^{a-1}
q^{(a-s)^2}
\begin{tikzpicture}[anchorbase,scale=0.8]
  \draw[-,thin] (.274,-.8) to (.274,-.65) to (-.279,-.29);
  \draw[-,line width=3pt,white] (-.274,-.8) to (-.274,-.65) to (.279,-.29);  \draw[-,thin] (-.274,-.8) to (-.274,-.65) to (.279,-.29);
  \draw[-,thin] (-.289,-.8) to (-.289,-.29);
  \draw[-,thin] (.289,-.8) to (.289,-.29);
	\draw[-,thick] (-0.28,-.3) to [out=90,in=-130] (0,.2) to [out=50,in=-90] (0.28,.6);	\draw[-,line width=4pt,white] (0.28,-.3) to [out=90,in=-50] (0,.2) to [out=130,in=-90] (-0.28,.6);
 \draw[-,thick] (0.28,-.3) to [out=90,in=-50] (0,.2) to [out=130,in=-90] (-0.28,.6);
     \node at (0.3,-.95) {$\scriptstyle b$};
        \node at (-0.3,-.95) {$\scriptstyle a$};
       \node at (-.65,-.45) {$\scriptstyle a-s$};
       \node at (0.3,.77) {$\scriptstyle b$};
        \node at (-0.3,.75) {$\scriptstyle a$};
\end{tikzpicture}.
\end{align*}
Now we use (b) from the corollary at the start of the appendix and the claim to simplify to
$$
q^{ab}\sum_{s=0}^{a}
q^{-s(b-a+s)}
\begin{tikzpicture}[anchorbase,scale=.8]
	\draw[-,thick] (0.02,.2) to (0.02,.3) to (.58,.9) to (.58,1);
 \draw[-,line width=4pt,white] (.58,.2) to (.58,.3) to (.02,.9) to (.02,1); \draw[-,thick] (.58,.2) to (.58,.3) to (.02,.9) to (.02,1);
	\draw[-,thin] (0,0.2) to (0,1);
	\draw[-,thin] (.6,0.2) to (.6,1);
        \node at (0,.05) {$\scriptstyle a$};
        \node at (.58,.05) {$\scriptstyle b$};
        \node at (-0.16,.6) {$\scriptstyle s$};
       \node at (0.6,1.17) {$\scriptstyle b$};
        \node at (0,1.15) {$\scriptstyle a$};
\end{tikzpicture}
-\sum_{s=0}^{a-1}
q^{(a-s)^2+(a-s)(b-a+2s)}
\begin{tikzpicture}[anchorbase,scale=.8]
	\draw[-,thick] (0.02,0.2) to (0.02,.3) to (.58,.9) to (.58,1);
 \draw[-,line width=4pt,white] (.58,.2) to (.58,.3) to (.02,.9) to (.02,1);
  \draw[-,thick] (.58,.2) to (.58,.3) to (.02,.9) to (.02,1);
	\draw[-,thin] (0,0.2) to (0,1);
	\draw[-,thin] (.6,0.2) to (.6,1);
        \node at (0,.05) {$\scriptstyle a$};
        \node at (.58,.05) {$\scriptstyle b$};
        \node at (-0.16,.6) {$\scriptstyle s$};
       \node at (0.6,1.17) {$\scriptstyle b$};
        \node at (0,1.15) {$\scriptstyle a$};
\end{tikzpicture}
=
\begin{tikzpicture}[anchorbase,scale=.8]
	\draw[-,thick] (0,0.2) to (0,1);
	\draw[-,thick] (.6,0.2) to (.6,1);
        \node at (0,.05) {$\scriptstyle a$};
        \node at (.58,.05) {$\scriptstyle b$};
       \node at (0.6,1.17) {$\scriptstyle b$};
        \node at (0,1.15) {$\scriptstyle a$};
\end{tikzpicture}.
$$
We noted also here that
$(a-s)^2+(a-s)(b-a+2s)=ab-s(b-a+s)$.

\color{black}

\vspace{2mm}
\noindent
{\bf \cref{braid}.}
This follows from the other relations, since they are already enough to show that $\qSchur'$ is a braided monoidal category.
To see it directly, 
replace the crossing of the
strings of thickness $a,b$ on both sides 
with \cref{jonsquare}. Then use \cref{sliders,swallows-symmetric} 
to pull the string of thickness 
$c$ past this expansion of the crossing.